\documentclass[12pt]{amsart}

\usepackage{amsmath,amsthm,amssymb}
\usepackage{tikz}
\usetikzlibrary{arrows.meta}

\usepackage{rotating}

\usepackage{pdflscape}
\usepackage{a4wide}
\usepackage{enumerate}
\usepackage[all]{xy}
\usepackage{multicol}
\usepackage{multirow}

\usepackage[colorlinks, linkcolor = blue, citecolor = green ]{hyperref}
\newtheorem{theorem}{Theorem}[section]
\newtheorem{lemma}[theorem]{Lemma}
\newtheorem{proposition}[theorem]{Proposition}
\newtheorem{corollary}[theorem]{Corollary}
\newtheorem*{MainTheorem}{Main Theorem}
\newtheorem*{claim*}{Claim}
\newtheorem*{fact*}{Fact}

\theoremstyle{remark}
\newtheorem{remark}[theorem]{Remark}

\newcommand{\g}[1]{\ensuremath{\mathfrak{#1}}}

\DeclareMathOperator{\spann}{span}

\begin{document}
\title[Strings in abstract root systems]{Strings in abstract root systems}

\author[V.~Sanmart\'in-L\'opez]{V\'ictor Sanmart\'in-L\'opez}
\address{CITMAga, 15782 Santiago de Compostela, Spain.\newline\indent Department of Mathematics, Universidade de Santiago de Compostela, Spain}
\email{victor.sanmartin@usc.es}

\begin{abstract}
Let $\Phi$ be a subset of the simple roots of a (possibly non-reduced) abstract root system $\Sigma$, and let $\lambda \in \Sigma$. We define the $\Phi$-string of $\lambda$ as the set of elements in $\Sigma \cup \{0\}$ of the form $\lambda + \sum_{\alpha \in \Phi} n_\alpha \alpha$, where $n_\alpha$ is an integer for each $\alpha \in \Phi$. This notion can be regarded as some sort of generalization of the classical notion of $\alpha$-string, where $\alpha \in \Sigma$.

In this paper, we calculate the $\Phi$-string of $\lambda$ explicitly, for any abstract root system $\Sigma$, any subset $\Phi$ of simple roots, and any root $\lambda \in \Sigma$.
\end{abstract}

\thanks{The author has been supported by Grant PID2022-138988NB-I00 funded by MICIU/AEI/10.13039/501100011033 and by ERDF, EU, and by ED431C 2023/31 (Xunta de Galicia).}
\maketitle

\vspace{-0.8cm}
\section{Introduction}\label{section:introduction}

The classification of semisimple Lie algebras, both complex and real, relies on the investigation of abstract root systems, and leads to Dynkin and Vogan or Satake diagrams. One of the key notions in order to address the inspection of abstract root systems is that of \emph{string}~\cite[p.~152]{K}. In this paper we focus on a generalization of this~notion.

More precisely, let $\Sigma$ be a (possibly non-reduced) abstract root system, let $\Sigma^+$ denote the set of positive roots, and let $\Pi$ be the set of simple roots for this criterion of positivity. Now, any root $\lambda \in \Sigma^+$ can be expressed with respect to $\Pi$ as a linear combination of non-negative integers whose sum is called the~\emph{level} of $\lambda$. Let $\Phi$ be a subset of $\Pi$ and let $\lambda$ be a root in $\Sigma$. We define the $\Phi$-string of $\lambda$ as the set $I_\Phi^\lambda$ of elements in $\Sigma \cup \{0\}$ of the form $\lambda + \sum_{\alpha \in \Phi} n_\alpha \alpha$, where $n_\alpha$ is an integer for each $\alpha \in \Phi$. We say that $\Phi$ is \emph{connected} if there is a connected subgraph of the Dynkin diagram of $\Sigma$ whose nodes correspond precisely to the roots in $\Phi$. The aim of this paper is to prove the following

\begin{MainTheorem}\label{Main:Theorem}
Let $\Phi$ be a proper subset of the set of simple roots $\Pi$ of the root system $\Sigma$, and let $\lambda$ be a root in $\Sigma$. Then:
\begin{enumerate}[\rm (i)]
\item If $\lambda\in \spann\Phi$, then the $\Phi$-string of $\lambda$ is the root subsystem $\Sigma_\Phi$ of $\Sigma$ spanned by $\Phi$. \label{Main:Theorem:1}
\end{enumerate}
If $\lambda \in \Sigma^+$ is not spanned by $\Phi$ and it is of minimum level in its $\Phi$-string, then:
\begin{enumerate}[\rm (i)]
\setcounter{enumi}{1}
\item The set $\Sigma_{\lambda, \Phi} = \spann_{\mathbb{Z}} (\{\lambda\} \cup \Phi) \cap \Sigma$ is a root subsystem of $\Sigma$ for which $\Pi_{\lambda, \Phi}=\{\lambda\} \cup \Phi$ is a simple system, and $\lambda$ is the unique root of minimum level in its $\Phi$-string.\label{Main:Theorem:2}
\item The $\Phi$-string of $\lambda$ has no gaps, that is, given any root $\gamma$ in the $\Phi$-string of $\lambda$ different from $\lambda$, there exists some $\alpha \in \Phi$ such that $\gamma - \alpha$ is a root.\label{Main:Theorem:3}
\item If $\Phi$ is a connected subset of $\Pi$, the $\Phi$-string of $\lambda$ is explicitly calculated in Table~\ref{table:classical} when $\Sigma_{\lambda, \Phi}$ is of classical type, and in Table~\ref{table:exceptional} if $\Sigma_{\lambda, \Phi}$ is of exceptional type.\label{Main:Theorem:4}
\item If $\Phi = \Phi_1 \cup \dots \cup \Phi_n$, where the $\Phi_i$'s are mutually orthogonal connected subsets of $\Pi$, then 
$I_\Phi^\lambda = \{ \lambda + \sum_{k = 1}^n  \gamma_k \, : \, \gamma_k \in \spann \Phi_k \, \, \, \text{and} \, \, \, \lambda + \gamma_k \in I_{\Phi_k}^\lambda \subset \Sigma^+,  \, \text{for all} \, \, \, k = 1, \dots, n    \}$.\label{Main:Theorem:5}
\end{enumerate}
\end{MainTheorem}

\begin{landscape}
\begin{table}
\renewcommand*{\arraystretch}{2}
\vspace{1.5cm}
\begin{tabular}{|c|c|c|c|}
\hline
\large \textbf{$\Sigma_\Phi$} & \textbf{$\Sigma_{\lambda, \Phi}$} & \textbf{Dynkin diagram of $\Sigma_{\lambda, \Phi}$} &  \textbf{$\Phi$-string of $\lambda$} \\
\hline
\hline
\hline
\multirow{4}{*}{$A_n$} & \multirow{2}{*}{$A_{n+1}$} & \multirow{2}{*}{
\begin{tikzpicture}
\phantom{\draw (0, 0.2) circle (0.1);}
\draw (0, 0) circle (0.1);
\draw (1, 0) circle (0.1);
\draw (3, 0) circle (0.1);
\draw (0.1, 0.) -- (0.9, 0.);
\draw (1., -0.3) node {$\alpha _1$};
\draw (1.1, 0.) -- (1.5, 0.);
\draw (1.7, 0.) -- (1.82, 0.);
\draw (1.94, 0.) -- (2.06, 0.);
\draw (2.18, 0.) -- (2.3, 0.);
\draw (2.5, 0) -- (2.9, 0);
\draw (3., -0.3) node {$\alpha _n$};
\draw (0, -0.3) node {$\lambda$};
\end{tikzpicture}
}
& \multirow{4}{*}{$\{ \lambda \} \cup \left\{ \lambda + \displaystyle\sum_{i=1}^l \alpha_i \, : \, 1 \leq l \leq n  \right\}$} \\
&&&\\
\cline{2-3} &$B_{n+1}$ & 
\multirow{2}{*}{
\begin{tikzpicture}
\phantom{\draw (0, 0.2) circle (0.1);}
\draw (0, 0) circle (0.1);
\draw (1, 0) circle (0.1);
\draw (3, 0) circle (0.1);
\draw (0.1, -0.04) -- (0.9, -0.04);
\draw (0.1, 0.04) -- (0.9, 0.04);
\draw (1., -0.3) node {$\alpha _1$};
\draw (1.1, 0.) -- (1.5, 0.);
\draw (1.7, 0.) -- (1.82, 0.);
\draw (1.94, 0.) -- (2.06, 0.);
\draw (2.18, 0.) -- (2.3, 0.);
\draw (2.5, 0) -- (2.9, 0);
\draw (3., -0.3) node {$\alpha _n$};
\draw (0, -0.3) node {$\lambda$};
\end{tikzpicture}}  &\\
\cline{2-2}
&$BC_{n+1}$& &\\
\hline \hline
\multirow{1}{*}{$A_n$} & \multirow{1}{*}{$C_{n+1}$}  & 
\multirow{1}{*}{
\begin{tikzpicture}
\phantom{\draw (0, 0.2) circle (0.1);}
\draw (0, 0) circle (0.1);
\draw (1, 0) circle (0.1);
\draw (3, 0) circle (0.1);
\draw (0.1, -0.04) -- (0.9, -0.04);
\draw (0.1, 0.04) -- (0.9, 0.04);
\draw (1., -0.3) node {$\alpha _1$};
\draw (1.1, 0.) -- (1.5, 0.);
\draw (1.7, 0.) -- (1.82, 0.);
\draw (1.94, 0.) -- (2.06, 0.);
\draw (2.18, 0.) -- (2.3, 0.);
\draw (2.5, 0) -- (2.9, 0);
\draw (3., -0.3) node {$\alpha _n$};
\draw (0, -0.3) node {$\lambda$};
\end{tikzpicture}
}
& $\{ \lambda \} \cup \left\{ \lambda +\displaystyle \sum_{i=1}^l \alpha_i, \, \lambda +\sum_{i=1}^l \alpha_i + \sum_{j=1}^m \alpha_j\, : \, \substack{  1 \leq l \leq n \\ \, 1 \leq m \leq l} \right\}$ \\ \hline \hline
\multirow{1}{*}{$B_n$} & \multirow{1}{*}{$B_{n+1}$} & 
\multirow{2}{*}{
\begin{tikzpicture}
\phantom{\draw (0, 0.2) circle (0.1);}
\draw (0, 0) circle (0.1);
\draw (1, 0) circle (0.1);
\draw (3, 0) circle (0.1);
\draw (4, 0) circle (0.1);
\draw (3.1, -0.04) -- (3.9, -0.04);
\draw (3.1, 0.04) -- (3.9, 0.04);
\draw (4., -0.3) node {$\alpha _n$};
\draw (0.1, 0.) -- (0.9, 0.);
\draw (1., -0.3) node {$\alpha _1$};
\draw (1.1, 0.) -- (1.5, 0.);
\draw (1.7, 0.) -- (1.82, 0.);
\draw (1.94, 0.) -- (2.06, 0.);
\draw (2.18, 0.) -- (2.3, 0.);
\draw (2.5, 0) -- (2.9, 0);
\draw (3., -0.3) node {$\alpha_{n-1}$};
\draw (0, -0.3) node {$\lambda$};
\end{tikzpicture}
}
& \multirow{2}{*}{$\left\{ \lambda \right\} \cup \left\{\lambda + \displaystyle \sum_{i=1}^l \alpha_i, \lambda +\displaystyle \sum_{i=1}^n \alpha_i  + \sum_{j=0}^k \alpha_{n-j}\, : \, \substack{1 \leq l \leq n \\ 0 \leq k \leq n-1}  \right\}$} \\
 \cline{1-2}
\multirow{1}{*}{$BC_n$} & \multirow{1}{*}{$BC_{n+1}$} & &  \\
\hline \hline
 \multirow{1}{*}{$C_n$} & \multirow{1}{*}{$C_{n+1}$} & \multirow{1}{*}{
\begin{tikzpicture}
\phantom{\draw (0, 0.2) circle (0.1);}
\draw (0, 0) circle (0.1);
\draw (1, 0) circle (0.1);
\draw (3, 0) circle (0.1);
\draw (4, 0) circle (0.1);
\draw (3.1, -0.04) -- (3.9, -0.04);
\draw (3.1, 0.04) -- (3.9, 0.04);
\draw (4., -0.3) node {$\alpha _n$};
\draw (0.1, 0.) -- (0.9, 0.);
\draw (1., -0.3) node {$\alpha _1$};
\draw (1.1, 0.) -- (1.5, 0.);
\draw (1.7, 0.) -- (1.82, 0.);
\draw (1.94, 0.) -- (2.06, 0.);
\draw (2.18, 0.) -- (2.3, 0.);
\draw (2.5, 0) -- (2.9, 0);
\draw (3., -0.3) node {$\alpha_{n-1}$};
\draw (0, -0.3) node {$\lambda$};
\end{tikzpicture}}
& $\left\{ \lambda \right\} \cup \left\{\lambda + \displaystyle \sum_{i=1}^l \alpha_i, \,  \lambda + \sum_{i=1}^{n} \alpha_i +\sum_{j=1}^{k} \alpha_{n-j} \, :  \, \substack{1 \leq l \leq n \\ 1 \leq k \leq  n-1}  \right\}$\\
\hline \hline 
 & & \multirow{3}{*}{
\begin{tikzpicture}
\draw (0, 0) circle (0.1);
\draw (1, 0) circle (0.1);
\draw (3, 0) circle (0.1);
\draw (4, 0) circle (0.1);
\draw (3, 1) circle (0.1);
\draw (3., 0.1) -- (3., 0.9);
\draw (3., 1.3) node {$\alpha _{n-1}$};
\draw (3.1, 0.) -- (3.9, 0.);
\draw (4., -0.3) node {$\alpha _n$};
\draw (0.1, 0.) -- (0.9, 0.);
\draw (1., -0.3) node {$\alpha _1$};
\draw (1.1, 0.) -- (1.5, 0.);
\draw (1.7, 0.) -- (1.82, 0.);
\draw (1.94, 0.) -- (2.06, 0.);
\draw (2.18, 0.) -- (2.3, 0.);
\draw (2.5, 0) -- (2.9, 0);
\draw (3., -0.3) node {$\alpha _{n-2}$};
\draw (0, -0.3) node {$\lambda$};
\end{tikzpicture}
}
& \multirow{3}{*}{$\left\{ \lambda, \, \lambda +\displaystyle \sum_{\substack{i=1 \\ i \neq n-1}}^n \alpha_i  \right\} \cup \left\{\,\lambda + \displaystyle\sum_{i=1}^l \alpha_i,  \,  \lambda + \displaystyle\sum_{i=1}^{n} \alpha_i +\displaystyle\sum_{j=2}^{k} \alpha_{n-j} \, :  \, \substack{1 \leq l \leq n \\  2 \leq k \leq n-1}  \right\}$}  \\ 
$D_n$ & $D_{n+1}$ & & \\ 
   & & & \\ \hline
\end{tabular}
\bigskip
\caption{$\Sigma_{\lambda, \Phi}$ classical root system.}
\label{table:classical}
\end{table}
\end{landscape}

\begin{table}[h]
\renewcommand*{\arraystretch}{2}
\begin{tabular}{|c|c|c|c|c|}
\hline
\large \textbf{$\Sigma_\Phi$} & \textbf{$\Sigma_{\lambda, \Phi}$} &\textbf{Conditions} & \textbf{Dynkin diagram of $\Sigma_{\lambda, \Phi}$} &  \textbf{$\Phi$-string of $\lambda$} \\
\hline
\hline
\hline
& & & \multirow{3}{*}{
\begin{tikzpicture}
\draw (1, 0) circle (0.1);
\draw (3, 0) circle (0.1);
\draw (4, 0) circle (0.1);
\draw (5, 0) circle (0.1);
\draw (4, 1) circle (0.1);
\draw (6, 0) circle (0.1);
\draw (4., 0.1) -- (4., 0.9);
\draw (4., 1.3) node {$\lambda$};
\draw (3.1, 0.) -- (3.9, 0.);
\draw (4., -0.3) node {$\alpha _3$};
\draw (4.1, 0.) -- (4.9, 0.);
\draw (5., -0.3) node {$\alpha _2$};
\draw (5.1, 0.) -- (5.9, 0.);
\draw (6., -0.3) node {$\alpha _1$};
\draw (1.1, 0.) -- (1.5, 0.);
\draw (1.7, 0.) -- (1.82, 0.);
\draw (1.94, 0.) -- (2.06, 0.);
\draw (2.18, 0.) -- (2.3, 0.);
\draw (2.5, 0) -- (2.9, 0);
\draw (3., -0.3) node {$\alpha _4$};
\draw (1, -0.3) node {$\alpha _n$};
\end{tikzpicture}
}
& \multirow{3}{*}{Proposition~\ref{proposition:a:e:string:type}} \\ 
$A_n$ & $E_{n+1}$ & $n \in \{5,6,7\}$ & & \\ 
& & & & \\ \hline \hline
& & & \multirow{3}{*}{
\begin{tikzpicture}
\draw (1, 0) circle (0.1);
\draw (3, 0) circle (0.1);
\draw (4, 0) circle (0.1);
\draw (5, 0) circle (0.1);
\draw (4, 1) circle (0.1);
\draw (6, 0) circle (0.1);
\draw (4., 0.1) -- (4., 0.9);
\draw (4., 1.3) node {$\alpha _2$};
\draw (3.1, 0.) -- (3.9, 0.);
\draw (4., -0.3) node {$\alpha _3$};
\draw (4.1, 0.) -- (4.9, 0.);
\draw (5., -0.3) node {$\alpha _1$};
\draw (5.1, 0.) -- (5.9, 0.);
\draw (6., -0.3) node {$\lambda$};
\draw (1.1, 0.) -- (1.5, 0.);
\draw (1.7, 0.) -- (1.82, 0.);
\draw (1.94, 0.) -- (2.06, 0.);
\draw (2.18, 0.) -- (2.3, 0.);
\draw (2.5, 0) -- (2.9, 0);
\draw (3., -0.3) node {$\alpha _4$};
\draw (1, -0.3) node {$\alpha _n$};
\end{tikzpicture}
}
& \multirow{3}{*}{Proposition~\ref{proposition:d:e:string:type}} \\ 
$D_n$ & $E_{n+1}$ & $n \in \{5,6,7\}$ & & \\ 
& & & & \\ \hline \hline
& & & \multirow{3}{*}{
\begin{tikzpicture}
\draw (1, 0) circle (0.1);
\draw (2, 0) circle (0.1);
\draw (3, 0) circle (0.1);
\draw (4, 0) circle (0.1);
\draw (5, 0) circle (0.1);
\draw (6, 0) circle (0.1);
\draw (4, 1) circle (0.1);
\draw (1.1, 0.) -- (1.9, 0.);
\draw (2., -0.3) node {$\alpha _6$};
\draw (2.1, 0.) -- (2.9, 0.);
\draw (3., -0.3) node {$\alpha _5$};
\draw (3.1, 0.) -- (3.9, 0.);
\draw (4., -0.3) node {$\alpha _4$};
\draw (4.1, 0.) -- (4.9, 0.);
\draw (5., -0.3) node {$\alpha _3$};
\draw (5.1, 0.) -- (5.9, 0.);
\draw (6., -0.3) node {$\alpha _1$};
\draw (4., 0.1) -- (4., 0.9);
\draw (4., 1.3) node {$\alpha _2$};
\draw (1, -0.3) node {$\lambda$};
\end{tikzpicture}
}
& \multirow{3}{*}{Proposition~\ref{proposition:e6:e7:string:type}} \\ 
$E_6$ & $E_7$ &  & & \\ 
& & & & \\ \hline \hline
& & & \multirow{3}{*}{
\begin{tikzpicture}
\draw (0, 0) circle (0.1);
\draw (1, 0) circle (0.1);
\draw (2, 0) circle (0.1);
\draw (3, 0) circle (0.1);
\draw (4, 0) circle (0.1);
\draw (5, 0) circle (0.1);
\draw (6, 0) circle (0.1);
\draw (4, 1) circle (0.1);
\draw (0.1, 0.) -- (0.9, 0.);
\draw (1., -0.3) node {$\alpha _7$};
\draw (1.1, 0.) -- (1.9, 0.);
\draw (2., -0.3) node {$\alpha _6$};
\draw (2.1, 0.) -- (2.9, 0.);
\draw (3., -0.3) node {$\alpha _5$};
\draw (3.1, 0.) -- (3.9, 0.);
\draw (4., -0.3) node {$\alpha _4$};
\draw (4.1, 0.) -- (4.9, 0.);
\draw (5., -0.3) node {$\alpha _3$};
\draw (5.1, 0.) -- (5.9, 0.);
\draw (6., -0.3) node {$\alpha _1$};
\draw (4., 0.1) -- (4., 0.9);
\draw (4., 1.3) node {$\alpha _2$};
\draw (0, -0.3) node {$\lambda$};
\end{tikzpicture}
}
& \multirow{3}{*}{Proposition~\ref{proposition:e7:e8:string:type}} \\ 
$E_7$ & $E_8$ &  & & \\ 
& & & & \\ \hline \hline
 \multirow{1}{*}{$B_3$} &  \multirow{1}{*}{$F_4$} & & 
 \multirow{1}{*}{\begin{tikzpicture}
 \phantom{\draw (0, 0.2) circle (0.1);}
 \draw (0, 0) circle (0.1);
 \draw (1, 0) circle (0.1);
 \draw (2, 0) circle (0.1);
 \draw (3, 0) circle (0.1);
 \draw (0.1, 0.) -- (0.9, 0.);
 \draw (1., -0.3) node {$\alpha _1$};
 \draw (1.1, -0.04) -- (1.9, -0.04);
 \draw (1.1, 0.04) -- (1.9, 0.04);
 \draw (2., -0.3) node {$\alpha _2$};
 \draw (2.1, 0.) -- (2.9, 0.);
 \draw (3., -0.3) node {$\alpha_3$};
 \draw (0, -0.3) node {$\lambda$};
 \end{tikzpicture}
}
 & Proposition~\ref{proposition:b:f:string:type} \\ \hline \hline
\multirow{1}{*}{$C_3$} &  \multirow{1}{*}{$F_4$} & & 
\multirow{1}{*}{\begin{tikzpicture}
\phantom{\draw (0, 0.2) circle (0.1);}
\draw (0, 0) circle (0.1);
\draw (1, 0) circle (0.1);
\draw (2, 0) circle (0.1);
\draw (3, 0) circle (0.1);
\draw (0.1, 0.) -- (0.9, 0.);
\draw (1., -0.3) node {$\alpha _2$};
\draw (1.1, -0.04) -- (1.9, -0.04);
\draw (1.1, 0.04) -- (1.9, 0.04);
\draw (2., -0.3) node {$\alpha _1$};
\draw (2.1, 0.) -- (2.9, 0.);
\draw (3., -0.3) node {$\lambda$};
\draw (0, -0.3) node {$\alpha _3$};
\end{tikzpicture}
}
& Proposition~\ref{proposition:c:f:string:type} \\ \hline 

\end{tabular}
\bigskip
\caption{$\Sigma_{\lambda, \Phi}$ exceptional root system.}
\label{table:exceptional}
\end{table}  
Put $\Sigma^- := -\Sigma^+$ for the set of negative roots. Note that the roots in the $\Phi$-string of $\lambda$ are exactly the opposite to these in the $\Phi$-string of $-\lambda$, that is, $I_\Phi^\lambda = - I_\Phi^{-\lambda}$. Therefore, despite the fact that before item~(\ref{Main:Theorem:2}) of the Main Theorem we select a positive root $\lambda \in \Sigma^+$, we are still obtaining the $\Phi$-string of $\lambda$ explicitly for any subset $\Phi$ of $\Pi$ and any root $\lambda \in \Sigma$. Furthermore, if $\Sigma$ is a $G_2$ root system, then any non-trivial proper subset $\Phi$ of $\Pi$ would contain exactly one root, and this case is treated in~\cite[Proposition~2.48]{K}. 

The idea in order to prove the Main Theorem is the following. Let $\Phi$ be a subset of the set $\Pi$ of simple roots of a root system $\Sigma$. Note that $\Phi$ generates a root subsystem $\Sigma_\Phi$ of $\Sigma$ for which $\Phi$ constitutes a simple system. Let $\lambda \in \Sigma^+$ be a root, not spanned by $\Phi$, of minimum level in its $\Phi$-string. First, we see that $\spann_{\mathbb{Z}} (\{\lambda\} \cup \Phi) \cap \Sigma$ constitutes a root subsystem $\Sigma_{\lambda, \Phi}$ of $\Sigma$, for which $\Pi_{\lambda, \Phi} = \{ \lambda \} \cup \Pi$ is a simple system, using standard techniques of root systems (see Proposition~\ref{proposition:root:system:lambda:phi}). But then, calculating the roots in the $\Phi$-string of $\lambda$ turns out to be equivalent to calculating the positive roots in $\Sigma_{\lambda, \Phi}$ whose coefficient corresponding to $\lambda$ with respect to $\Pi_{\lambda, \Phi}$ is equal to one. Therefore, assuming that $\Phi$ is connected, we start a case-by-case analysis on the possibilities for the pair $(\Sigma_\Phi, \Sigma_{\lambda, \Phi})$. Note that the Dynkin diagram of $\Sigma_\Phi$ has to be subdiagram of the Dynkin diagram correspoding to $\Sigma_{\lambda, \Phi}$. For instance, if $\Sigma_\Phi$ were an $A_n$ root system, then $\Sigma_{\lambda, \Phi}$ could be of type $A_{n+1}$, $B_{n+1}$, $BC_{n+1}$, $C_{n+1}$, $D_{n+1}$ or $E_{n+1}$ (if $n \in \{ 5, 6 , 7\}$). However, if $\Sigma_\Phi$ is of $B_n$ type, then $\Sigma_{\lambda, \Phi}$ could be either of type $B_{n+1}$, or of type $F_4$ with $n=3$. We later address the case when $\Phi$ is a non-connected subset. Finally, we associate an oriented graph to the $\Phi$-string of $\lambda$, for any $\lambda$ not spanned by $\Phi$, which permits to understand better its configuration.

Although abstract root systems arise in connection with Lie algebra theory, our source of interest stems from differential geometry or, more precisely, from the investigation of extrinsically homogeneous submanifolds in symmetric spaces of non-compact type. Roughly speaking, these are submanifolds whose geometric properties are the same at any point. More precisely, let $M$ be a symmetric space of non-compact type and let $\g{g}$ be the Lie algebra of the connected component of the identity of its isometry group $G = I^0(M)$. Then, $\g{g}$ admits the so-called \emph{root space decomposition}, namely $\g{g} = \g{g}_0 \oplus  \bigl(\bigoplus_{\lambda \in \Sigma} \g{g}_\lambda \bigr)$, and $\Sigma$ happens to be an abstract root system. Furthermore, $M$ turns out to be isometric to a solvable Lie group $AN$, which is the semidirect product of an abelian Lie group $A$ with Lie algebra $\g{a}$ and a nilpotent Lie group $N$, both Lie subgroups of $G$. The Lie algebra of $N$ is $\g{n} = \bigoplus_{\lambda \in \Sigma^+} \g{g}_\lambda$, that is, the sum of root spaces associated with the positive roots.

All in all, important family of homogeneous submanifolds of $M \cong AN$ are achieved as orbits in $M$ of Lie subgroups of $AN$. Indeed, let $S$ be a Lie subgroup of $AN$. We do not enter in detail, but we can study the tangent space of certain orbits of $S$ in $M$ by means of the Lie algebra $\g{s}$ of $S$, which is a subalgebra of of $\g{a} \oplus \g{n}$. The Levi-Civita connection of $M \cong AN$, which is the main tool to investigate many geometric properties, can be expressed in terms of the Lie bracket, and we have $[\g{g}_\alpha, \g{g}_\beta] \subset \g{g}_{\alpha+\beta}$ for any $\alpha$, $\beta \in \Sigma$. Thus, at first glance, it seems to relate root spaces associated with positive roots between them in an intricate manner which is difficult to control. But $\Phi$-strings induce decompositions of the tangent space into subspaces which are invariant for the Levi-Civita connection. Since $\Phi$-strings can just adopt a couple of configurations, as follows from the Main Theorem (see Table~\ref{table:classical} and Table~\ref{table:exceptional}), they enable to calculate explicitly the extrinsic geometry of important families of homogeneous submanifolds in symmetric spaces of non-compact type.

Indeed, if $\Phi$ contains exactly two roots, $\Phi$-strings were studied and determined in~\cite{BS18} in order to produce, describe and classify certain homogeneous submanifolds of symmetric spaces of non-compact type with constant principal curvatures. They were also used in~\cite{DST20} in order to calculate the extrinsic geometry of certain homogeneous hypersurfaces of symmetric spaces of non-compact type that are Ricci solitons when they are considered with the induced metric. Furthermore, we are using them \emph{to conclude the classification of cohomogeneity one actions in symmetric spaces of non-compact type}, where $\Phi$-strings are also utilized to understand better the theory of parabolic Lie algebras, which has played a crucial role in the development of the problem~\cite{BT13}.

\medskip

This paper is organized as follows. In Section~\ref{section:preliminaries} we recall the notion of abstract root system together with some well known results concerning them. Take $\Phi$ a subset of $\Pi$ and $\lambda \in \Sigma^+$ not spanned by $\Phi$, and of minimum level in its $\Phi$-string. Section~\ref{section:structure} is mainly devoted to prove that $\{\lambda\} \cup \Phi$ is a simple system for the root subsystem $\Sigma_{\lambda, \Phi}$ of $\Sigma$ that it spans. Thus, the $\Phi$-string of $\lambda$ can be thought of inside the root system $\Sigma_{\lambda, \Phi}$. This perspective allows us to generalize the fact that the root of minimum level is unique, and that $\Phi$-strings do not have gaps (see~\cite[Proposition~2.48~(g)]{K}). Moreover, it allows us to calculate any $\Phi$-string of $\lambda$ by means of a case-by-case analysis of the pair $(\Sigma_\Phi, \Sigma_{\lambda, \Phi})$, when $\Phi$ is connected. First, we assume that $\Phi$ is connected. On the one hand, in Section~\ref{section:classical:root:system} we address all the cases when $\Sigma_{\lambda, \Phi}$ is a classical root system. On the other hand, in Section~\ref{section:exceptional:root:system}, we focus on the cases when $\Sigma_{\lambda, \Phi}$ is an exceptional root system. In Section~\ref{section:non:connected:phi}, we address the case when $\Phi$ is a non-connected subset of $\Pi$. Finally, in Section~\ref{section:diagrams} we include the diagrams of the strings.

\medskip
I would like to thank J. C. Díaz-Ramos, M. Domínguez-Vázquez and H. Tamaru for their interesting comments and remarks on several aspects of this work.
\section{Preliminaries}\label{section:preliminaries}
In this section, we introduce the necessary preliminaries for the development of the article. We start by presenting the notion of abstract root system, and also several well-known results related to them.

Following~\cite[p. 149]{K}, an \emph{abstract root system} $\Sigma$, in a finite dimensional real vector space $V$ equipped with an inner product $\langle \cdot, \cdot \rangle$ and norm $| \cdot |$, is a finite subset of non-zero elements of $V$ such that:
\begin{enumerate}[{\rm (a)}]
\item $V=\spann\Sigma$,
\item the number $A_{\alpha,\beta} =2\langle\alpha,\beta\rangle/ |\alpha|^2$ is an integer for any $\alpha,\beta\in \Sigma$, 
\item $\beta-A_{\alpha, \beta}\,\alpha\in \Sigma$ for any $\alpha,\beta\in \Sigma$.
\end{enumerate}
An abstract root system is said to be \emph{reduced} if $\lambda\in\Sigma$ implies that $2\lambda\notin \Sigma$. Otherwise, it is said to be \emph{non-reduced}. The notation $|\cdot|$ will be also used to denote the cardinal of certain sets.

Now, we can define a positivity criterion on $\Sigma$ by declaring those roots that lie at one of the two half-spaces determined by a hyperplane in $V$ not containing any root to be positive. If $\Sigma^+$ denotes the set of positive roots, then $\Sigma=\Sigma^+\cup \Sigma^-$, where $\Sigma^- := - \Sigma^+$ denotes the set of negative roots. 

We define here the concept of string~\cite[p.~152]{K}, whose generalization will be the central concept of this work. Let $\alpha \in \Sigma$ and $\lambda \in \Sigma \cup \{0\}$. The $\alpha$-\emph{string} containing $\lambda$ is defined as the set of all elements in $\Sigma \cup \{0\}$ of the form $\lambda + n \alpha$, with $n \in \mathbb{Z}$. We state now a result concerning the algebraic structure of the root system $\Sigma$. In particular, it provides really useful information about the Cartan integers, that is, the integers of the form
\[
A_{\alpha, \lambda} = \frac{2 \langle \alpha, \lambda \rangle}{|\alpha|^2},
\]
where $\alpha$, $\lambda \in \Sigma$. The calculation of Cartan integers allows to control how roots are constructed, and, in particular, it allows to determine strings explicitly. 

\begin{proposition}\label{proposition:strings}\emph{\cite[Proposition 2.48]{K}}
Let $\Sigma$ be an abstract root system in the inner product space $V$. Then:
\begin{enumerate}[{\rm (i)}]
\item If $\alpha \in\Sigma$, then $-\alpha \in \Sigma$. \label{proposition:strings:1}
\item If $\alpha \in \Sigma$ and $\lambda \in \Sigma \cup \{0\}$, then
\[
A_{\alpha,\lambda} = \frac{2 \langle \lambda, \alpha \rangle}{|\alpha|^2} \in \{0, \pm 1, \pm 2, \pm 3, \pm 4\},
\]
and $\pm 4$ occurs only when $\Sigma$ is non-reduced  and  $\lambda = \pm 2 \alpha$.\label{proposition:strings:2}
\item If $\alpha,\lambda\in\Sigma$ are non-proportional and $|\lambda| \leq |\alpha|$, then $A_{\alpha,\lambda} \in \{0, \pm 1\}$. \label{proposition:strings:3}
\item If $\alpha,\lambda \in \Sigma$ with $\langle \alpha, \lambda \rangle > 0$, then $\alpha - \lambda \in \Sigma \cup \{0\}$. If $\alpha,\lambda\in\Sigma$ with $\langle \alpha, \lambda \rangle < 0$, then $\alpha + \lambda \in \Sigma \cup \{0\}$. \label{proposition:strings:4}
\item If $\alpha\in\Sigma$ and $\lambda\in\Sigma \cup \{0\}$, then the $\alpha$-string containing $\lambda$ has the form $\lambda + n \alpha$ for $-p \leq n \leq q$ with $p, q \geq 0$. There are no gaps. Furthermore $p-q = A_{\alpha,\lambda}$. The $\alpha$-string containing $\lambda$ contains at most four roots. \label{proposition:strings:5}
\end{enumerate}
\end{proposition}
As it is usual in the theory of root systems, one can consider a subset $\Pi\subset \Sigma^+$ of simple roots. A positive root is said to be \emph{simple} if it cannot be written as the sum of two positive roots. The set of simple roots $\Pi$ is a basis of $V$ made of positive roots. Furthermore, we have the following 

\begin{proposition}\label{proposition:simple:basis}\emph{\cite[Proposition 2.49]{K}}
Let $n = \dim V$. Then, there are $n$ simple roots and they are linearly independent. If $\lambda$ is a root and it is written as $\lambda = \sum_{\alpha \in \Pi }n_{\alpha} \alpha$, then all the coefficients $n_\alpha$, with $\alpha \in \Pi$, are either non-negative integers (if $\lambda$ is a positive root), or non-positive integers (if $\lambda$ is a negative root).
\end{proposition}
For each root $\lambda = \sum_{\alpha \in \Pi }n_{\alpha} \alpha \in \Sigma$, the sum 
\begin{equation}\label{definition:level}
l(\lambda) = \sum_{\alpha \in \Pi} n_{\alpha}
\end{equation}
is called the \emph{level} of the root $\lambda$. Note that positive roots have positive level and negative roots have negative level. 
Any positive root of level at least two can be constructed as the sum of a positive root and a simple root, as we make more precise in the following

\begin{lemma}\label{exercise}\emph{\cite[p. 204, Exercise 7]{K}}
Let $\Pi$ be a set of simple roots of a root system $\Sigma$. Any $\lambda \in \Sigma^+$ can be written in the form 
\[
\lambda=\lambda_{i_1}+\lambda_{i_2} + \dots +\lambda_{i_k},
\]
where $\lambda_{i_j} \in \Pi$ and each partial summand from the left is in $\Sigma^+$.
\end{lemma}

\begin{proof}
We proceed by induction on the level $l(\lambda)$ of the root $\lambda \in \Sigma^{+}$. Recall from Proposition~\ref{proposition:simple:basis} that $\lambda$ can be written as $\lambda = \sum_{\alpha \in \Pi} n_{\alpha} \alpha$, for some integers $n_{\alpha} \geq 0$, for each $\alpha \in \Pi$. The claim is obvious if $l(\lambda) = 1$. Assume that it is true for level $k > 1$ and let $\lambda \in \Sigma^+$ be a positive root satisfying $l(\lambda) = k+1$. If $\langle \lambda, \alpha \rangle \leq 0$ for all $\alpha \in \Pi$, then we would have
\[
0 < \langle \lambda, \lambda \rangle = \langle \lambda, \sum_{\alpha \in \Pi} n_{\alpha} \alpha \rangle \leq 0,
\]
which is a contradiction. Thus, there exists $\alpha \in \Pi$ such that $\langle \lambda, \alpha \rangle > 0$. Now, using Proposition~\ref{proposition:strings}~(\ref{proposition:strings:4}), $l(\lambda) > 2$ and $l(\alpha) = 1$, we deduce that $\lambda - \alpha \in \Sigma$. Hence, $\lambda = (\lambda - \alpha) + \alpha$ and the result follows by applying the induction hypothesis to $\lambda - \alpha$, which has level $k$. 
\end{proof}

The set $\Pi$ of simple roots allows to construct the Dynkin diagram associated with the root system $\Sigma$, which is a graph whose nodes correspond to the simple roots. The nodes corresponding to the simple roots $\alpha$, $\beta \in \Pi$ are joined by $A_{\alpha, \beta} \cdot A_{\beta, \alpha}$ edges. Dynkin diagrams were utilized to classify complex semisimple Lie algebras. See~\cite[Chapter~2]{K} for more details on Dynkin diagrams.

 A subset $\Phi$ of the set of simple roots $\Pi$ is said to be \emph{connected} if it cannot be expressed as a non-trivial union $\Phi_1 \cup \Phi_2$ where any element in $\Phi_1$ is orthogonal to any element in $\Phi_2$, or, equivalently, if there is a connected subgraph of the Dynkin diagram of $\Pi$ whose nodes correspond precisely to the roots in $\Phi$. Moreover, a root system $\Sigma$ is said to be \emph{irreducible} if its Dynkin diagram is connected, and~\emph{reducible} otherwise (see~\cite[p.~160]{K}). We finish this section with some auxiliary results, which are probably known, but we did not find an appropriate reference in the literature.

\begin{lemma}\label{lemma:sum:orthogonal:subsets}
Let $\Phi_0$, $\Phi_1$ be orthogonal subsets of $\Pi$. Let $\lambda_0$, $\lambda_1 \in \Sigma$ be roots spanned by $\Phi_0$ and $\Phi_1$, respectively. Then $\pm \lambda_0 \pm \lambda_1$ cannot be roots. 
\end{lemma}

\begin{proof}
Note that $\Sigma^\prime = \spann_\mathbb{Z}(\Phi_0 \cup \Phi_1) \cap \Sigma$ is a root subsystem of $\Sigma$, for which $\Phi_0 \cup \Phi_1$ is a simple system (after inducing in $\Sigma^\prime$ the positivity criterion of $\Sigma$). Since $\Phi_0$ is orthogonal to $\Phi_1$, the Dynkin diagram of $\Sigma^\prime$ is not connected. This implies that $\Sigma^\prime$ is a reducible root system~\cite[p.~160]{K}, that is, it can be written as the disjoint union of the orthogonal sets $\Sigma_0 = \spann_\mathbb{Z} (\Phi_0) \cap \Sigma$ and $\Sigma_1=\spann_\mathbb{Z} (\Phi_1) \cap \Sigma$. Take $\lambda_k \in \Sigma_k$, for $k \in \{0,1\}$. If $\pm \lambda_0 \pm \lambda_1$ is in $\spann_\mathbb{Z}(\Phi_0 \cup \Phi_1) \cap \Sigma$, then it must belong either to $\Sigma_0$ or to $\Sigma_1$. However $\pm \lambda_0 \pm \lambda_1$ is neither orthogonal to all the elements in $\Sigma_0$ nor to all the elements in $\Sigma_1$.
\end{proof}

\begin{lemma}\label{lemma:sum:connected:diagram}
Let $\Psi$ be a connected subset of the set $\Pi$ of simple roots of the root system $\Sigma$. Then, $\sum_{\alpha \in \Psi} \alpha$ is a positive root of $\Sigma$.
\end{lemma}

\begin{proof}
We proceed by induction on the number of elements in $\Psi$. If $\Psi$ contains just one root, the result is trivial. Hence, we will assume that $\Psi$ contains $n$ roots and that our claim is true for connected subsets of $\Pi$ of $n-1$ roots. Since $\Psi$ is connected, there must exist $\alpha \in \Psi$ just connected to one root $\beta \in \Psi$ in the Dynkin diagram of $\Sigma$. Otherwise, there is a loop, which is not possible~\cite[Proposition~2.78~(b)]{K}. In other words, there must exist $\alpha$ and $\beta$ in $\Psi$ such that $A_{\alpha, \beta} < 0 $ and $A_{\alpha, \nu} = 0$ for any $\nu \in \Psi \backslash \{ \alpha, \beta \}$. Now, $\Psi \backslash \{ \alpha \}$ is a connected subset of $\Pi$ of $n-1$ elements. Thus, $\gamma = \sum_{\nu \in \Psi \backslash \{ \alpha \}} \nu$ is a root by induction hypothesis. Moreover, $A_{\alpha, \gamma} <0 $ and $\gamma + \alpha = \sum_{\nu \in \Psi } \nu$ is a root by virtue of Proposition~\ref{proposition:strings}~(\ref{proposition:strings:4}).
\end{proof}

From Lemma~\ref{lemma:sum:connected:diagram} and a straightforward calculation with Cartan integers, we get 

\begin{corollary}\label{corollary:an}
Let $\Sigma$ be a root system of type $A_n$, $n \geq 1$, whose Dynkin diagram is 
\vspace{0.2cm}
\begin{center}
\begin{tikzpicture}[scale=1.5]
\draw (1, 0) circle (0.1);
\draw (3, 0) circle (0.1);
\draw (1.1, 0.) -- (1.5, 0.);
\draw (1.7, 0.) -- (1.82, 0.);
\draw (1.94, 0.) -- (2.06, 0.);
\draw (2.18, 0.) -- (2.3, 0.);
\draw (2.5, 0) -- (2.9, 0);
\draw (3., -0.3) node {$\alpha _n$};
\draw (1, -0.3) node {$\alpha _1$};
\end{tikzpicture}
\end{center}
Then $\beta_l = \sum_{i=1}^l \alpha_i$ is a root, with $1 \leq l \leq n$, and $A_{\beta_m, \beta_l} = 1$, with $1 \leq m < l \leq n$.
\end{corollary}

\section{Structure of strings and general strategy}\label{section:structure}
From now on, let $\Sigma$ be a (possibly non-reduced) abstract root system, let $\Sigma^+$ be the set of positive roots, and let $\Pi$ be the set of simple roots for this criterion of positivity. Let $\Phi$ be a subset of $\Pi$ and let $\lambda$ be a root in $\Sigma \cup \{0\}$. We define the \emph{$\Phi$-string of $\lambda$} as the set of all the elements in $\Sigma \cup \{0\}$ of the form $\lambda + \sum_{\alpha \in \Phi} n_{\alpha} \alpha$, with $n_{\alpha} \in \mathbb{Z}$ for each $\alpha \in \Phi$. In what follows, we write $I_\Phi^\lambda$ to denote the set of elements of the $\Phi$-string of $\lambda$. Recall that the main aim of this paper is to determine the $\Phi$-string of $\lambda$ explicitly, for any abstract root system $\Sigma$, any subset $\Phi$ of simple roots, and any root $\lambda \in \Sigma$.

In this regard, this section is completely devoted to study the structure of these $\Phi$-strings, in order to design a strategy that will allow to determine all of them by means of a case-by-case investigation. Let us be more precise. As above, let $\Phi$ be a subset of $\Pi$ and $\lambda$ a root in $\Sigma$. Under certain extra circumstances (see Proposition~\ref{proposition:root:system:lambda:phi}), we will prove that $\Phi \cup \{\lambda\}$ constitutes a simple system for the root subsystem of $\Sigma$ that it spans. All the roots of the $\Phi$-string of $\lambda$ will be in this root subsystem spanned by $\Phi \cup \{\lambda\}$. This will allow us to reduce the determination of $\Phi$-strings in $\Sigma$ to a simpler computation in a smaller root system whose type can be deduced by the initial choice of $\Phi$.

\subsection{Structure of strings}
Let $\Phi$ be a subset of the set $\Pi$ of simple roots. Thus $\Phi$ spans a root subsystem of $\Sigma$, which we denote by $\Sigma_\Phi = \spann_\mathbb{Z} (\Phi) \cap \Sigma$. We take in $\Sigma_\Phi$ the positivity criterion of $\Sigma$, that is, $\Sigma_{\Phi}^{+} = \Sigma^{+ } \cap \Sigma_\Phi$. Now, $\Phi$ constitutes a simple system for $\Sigma_\Phi$. Moreover, we denote by $\Sigma^{\Phi}$ the set of positive roots of $\Sigma$ not spanned by $\Phi$, that is, $\Sigma^\Phi = \Sigma^{+} \backslash  \Sigma^{+}_{\Phi}$.

If $\lambda \in \Sigma$ is a root spanned by $\Phi$, $\lambda \in \Sigma_\Phi$, then the $\Phi$-string of $\lambda$ consists of the root subsystem $\Sigma_\Phi$ of $\Sigma$. Therefore, from now on we will assume that $\Phi$ is a proper subset of $\Pi$ and that $\lambda$ is not spanned by $\Phi$. Moreover, we can and will assume that $\lambda$ belongs to $\Sigma^{\Phi}$, that is, $\lambda$ is a positive root of $\Sigma$ not spanned by $\Phi$. This is possible because the $\Phi$-string of $\lambda$ would consist of the opposite roots to these in the $\Phi$-string of $-\lambda$, that is, $I_\Phi^\lambda = - I_\Phi^{-\lambda}$.

\begin{proposition} \label{proposition:root:system:lambda:phi}
Let $\Phi$ be a proper subset of the set $\Pi$ of simple roots and let $\lambda \in \Sigma^{\Phi}$ be a root of minimum level in its $\Phi$-string. Then:
\begin{enumerate}[{\rm (i)}]
\item The set $\Sigma_{\lambda, \Phi} = \spann_{\mathbb{Z}} (\{\lambda\} \cup \Phi) \cap \Sigma$ is a root subsystem of $\Sigma$ for which $\Pi_{\lambda, \Phi}=\{\lambda\} \cup \Phi$ is a simple system. \label{proposition:root:system:lambda:phi:1}
\item $\lambda$ is the unique root of minimum level in its $\Phi$-string.\label{proposition:root:system:lambda:phi:2}
\item If $\lambda + \sum_{\alpha \in \Phi} n_\alpha \alpha \in I_\Phi^\lambda$, then $n_\alpha$ is a non-negative integer for each $\alpha \in \Phi$.\label{proposition:root:system:lambda:phi:3}
\end{enumerate}
\end{proposition}

\begin{proof}
(\ref{proposition:root:system:lambda:phi:1}): Let $\lambda \in \Sigma^{\Phi}$ be a root of minimum level in its $\Phi$-string. Since $\lambda$ is not spanned by $\Phi$, we have that $\Pi_{\lambda, \Phi} = \{ \lambda \} \cup \Phi$ is a basis for its span. The set $\Sigma\cap\mathrm{span}_\mathbb{Z} \Pi_{\lambda, \Phi}$ satisfies the three conditions of a root system (see Section~\ref{section:preliminaries}). We denote by $\Sigma_{\lambda, \Phi} =\Sigma\cap\mathrm{span}_\mathbb{Z} \Pi_{\lambda, \Phi}$ the new root system and use the positivity criterion in $\Sigma$ to induce a positivity criterion in $\Sigma_{\lambda, \Phi}$. Now, we need to see that $\Pi_{\lambda, \Phi} = \{ \lambda \} \cup \Phi$ is a simple system for the root system $\Sigma_{\lambda, \Phi}$. In other words, we need to see that each root $\alpha \in \Pi_{\lambda, \Phi}$ cannot be written as $\alpha = \nu_1 + \nu_2$, for any $\nu_1$, $\nu_2 \in {\Sigma_{\lambda, \Phi}^{+}}$. In particular, this is true for all $\alpha \in \Phi$, since it is true in the root system $\Sigma$ and $\Sigma_{\lambda, \Phi} \subset \Sigma$. Put now  
\begin{equation}\label{equation:minimum:root}
\lambda = \nu_1 + \nu_2 = \left( n^{1}_{\lambda} \lambda + \sum_{\alpha \in \Phi} n^{1}_{\alpha} \alpha \right) + \left( n^{2}_{\lambda} \lambda + \sum_{\alpha \in \Phi} n^{2}_{\alpha} \alpha \right),
\end{equation}
where $\nu_1$, $\nu_2 \in {\Sigma^{+}_{\lambda, \Phi}}$ and the coefficient $n^{k}_{\nu}$ is integer for each $k \in \{1,2\}$ and each $\nu \in \{\lambda\} \cup \Phi$. Since $\Pi_{\lambda, \Phi}$ is a basis for its span, from~\eqref{equation:minimum:root} we deduce $1 = n^{1}_{\lambda} + n^{2}_{\lambda}$ and $n^{2}_{\alpha} = -n^{1}_{\alpha}$, for each $\alpha \in \Phi$. Since $\nu_1$, $\nu_2$ must be in $\Sigma^{+}$ and $\lambda$ is not spanned by $\Phi$, then $n^1_\lambda$, $n^2_\lambda = 1-n^1_\lambda \geq 0$ and hence we can and will assume $n^1_\lambda = 1$ (otherwise $n^2_\lambda = 1$ and we rename coefficients). Thus, we have
\begin{equation}\label{equation:minimum:root:third:step}
\lambda = \nu_1 + \nu_2 = \left(\lambda + \sum_{\alpha \in \Phi} n^1_{\alpha} \alpha \right) + \left(\sum_{\alpha \in \Phi} (-n^1_{\alpha}) \alpha \right).
\end{equation}
Now, if $\sum_{\alpha \in \Phi} n^1_{\alpha} < 0$, then $\nu_1$ would be a root in the $\Phi$-string of $\lambda$ of lower level than $\lambda$, which is a contradiction. Thus $\sum_{\alpha \in \Phi} n^1_{\alpha} \geq 0$ and $\nu_2$ cannot be a positive root. This proves that $\Pi_{\lambda, \Phi}$ is a simple system for the root system $\Sigma_{\lambda, \Phi}$.
		
(\ref{proposition:root:system:lambda:phi:2}): Let $\gamma = \lambda + \sum_{\alpha \in \Phi} n_{\alpha} \alpha$ be another root of minimum level in the $\Phi$-string of $\lambda$. We think of $\lambda$ and $\gamma$ in the root system $\Sigma_{\lambda, \Phi}$ with simple system $\Pi_{\lambda, \Phi}$. When we express $\gamma$ with respect to $\Pi_{\lambda, \Phi}$, the coefficient corresponding to $\lambda$ is one, and thus we deduce $n_{\alpha} \geq 0$ for all $\alpha \in \Phi$ by virtue of Proposition~\ref{proposition:simple:basis}. Since $\gamma$ is also of minimum level by assumption, we must have $\sum_{\alpha \in \Phi} n_{\alpha} = 0$. This means that $n_{\alpha} = 0$ for each $\alpha \in \Phi$, and thus $\gamma = \lambda$.

(\ref{proposition:root:system:lambda:phi:3}): This follows from combining the fact that $\Sigma_{\lambda, \Phi}$ is root system by virtue of~(\ref{proposition:root:system:lambda:phi:1}) with Proposition~\ref{proposition:simple:basis}.
\end{proof}

If $\Sigma$ is a root system, then each node of its Dynkin diagram is connected to at most three nodes~\cite[Proposition~2.78 (c)]{K}. From this and Proposition~\ref{proposition:root:system:lambda:phi}, we obtain

\begin{corollary}\label{corollary:three:simple:roots}
Let $\Phi_1, \dots, \Phi_n$ be connected and mutually orthogonal subsets of the set $\Pi$ of simple roots, and put $\Phi = \Phi_1 \cup \dots \cup \Phi_n$, with $n \geq 1$. Assume that $\Phi \neq \Pi$. Let $\lambda \in \Sigma^{\Phi}$ be the root of minimum level in its $\Phi$-string. Then:
\begin{enumerate}[{\rm (i)}]
\item If $n > 3$, the root system $\Sigma_{\lambda, \Phi}$ is reducible. In other words, the set $ \{ \alpha \in \Phi \, : \, A_{\alpha, \lambda} < 0  \} = \{ \alpha \in \Phi \, : \, \lambda + \alpha \in \Sigma^+ \}$ contains at most three roots.
\item  The $\Phi$-string of $\lambda$ coincides with the $\Psi$-string of $\lambda$, where $\Psi = \bigcup_{i=1}^k \Phi_{\sigma(i)}$ with $k \leq 3$, for some permutation $\sigma$ of the set $\{1, \dots, n \}$.
\end{enumerate}

\end{corollary}

The next result can be understood as a generalization of the fact that $\alpha$-strings have no gaps, as expressed in Proposition~\ref{proposition:strings}~(\ref{proposition:strings:5}).

\begin{proposition}\label{proposition:string:simple:rest}
Let $\Phi$ be a proper subset of $\Pi$ and let $\gamma \in \Sigma^{\Phi}$ be a root not of minimum level in its $\Phi$-string. Then there exists $\beta \in \Phi$ such that $\gamma - \beta$ is a root in the $\Phi$-string of $\lambda$. 
\end{proposition}

\begin{proof}
First, we will assume that $\Phi$ is connected and proceed by induction on the number of roots in $\Phi$. If $\Phi$ contains just one root, our claim follows from Proposition~\ref{proposition:strings}~(\ref{proposition:strings:5}). From now on, we assume that $\Phi$ contains at least two roots. Hence, $\Sigma$ is not of type $G_2$, as this root system has two simple roots and $\Phi$ is a proper subset of $\Pi$.

Let us then assume that the result is true for $\Psi$-strings with $\Psi$ connected and $|\Psi| = n-1$, and put $|\Phi| = n$. Let $\gamma \in \Sigma^{\Phi}$ be a root not of minimum level in its $\Phi$-string. We will think of $\gamma$ in the root system $\Sigma_{\lambda, \Phi}$ with simple system $\Pi_{\lambda, \Phi}$ (see Proposition~\ref{proposition:root:system:lambda:phi}). According to Lemma~\ref{exercise}, we can write $\gamma = \mu + \nu$, for $\mu \in \Sigma_{\lambda, \Phi}^{+}$ and $\nu \in \Pi_{\lambda, \Phi}$. If $\nu$ is in $\Phi$, then we are done. Thus, assume that $\mu$ is a root spanned by $\Phi$ and $\nu = \lambda$. If $\mu \in \Phi$, we are done again. 
	
Hence, assume that $\mu$ is a root spanned by $\Phi$ with level $\geq 2$ (since it is spanned by $\Phi$ its level is the same in $\Sigma$ and in $\Sigma_{\lambda, \Phi}$). Since both $\mu$ and $\mu + \lambda = \gamma$ are roots but $\mu - \lambda$ cannot be a root, from Proposition~\ref{proposition:strings}~(\ref{proposition:strings:5}) we deduce $A_{\lambda, \mu} < 0$. By regarding $\mu$ as a root in the root system $\Sigma_{\Phi}$ with simple system $\Phi$, we get that $\mu - \alpha$ is a positive root spanned by $\Phi$ for some $\alpha \in \Phi$, by virtue of Lemma~\ref{exercise}. On the one hand, if $A_{\lambda, \mu - \alpha} >0$, from Proposition~\ref{proposition:strings}~(\ref{proposition:strings:4}) we would obtain that $\mu - \alpha- \lambda$ is a positive root or zero. This is not possible since $\lambda$ is not spanned by $\Phi$ and $\mu - \alpha$ is a positive root spanned by $\Phi$. On the other hand, if $A_{\lambda, \mu - \alpha} < 0$, from Proposition~\ref{proposition:strings}~(\ref{proposition:strings:4}) we would obtain that $\mu - \alpha + \lambda = \gamma - \alpha$ is a root and the result follows. Thus, assume that $0 = A_{\lambda, \mu - \alpha} = A_{\lambda, \mu} - A_{\lambda,\alpha}$. Then $ A_{\lambda,\alpha} = A_{\lambda, \mu} <0$ and $\alpha$ is connected to $\lambda$ in the Dynkin diagram of the root system $\Sigma_{\lambda, \Phi}$. Since $\Phi$ is connected, if there were another root $\beta \in \Phi$ connected to $\lambda$, we would have a loop. Thus, $\alpha \in \Phi$ is the unique root connected to $\lambda$ in the Dynkin diagram of the root system $\Sigma_{\lambda, \Phi}$. Put $\mu = \sum_{\beta \in \Phi } n_{\beta} \beta$, where $n_\beta \geq 0$ is an integer for each $\beta \in \Phi$. 
	
Therefore, from $ A_{\lambda,\alpha} = A_{\lambda, \mu}  <0$ and $A_{\beta, \lambda} = 0$ for all $\beta \in \Phi \backslash \{ \alpha \}$, we deduce $n_{\alpha} =1$. From this and Lemma~\ref{lemma:sum:orthogonal:subsets}, we deduce that $\mu - \alpha$ is spanned by a connected subset $\Psi \subset \Phi \backslash \{ \alpha \}$. Thus, we have that $\mu = \alpha + \sum_{\beta \in \Psi} n_{\beta} \beta$ is a root in the $\Psi$-string of $\alpha$, with $\alpha \in \Phi$ not spanned by $\Psi$. Thus $\mu$ is not of minimum level in its $\Psi$-string, as in that case $\mu = \alpha$ and we are assuming $l(\mu) \geq 2$. Thus, by induction hypothesis, we can take $\beta \in \Psi \subset \Phi \backslash \{\alpha\}$ such that $\mu - \beta$ is a root. But then $A_{\lambda, \mu-\beta} = A_{\lambda, \mu} <0$ and thus $\lambda + \mu -\beta = \gamma -\beta$ is a positive root by means of Proposition~\ref{proposition:strings}~(\ref{proposition:strings:4}). This proves our claim when $\Phi$ is a connected subset of $\Pi$.

Thus, let $\Phi_1, \dots, \Phi_n$ be mutually orthogonal connected subsets of $\Pi$ and put $\Phi = \Phi_1 \cup \dots \cup \Phi_n$, with $n \geq 1$. Let $\lambda \in \Sigma^\Phi$ be of minimum level in its $\Phi$-string and let $\gamma = \lambda + \sum_{k=1}^n \gamma_k$ be a root not of minimum level in the $\Phi$-string of $\lambda$, with $\gamma_k$ in $\spann \Phi_k$ for each $k \in \{1, \dots, n\}$. Considering $\gamma$ in the root system $\Sigma_{\lambda, \Phi}$, there must exist $\nu \in \Pi_{\lambda, \Phi}$ such that $\gamma - \nu$ is a root, by means of Lemma~\ref{exercise}. If $\nu$ belongs to $\Phi$ our claim follows. Thus, assume $\nu = \lambda$. Hence, using Lemma~\ref{lemma:sum:orthogonal:subsets} we deduce that $\gamma = \lambda + \gamma_l$, for some $l \in \{1, \dots, n\}$. Therefore, $\gamma$ is a root in the $\Phi_l$-string of $\lambda$, and $\Phi_l$ is connected, so we are led to the case studied above.
\end{proof}

\begin{proposition}\label{proposition:root:minimum:level}
Let $\Phi$ be a connected and proper subset of $\Pi$ and let $\lambda \in \Sigma^\Phi$ be a root with non-trivial $\Phi$-string. If $\lambda$ is of minimum level in its $\Phi$-string, then there exists a root $\alpha \in \Phi$ such that $A_{\alpha, \lambda} < 0$ and $A_{\beta, \lambda} = 0$ for each $\beta \in \Phi \backslash \{ \alpha\}$. If $|\beta| \leq |\alpha|$ for each $(\beta, \alpha) \in I_\Phi^\lambda \times \Phi$, the converse is true. 
\end{proposition}

\begin{proof}
Let $\lambda \in \Sigma^{\Phi}$ be of minimum level in its $\Phi$-string. Since $\Sigma_{\lambda, \Phi}$ is a root system  for which $\Pi_{\lambda, \Phi}$ is a simple system (see Proposition~\ref{proposition:root:system:lambda:phi}), then $A_{\alpha, \lambda} \leq 0$ for each $\alpha \in \Phi$. If $A_{\alpha, \lambda} = 0$ for all $\alpha \in \Phi$, then $\Sigma_{\lambda, \Phi}$ is reducible and thus the $\Phi$-string of $\lambda$ would be trivial. Hence, there must exist a root $\alpha \in \Phi$ such that $A_{\alpha, \lambda} < 0$. If $A_{\beta, \lambda} < 0$ for some $\beta \in \Phi \backslash \{ \alpha \}$, then $\beta$ and $\alpha$ are connected to $\lambda$ in the Dynkin diagram of $\Sigma_{\lambda, \Phi}$. Since $\Phi$ is connected, there would be a loop, which is a contradiction. 

Let us prove the converse under the assumptions of the statement. Suppose that $\lambda$ is not of minimum level in its $\Phi$-string. Hence, from Proposition~\ref{proposition:string:simple:rest} we have that $\lambda - \alpha$ is a positive root in the $\Phi$-string of $\lambda$ for some $\alpha \in \Phi$. From Proposition~\ref{proposition:strings}~(\ref{proposition:strings:3}) we get that $A_{\alpha, \lambda}$, $A_{\alpha, \lambda-\alpha} \in \{0, \pm 1\}$. However, if $A_{\alpha, \lambda} \leq 0$, then $A_{\alpha, \lambda - \alpha} \leq -2$, which is a contradiction. Therefore, we deduce $A_{\alpha, \lambda} = 1$ and the result follows.
\end{proof}

The condition on the length of the roots in order to get the second implication of Proposition~\ref{proposition:root:minimum:level} cannot be omitted. See Remark~\ref{remark:minimum:level:characterization}.

\subsection{Strategy}\label{subsection:strategy}
Let $\Phi$ be a proper subset of the set $\Pi$ of simple roots of a root system $\Sigma$. Recall that $\Phi$ generates a root system $\Sigma_{\Phi}$ for which $\Phi$ is a simple system. Let $\lambda \in \Sigma^{\Phi}$ be the root of minimum level in its non-trivial $\Phi$-string. Now, $\Pi_{\lambda, \Phi} = \{ \lambda \} \cup \Phi$ constitutes a simple system for the root system $\Sigma_{\lambda, \Phi} =\Sigma\cap\mathrm{span}_\mathbb{Z} \Pi_{\lambda, \Phi}$, taking in this latter the positivity criterion induced by the one in $\Sigma$ (see Proposition~\ref{proposition:root:system:lambda:phi} for details). Thus, calculating the $\Phi$-string of $\lambda$ in $\Sigma$ is equivalent to obtaining the roots of the root system $\Sigma_{\lambda, \Phi}$ of the form
\begin{equation}\label{equation:string:count}
n_\lambda \lambda + \sum_{\alpha \in \Phi} n_\alpha \alpha
\end{equation}
with $n_\lambda =1$ and integers $n_\alpha \geq 0$ for each $\alpha \in \Phi$. In this regard, our general strategy to calculate the $\Phi$-string of $\lambda$ explicitly, for each $\Phi \subset \Pi$ and each $\lambda \in \Sigma^\Phi$, is the following:

\begin{enumerate}[1.]
\item First, we will determine the number of roots of the $\Phi$-string of $\lambda$, which is: the number of positive roots of the root system $\Sigma_{\lambda, \Phi}$, minus the number of positive roots of the root system $\Sigma_\Phi$ ($n_\lambda = 0$ in~\eqref{equation:string:count} for them), and minus the number of positive roots in $\Sigma_{\lambda, \Phi}$ with $n_\lambda \geq2$ in~\eqref{equation:string:count}. 


\item Then, we will explicitly obtain as many roots in the $\Phi$-string of $\lambda$ as it has. This process will be slightly different depending on whether $\Sigma_{\lambda, \Phi}$ is a classical root space or an exceptional root space. In the first case, we will obtain the roots starting from $\lambda$ and using results such as Proposition~\ref{proposition:strings}, Proposition~\ref{proposition:root:system:lambda:phi} or Proposition~\ref{proposition:root:minimum:level}. If $\Sigma_{\lambda, \Phi}$ happens to be an exceptional root system, then the list of its positive roots can be found in~\cite[p.~686--692]{K}, so we just gather those corresponding to the $\Phi$-string of~$\lambda$ from there.
	
\end{enumerate}

\section{Classical root systems}\label{section:classical:root:system}

In this section, $\Phi = \{ \alpha_1, \dots, \alpha_n \}$ will denote a proper connected subset of the set $\Pi$ of simple roots of the root system $\Sigma$, and $\lambda \in \Sigma^\Phi$ will denote the root of minimum level in its $\Phi$-string. Therefore, $\Sigma_{\lambda, \Phi} = \spann_{\mathbb{Z}} (\{\lambda\} \cup \Phi) \cap \Sigma$ is a root subsystem of $\Sigma$ for which $\Pi_{\lambda, \Phi} = \{\lambda\} \cup \Phi$ is a simple system, by virtue of Proposition~\ref{proposition:root:system:lambda:phi}. This section is completely devoted to determine the $\Phi$-string of $\lambda$ explicitly, whenever $\Sigma_{\lambda, \Phi}$ gives rise to a classical root system. 

In this line, in Subsection~\ref{subsection:an} we address all the cases when $\Sigma_\Phi$ is an $A_n$ root system, with $n \geq 1$. Similarly, in Subsection~\ref{subsection:bn} we consider simultaneously the pairs $(\Sigma_{\Phi}, \Sigma_{\lambda, \Phi}) \cong (B_n, B_{n+1})$ and $(\Sigma_{\Phi}, \Sigma_{\lambda, \Phi}) \cong (BC_n, BC_{n+1})$, with $n \geq 2$. Then, in Subsection~\ref{subsection:cn}, we analyze the pair $(\Sigma_{\Phi}, \Sigma_{\lambda, \Phi}) \cong (C_n, C_{n+1})$, with $n \geq 3$. Finally in Subsection~\ref{subsection:dn}, we consider the pair $(\Sigma_{\Phi}, \Sigma_{\lambda, \Phi}) \cong (D_n, D_{n+1})$, with $n \geq 3$.

In order to determine all these $\Phi$-strings, we will follow the strategy explained in Subsection~\ref{subsection:strategy}: first, we calculate the number of roots of the $\Phi$-string of $\lambda$, and then we will explicitly construct as many roots belonging to the $\Phi$-string of $\lambda$ as it has.

\subsection{$\Sigma_\Phi$ is an $A_n$ root subsystem of $\Sigma$}\label{subsection:an}
\begin{proposition}\label{proposition:a:string:type}
Let $ (\Sigma_\Phi, \Sigma_{\lambda, \Phi})$ be a pair of type $(A_n, A_{n+1})$, $(A_n, B_{n+1})$ or $(A_n, BC_{n+1})$, with $n \geq 1$, and let 
\\
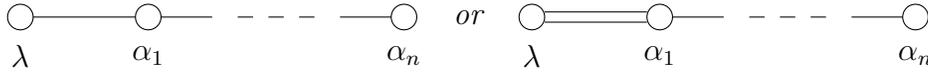
\begin{equation*}\label{string:a:b:1}
\begin{tikzpicture}[scale=1.7]
\draw (0, 0) circle (0.1);
\draw (1, 0) circle (0.1);
\draw (3, 0) circle (0.1);
\draw (0.1, 0.0) -- (0.9, -0.0);
\draw (1., -0.3) node {$\alpha _1$};
\draw (1.1, 0.) -- (1.5, 0.);
\draw (1.7, 0.) -- (1.82, 0.);
\draw (1.94, 0.) -- (2.06, 0.);
\draw (2.18, 0.) -- (2.3, 0.);
\draw (2.5, 0) -- (2.9, 0);
\draw (3., -0.3) node {$\alpha _n$};
\draw (0, -0.3) node {$\lambda$};
\draw (4, 0) circle (0.1);
\draw (5, 0) circle (0.1);
\draw (7, 0) circle (0.1);
\draw (4.1, -0.04) -- (4.9, -0.04);
\draw (4.1, 0.04) -- (4.9, 0.04);
\draw (5., -0.3) node {$\alpha _1$};
\draw (5.1, 0.) -- (5.5, 0.);
\draw (5.7, 0.) -- (5.82, 0.);
\draw (5.94, 0.) -- (6.06, 0.);
\draw (6.18, 0.) -- (6.3, 0.);
\draw (6.5, 0) -- (6.9, 0);
\draw (7., -0.3) node {$\alpha _n$};
\draw (4, -0.3) node {$\lambda$};
\draw (3.5, 0) node {$\text{or}$};
\end{tikzpicture}
\end{equation*}
be the Dynkin diagram of $\Sigma_{\lambda, \Phi}$, respectively (the second one holds for both $\Sigma_{\lambda, \Phi} \cong B_{n+1}$ and $\Sigma_{\lambda, \Phi} \cong BC_{n+1}$). Then, all the roots in $I_{\Phi}^{\lambda}$ have length $|\lambda|$ and
\[
I_{\Phi}^{\lambda} = \{ \lambda \} \cup \left\{ \lambda + \sum_{i=1}^l \alpha_i \, : \, 1 \leq l \leq n  \right\}.
\]
\end{proposition}

\begin{proof}
According to Subsection~\ref{subsection:strategy}, the number of roots in the $\Phi$-string of $\lambda$ will be at most: the number of positive roots of the root system $\Sigma_{\lambda, \Phi} \cong A_{n+1}$ (respectively $\Sigma_{\lambda, \Phi} \cong B_{n+1}$ or $\Sigma_{\lambda, \Phi} \cong BC_{n+1}$), minus the number of positive roots of $\Sigma_\Phi \cong A_n$. We will see that $|I_{\Phi}^{\lambda}| \leq n+1$, which is direct when $\Sigma_{\lambda, \Phi} \cong A_{n+1}$ (see~\cite[p.~684]{K}). 

Note that $\beta_k = \sum_{i=1}^k \alpha_i$ is in $\Sigma$ and $A_{\beta_m, \beta_l} = 1$, for each $k$, $l$, $m \in \{1, \dots, n\}$ with $m < l$, as follows from Corollary~\ref{corollary:an}. If $\Sigma_{\lambda, \Phi}$ is not of type $A_{n+1}$, we get that $\beta_l + 2 \lambda$ and $\beta_m + \beta_l + 2 \lambda$ are in $\Sigma_{\lambda, \Phi}$ (but not in $I_\Phi^\lambda$) for each $l$, $m \in \{1, \dots, n\}$ with $m < l$, by using Proposition~\ref{proposition:strings}~(\ref{proposition:strings:5}) twice, first for $A_{\lambda, \beta_l} = A_{\lambda, \alpha_1} = -2$ and then for $A_{\beta_m, 2\lambda+\beta_l} = -1$. Thus, if $\Sigma_{\lambda, \Phi} \cong B_{n+1}$, then $|I_{\Phi}^{\lambda}| \leq n+1$ again (see~\cite[p.~684]{K}). Finally, if $\Sigma_{\lambda, \Phi} \cong BC_{n+1}$, then $2 \lambda$ and $2 (\beta_l + \lambda)$ are roots, as follows from combining $A_{\beta_l, \beta_l + 2\lambda} = 0$ with the fact that $2 \lambda$ is a root, and with Proposition~\ref{proposition:strings}~(\ref{proposition:strings:5}), for each $l \in \{1, \dots, n\}$. In summary, we always have $|I_{\Phi}^{\lambda}| \leq n+1$ (see~\cite[p.~339]{Jurgen}).

Note that $A_{\beta_l, \lambda} = -1$ and then we deduce that $\lambda + \beta_l$ is a root, for each $l \in \{ 1, \dots, n\}$, as follows from Proposition~\ref{proposition:strings}~(\ref{proposition:strings:4}). Thus, the $\Phi$-string of $\lambda$ consists of at least the root $\lambda$ and the $n$ roots of the form $\lambda + \beta_l$, with $l \in \{ 1, \dots, n\}$. Since $|I_{\Phi}^{\lambda}| \leq n+1$, they are all. The claim concerning the length of the roots of the $\Phi$-string of $\lambda$ comes from the computation
\[
|\lambda+\beta_l|^2 = |\lambda|^2 + |\beta_l|^2 + |\beta_l|^2 A_{\beta_l, \lambda} = |\lambda|^2, \, \, \text{for each $l \in \{ 1, \dots, n\}$}. \qedhere
\]
\end{proof}

\begin{proposition}\label{proposition:c:string:type}
Let $ (\Sigma_\Phi, \Sigma_{\lambda, \Phi}) \cong (A_n, C_{n+1})$, with $n \geq 1$, and let 
\\
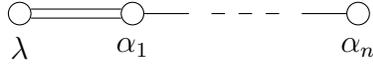
\begin{equation*}\label{string:c:1}
\begin{tikzpicture}[scale=1.5]
\draw (0, 0) circle (0.1);
\draw (1, 0) circle (0.1);
\draw (3, 0) circle (0.1);
\draw (0.1, -0.04) -- (0.9, -0.04);
\draw (0.1, 0.04) -- (0.9, 0.04);
\draw (1., -0.3) node {$\alpha _1$};
\draw (1.1, 0.) -- (1.5, 0.);
\draw (1.7, 0.) -- (1.82, 0.);
\draw (1.94, 0.) -- (2.06, 0.);
\draw (2.18, 0.) -- (2.3, 0.);
\draw (2.5, 0) -- (2.9, 0);
\draw (3., -0.3) node {$\alpha _n$};
\draw (0, -0.3) node {$\lambda$};
\end{tikzpicture}
\end{equation*}
be the Dynkin diagram of $\Sigma_{\lambda, \Phi}$. Then
\[
I_{\Phi}^{\lambda} = \{ \lambda \} \cup \left\{ \lambda + \sum_{i=1}^l \alpha_i, \, \lambda + \sum_{i=1}^l \alpha_i + \sum_{j=1}^m \alpha_j\, : \, \substack{  1 \leq l \leq n \\ \, 1 \leq m \leq l} \right\}.
\]
Moreover, for any $k$, $l$, $m \in \{1, \dots n\}$ with $m < l$, we have
\[
|\lambda|^2 = |\lambda + 2 \sum_{i=1}^k \alpha_i|^2 = 2 |\lambda + \sum_{i=1}^l \alpha_i + \sum_{j=1}^m \alpha_j|^2.
\]
\end{proposition}

\begin{proof}
According to Subsection~\ref{subsection:strategy}, the number of roots of $I_{\Phi}^{\lambda}$ will be at most: the number of positive roots of the root system $\Sigma_{\lambda, \Phi} \cong C_{n+1}$, minus the number of positive roots of $\Sigma_\Phi \cong A_n$. Thus, we have $|I_{\Phi}^{\lambda}| \leq (n+1)(n+2)/2$ (see~\cite[p. 685]{K}).

Recall that $\beta_l = \sum_{i=1}^l \alpha_i$ is a positive root of $\Sigma$ spanned by $\Phi$, for each $l \in \{1, \dots, n\}$, by means of Corollary~\ref{corollary:an}. Note also that $|\lambda|^2 = 2|\alpha|^2$ for any $\alpha \in \Sigma_\Phi$. Thus $A_{\beta_l, \lambda} = A_{\alpha_1, \lambda} = -2$, and from Proposition~\ref{proposition:strings}~\eqref{proposition:strings:5} we get that $\lambda + \varepsilon \beta_l$ is a positive root (in the $\Phi$-string of $\lambda$) for each $l \in \{1, \dots, n\}$ and each $\varepsilon \in \{1,2\}$. Now, recalling $A_{\beta_m, \beta_l} = 1$ for any $l$, $m \in \{1, \dots, n\}$ with $m < l$, from Corollary~\ref{corollary:an}, we have $A_{\beta_m, \lambda + \beta_l} = -1$ whenever $1\leq m < l \leq n$. Using again Proposition~\ref{proposition:strings}~\eqref{proposition:strings:5}, this implies that $\lambda+\beta_l+\beta_m$ is a root in the $\Phi$-string of $\lambda$ when $1\leq m < l \leq n$. Therefore, $\lambda$, $\lambda + \varepsilon \beta_l$, and $\lambda+\beta_l+\beta_m$, for each $l$, $m \in \{1, \dots, n\}$ with $m < l$, and each $\varepsilon \in \{1,2\}$, they are $(n+1)(n+2)/2$ roots in the $\Phi$-string of $\lambda$. Since $|I_{\Phi}^{\lambda}| \leq (n+1)(n+2)/2$, they are all of them.

In order to finish the proof, we just need to justify the claim concerning the length of the roots in $I_{\Phi}^{\lambda}$. Note that all the roots of the form $\beta_l$, with $l \in \{1, \dots, n\}$, have the same length, as they belong to $\Sigma_\Phi \cong A_n$. Recalling that $|\lambda|^2 = 2|\alpha|^2$ for any $\alpha \in \Sigma_\Phi$, we have
\[
|\lambda + \varepsilon \beta_l|^2 = |\lambda|^2 + \varepsilon^2 |\beta_l|^2 + \varepsilon |\beta_l|^2 A_{\beta_l, \lambda} = \varepsilon|\beta_l|^2,
\]
for each $l \in \{1, \dots, n\}$ and each $\varepsilon \in \{1,2\}$. Moreover,
\[
|\lambda + \beta_l + \beta_m|^2 = |\lambda + \beta_l|^2 + |\beta_m|^2 + |\beta_m|^2 A_{\beta_m, \lambda+\beta_l} = |\lambda + \beta_l|^2 =  |\beta_l|^2
\]
for any $l$, $m \in \{1, \dots, n\}$ with $m < l$. This completes the proof.
\end{proof}

\begin{remark}\label{remark:minimum:level:characterization}
In Proposition~\ref{proposition:c:string:type}, taking $n =2$ and $\gamma = \lambda + \alpha_1$, we have that $A_{\alpha_1, \gamma} = 0$ and $A_{\alpha_2, \gamma} = -1$. However, $\gamma$ is not of minimum level in the $\Phi$-string of $\lambda$. Thus, we cannot skip the condition on the lengths for the second implication in Proposition~\ref{proposition:root:minimum:level}.
\end{remark}

\begin{proposition}\label{proposition:d:string:type}
Let $ (\Sigma_\Phi, \Sigma_{\lambda, \Phi}) \cong (A_n, D_{n+1})$, with $n \geq 3$, and let 
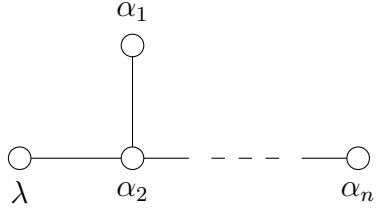
\begin{equation*}
\begin{tikzpicture}[scale=1.5]
\draw (0, 0) circle (0.1);
\draw (1, 0) circle (0.1);
\draw (1, 1) circle (0.1);
\draw (3, 0) circle (0.1);
\draw (0.1, 0.) -- (0.9, 0.);
\draw (1., -0.3) node {$\alpha _2$};
\draw (1., 0.1) -- (1., 0.9);
\draw (1., 1.3) node {$\alpha _1$};
\draw (1.1, 0.) -- (1.5, 0.);
\draw (1.7, 0.) -- (1.82, 0.);
\draw (1.94, 0.) -- (2.06, 0.);
\draw (2.18, 0.) -- (2.3, 0.);
\draw (2.5, 0) -- (2.9, 0);
\draw (3., -0.3) node {$\alpha _n$};
\draw (0, -0.3) node {$\lambda$};
\end{tikzpicture}
\end{equation*}
be the Dynkin diagram of $\Sigma_{\lambda, \Phi}$. Then all the roots in $I_{\Phi}^{\lambda}$ have the same length and
\[
I_{\Phi}^{\lambda}= \{ \lambda \} \cup \left\{ \lambda + \sum_{i=2}^l \alpha_i, \,  \lambda + \sum_{i=2}^l \alpha_i + \sum_{j=1}^k \alpha_j \, : \, \substack{2 \leq l \leq n \\ 1 \leq k \leq l-1}   \right\}.
\]
\end{proposition}

\begin{proof}
Note that $n \leq 2$ would lead to $(\Sigma_\Phi, \Sigma_{\lambda, \Phi}) \cong (A_n, A_{n+1})$, analyzed in Proposition~\ref{proposition:a:string:type}. According to Subsection~\ref{subsection:strategy}, the number of roots in the $\Phi$-string of $\lambda$ is at most: the number of positive roots of the root system $\Sigma_{\lambda, \Phi} \cong D_{n+1}$, minus the number of positive roots of $\Sigma_\Phi \cong A_n$. Thus, $|I_{\Phi}^{\lambda}| \leq n (n+1) /2$ (see~\cite[p. 684-685]{K}). 

Note that all the roots in the $\Phi \backslash \{ \alpha_1 \}$-string of $\lambda$ are contained in the $\Phi$-string of $\lambda$. Since $\lambda$ is the root of minimum level in its $\Phi$-string, then it is also the root of minimum level in its $\Phi \backslash \{ \alpha_1 \}$-string. From Proposition~\ref{proposition:a:string:type} we get that the $\Phi \backslash \{ \alpha_1 \}$-string of $\lambda$ consists of the root $\lambda$ and the roots of the form
\[
\lambda_l  = \lambda + \sum_{i=2}^{l} \alpha_i, \, \, \text{for each $l \in \{ 2, \dots, n \}$.}
\]
The root $\lambda_l$ is the root of minimum level in its $\{ \alpha_1, \dots, \alpha_{l-1}\}$-string for each $l \in \{ 2, \dots, n \}$, as follows from Proposition~\ref{proposition:root:minimum:level} taking into account $A_{\alpha_1,\lambda_l} <0$ and $A_{\alpha_i,\lambda_l}=0$ for each $i\in\{2,\dots,l-1\}$. Since $\{\alpha_1, \dots, \alpha_{l-1}\}$ is the simple system for an $A_{l-1}$ root system, then $\{ \alpha_1, \dots, \alpha_{l-1} \} \cup \{ \lambda_l\}$ is a simple system for an $A_l$ root system, for each $l \in \{2, \dots, n \}$, as follows from Proposition~\ref{proposition:root:system:lambda:phi}. Using again Proposition~\ref{proposition:a:string:type}, we obtain that the $\{ \alpha_1, \dots, \alpha_{l-1} \}$-string of $\lambda_l$ consists of the root $\lambda_l$ and the roots of the form
\[
\lambda_l^k =  \lambda_l + \sum_{j=1}^{k} \alpha_j, \, \, \text{for each $k \in \{1, \dots, l-1 \}$ and each $l \in \{ 2, \dots, n \}$.}
\]
Thus, $\lambda$, $\lambda_l$, with $l \in \{2, \dots, n\}$, and $\lambda_l^k$, with $l \in \{2, \dots, n\}$ and $k \in \{1, \dots, l-1 \}$, they are $n(n+1)/2 \geq |I_{\Phi}^{\lambda}|$ roots in the $\Phi$-string of $\lambda$, so they are all of them. They all have the same length since all the roots in a $D_{n+1}$ root system have the same length.
\end{proof}

\subsection{$\Sigma_\Phi$ is a $B_n$ or a $BC_n$ root subsystem of $\Sigma$.}\label{subsection:bn}

In this subsection, $\Phi = \{ \alpha_1, \dots, \alpha_n \}$ will be a proper subset of $\Pi$ generating a $B_n$ or a $BC_n$ root system $\Sigma_\Phi$, and $\lambda \in \Sigma^\Phi$ will denote, as usual, the root of minimum level in its $\Phi$-string. Then, $\Sigma_{\lambda, \Phi}$ must be: either a $B_{n+1}$ or an $F_4$ (provided that $n=3$) root system, when $\Sigma_\Phi \cong B_n$; or a $BC_{n+1}$ root system, when $\Sigma_\Phi \cong BC_n$. The $F_4$ case is addressed in Proposition~\ref{proposition:b:f:string:type}.

\begin{proposition}\label{proposition:b:string:type}
Let $(\Sigma_\Phi, \Sigma_{\lambda, \Phi})$ be a pair of type $(B_n, B_{n+1})$ or $(BC_n, BC_{n+1})$, $n \geq 2$, and let 
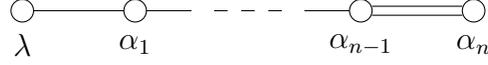
\begin{equation*}
\begin{tikzpicture}[scale=1.5]
\draw (0, 0) circle (0.1);
\draw (1, 0) circle (0.1);
\draw (3, 0) circle (0.1);
\draw (4, 0) circle (0.1);
\draw (3.1, -0.04) -- (3.9, -0.04);
\draw (3.1, 0.04) -- (3.9, 0.04);
\draw (4., -0.3) node {$\alpha _n$};
\draw (0.1, 0.) -- (0.9, 0.);
\draw (1., -0.3) node {$\alpha _1$};
\draw (1.1, 0.) -- (1.5, 0.);
\draw (1.7, 0.) -- (1.82, 0.);
\draw (1.94, 0.) -- (2.06, 0.);
\draw (2.18, 0.) -- (2.3, 0.);
\draw (2.5, 0) -- (2.9, 0);
\draw (3., -0.3) node {$\alpha_{n-1}$};
\draw (0, -0.3) node {$\lambda$};
\end{tikzpicture}
\end{equation*}
be the Dynkin diagram of $\Sigma_{\lambda, \Phi}$. Then
\[
I_{\Phi}^{\lambda} = \left\{ \lambda \right\} \cup \left\{\lambda + \sum_{i=1}^l \alpha_i, \lambda + \sum_{i=1}^n \alpha_i  + \sum_{j=0}^k \alpha_{n-j}\, : \, \substack{1 \leq l \leq n \\ 0 \leq k \leq n-1}  \right\}, 
\]
and $|\beta|^2 = 2 | \lambda + \sum_{i=1}^{n}  \alpha_i |^2$ for any $\beta \in I_{\Phi}^{\lambda}$ different from $\lambda + \sum_{i=1}^{n}  \alpha_i$.
\end{proposition}

\begin{proof}
According to Subsection~\ref{subsection:strategy}, the number of roots in the $\Phi$-string of $\lambda$ is at most: the number of positive roots of the root system $\Sigma_{\lambda, \Phi} \cong B_{n+1}$ (respectively $\Sigma_{\lambda, \Phi} \cong BC_{n+1}$), minus the number of positive roots of $\Sigma_\Phi \cong B_n$ (respectively $\Sigma_{\Phi} \cong BC_{n}$). Recall from the proof of Proposition~\ref{proposition:a:string:type} that $2 \lambda + 2 \sum_{\alpha\in\Phi} \alpha$ is a root (not in $I_\Phi^\lambda)$ when $\Sigma_{\lambda, \Phi} \cong BC_{n+1}$. Thus, for both cases we get $|I_{\Phi}^{\lambda}| \leq 2n + 1$ (see~\cite[p. 339]{Jurgen}  and~\cite[p. 684]{K}).

Since $\lambda$ is of minimum level in its $\Phi$-string, it is also of minimum level in its $\Phi \backslash \{ \alpha_n\}$-string. But $\Sigma_{\lambda, \Phi \backslash \{ \alpha_n\} }$ is an $A_n$ root system (see Proposition~\ref{proposition:root:system:lambda:phi}) and hence, from Proposition~\ref{proposition:a:string:type} we get that 
\[
\lambda_l =  \lambda + \sum_{i=1}^l \alpha_i, \, \, \, \text{for each $l \in \{ 1, \dots, n -1 \}$},
\]
is a root in the $\Phi$-string of $\lambda$, and they all have length $|\lambda|$. Moreover, $A_{\alpha_n, \lambda_{n-1}} =A_{\alpha_n, \alpha_{n-1}} = -2$ and hence $\lambda_n = \lambda_{n-1} + \alpha_n$ and $\lambda_n + \alpha_n$ are roots by means of Proposition~\ref{proposition:strings}~(\ref{proposition:strings:5}). Note that $2|\lambda_n|^2=|\lambda_n+\alpha_n|^2=|\lambda|^2$. Now, $\lambda_n + \alpha_n$ is the root of minimum level in its $\Phi \backslash \{ \alpha_n\}$-string, as follows from combining Proposition~\ref{proposition:root:minimum:level} with $A_{\alpha_{n-1}, \lambda_n + \alpha_n} = -1$ and $A_{\alpha_{l}, \lambda_n + \alpha_n} = 0$ for each $l$ in $\{ 1, \dots, n-2 \}$. Since $A_{\lambda_n + \alpha_n, \alpha_{n-1}} = -1$, then $\Sigma_{\lambda_n + \alpha_n, \Phi \backslash \{ \alpha_n \}}$ is an $A_n$ root system, and using Proposition~\ref{proposition:a:string:type}, we deduce that
\begin{equation}\label{equation:b:bc}
\lambda^k = \lambda + \sum_{i=1}^{n}  \alpha_i + \sum_{j=0}^k  \alpha_{n-j}, \, \, \text{for each $k$ in $\{1, \dots, n-1\}$},
\end{equation}
is a root in the $\Phi$-string of $\lambda$ of length $|\lambda_n + \alpha_n| = |\lambda|$. Note that $k = 0$ in~\eqref{equation:b:bc} gives $\lambda_n + \alpha_n$. Thus, $\lambda$, $\lambda_l$ with $l \in \{1, \dots, n\}$, and $\lambda^k$ with $k \in \{0, \dots, n-1\}$, they give rise to $2n+1 \geq |I_\Phi^\lambda|$ roots in the $\Phi$-string of $\lambda$, so they are all of them.
\end{proof}

\subsection{$\Sigma_\Phi$ is a $C_n$ root subsystem of $\Sigma$} \label{subsection:cn}

In this subsection, $\Phi = \{ \alpha_1, \dots, \alpha_n \}$ will be a subset of $\Pi$ generating a $C_n$ root system $\Sigma_\Phi$. As usual, $\lambda \in \Sigma^\Phi$ will denote the root of minimum level in its $\Phi$-string. Note that $\Sigma_{\lambda, \Phi}$ must be either a $C_{n+1}$ or an $F_4$, provided that $n=3$, root system. The $F_4$ case is addressed in Proposition~\ref{proposition:c:f:string:type}.

\begin{proposition}\label{proposition:c:pure:string:type}
Let $ (\Sigma_\Phi, \Sigma_{\lambda, \Phi}) \cong (C_n, C_{n+1})$, $n \geq 3$, and let 
\\
\begin{equation*}
\begin{tikzpicture}[scale=1.5]
\draw (0, 0) circle (0.1);
\draw (1, 0) circle (0.1);
\draw (3, 0) circle (0.1);
\draw (4, 0) circle (0.1);
\draw (3.1, -0.04) -- (3.9, -0.04);
\draw (3.1, 0.04) -- (3.9, 0.04);
\draw (4., -0.3) node {$\alpha _n$};
\draw (0.1, 0.) -- (0.9, 0.);
\draw (1., -0.3) node {$\alpha _1$};
\draw (1.1, 0.) -- (1.5, 0.);
\draw (1.7, 0.) -- (1.82, 0.);
\draw (1.94, 0.) -- (2.06, 0.);
\draw (2.18, 0.) -- (2.3, 0.);
\draw (2.5, 0) -- (2.9, 0);
\draw (3., -0.3) node {$\alpha_{n-1}$};
\draw (0, -0.3) node {$\lambda$};
\end{tikzpicture}
\end{equation*}
be the Dynkin diagram of $\Sigma_{\lambda, \Phi}$. Then all the roots in $I_{\Phi}^{\lambda}$ have the same length, and
\begin{equation*}
I_{\Phi}^{\lambda} = \left\{ \lambda \right\} \cup \left\{\lambda + \sum_{i=1}^l \alpha_i, \,  \lambda + \sum_{i=1}^{n} \alpha_i +\sum_{j=1}^{k} \alpha_{n-j} \, :  \, \substack{1 \leq l \leq n \\ 1 \leq k \leq  n-1}  \right\}.
\end{equation*}
\end{proposition}

\begin{proof}
The number of roots in the $\Phi$-string of $\lambda$ is at most: the number of positive roots of the root system $\Sigma_{\lambda, \Phi} \cong C_{n+1}$, minus the number of positive roots of $\Sigma_\Phi \cong C_n$. Furthermore, note that the root of maximum level in $\Sigma_{\lambda, \Phi}$ is $ 2 \lambda + 2\sum_{i=1}^{n-1} \alpha_i + \alpha_n$  (see~\cite[p. 685]{K}), which does not belong neither to $\Sigma_\Phi$ nor to $I_\Phi^\lambda$. Thus, we obtain $|I_{\Phi}^{\lambda}| \leq  2n$ (see~\cite[p. 685]{K}). 

Since $\lambda$ is of minimum level in its $\Phi$-string, it is also of minimum level in its $\Phi \backslash \{ \alpha_n \}$-string. Note that $\Sigma_{\lambda, \Phi \backslash \{ \alpha_n \} }$ is an $A_n$ root system. Therefore, using Proposition~\ref{proposition:a:string:type}, we deduce that 
\[
\lambda_l=\lambda + \sum_{i=1}^l \alpha_i, \, \, \text{for each $l \in \{1, \dots, n-1\}$},
\]
is a root of length of $|\lambda|$ in the $\Phi$-string of $\lambda$. Moreover, $A_{\alpha_n, \lambda_{n-1}} = -1$ and from Proposition~\ref{proposition:strings}~(\ref{proposition:strings:4}) we get that $\lambda_n = \lambda_{n-1} + \alpha_n$ is a positive root of the same length of $\lambda$, as $|\lambda_{n}|^2 = |\lambda_{n-1}|^2+|\alpha_n|^2 + |\alpha_n|^2 A_{\alpha_n, \lambda_{n-1}} = |\lambda_{n-1}|^2 = |\lambda|^2$. Note that $\lambda_n$ is of minimum level in its $\Phi \backslash \{\alpha_n\}$-string. This follows from Proposition~\ref{proposition:string:simple:rest}, taking into account that $\lambda_n - \alpha_i$ is not a root for any $i \in \{1, \dots, n-1\}$. Thus, $\Sigma_{\lambda_n, \Phi \backslash \{ \alpha_n \} }$ is an $A_n$ root system, as $A_{\alpha_{n-1}, \lambda_n} = A_{\lambda_n, \alpha_{n-1}} = -1$ and $A_{\alpha_{i}, \lambda_n} =0$ for any $i \in \{1, \dots, n-2 \}$. Thus, using again Proposition~\ref{proposition:a:string:type}, we deduce that 
\[
\lambda^k = \lambda_n +  \sum_{j=1}^k \alpha_{n-j} =  \lambda + \sum_{i=1}^n \alpha_i + \sum_{j=1}^k \alpha_{n-j}, \, \, \text{for each $k \in \{ 1, \dots ,n-1 \}$},
\]
is a root of length of $|\lambda|$ in the $\Phi$-string of $\lambda$. Then $\lambda$, $\lambda_l$ with $l \in \{1, \dots, n\}$, and $\lambda^k$ with $k \in \{1, \dots, n-1\}$, they are $2n \geq |I_{\Phi}^{\lambda}|$ roots of the $\Phi$-string of~$\lambda$, so they are all of them and we have seen that they all have the same length. 
\end{proof}

\subsection{$\Sigma_\Phi$ is a $D_n$ root subsystem of $\Sigma$}\label{subsection:dn}

\begin{proposition}\label{proposition:d:pure:string:type}
Let $ (\Sigma_\Phi, \Sigma_{\lambda, \Phi}) \cong (D_n, D_{n+1})$, with $n \geq 3$ and $D_3 \cong A_3$, and let 
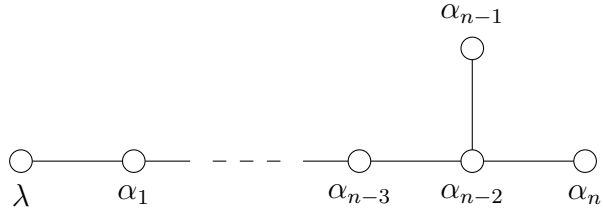
\begin{equation*}
\begin{tikzpicture}[scale=1.5]
\draw (0, 0) circle (0.1);
\draw (1, 0) circle (0.1);
\draw (3, 0) circle (0.1);
\draw (4, 1) circle (0.1);
\draw (4, 0) circle (0.1);
\draw (5, 0) circle (0.1);
\draw (0.1, 0.) -- (0.9, 0.);
\draw (1., -0.3) node {$\alpha _1$};
\draw (1.1, 0.) -- (1.5, 0.);
\draw (1.7, 0.) -- (1.82, 0.);
\draw (1.94, 0.) -- (2.06, 0.);
\draw (2.18, 0.) -- (2.3, 0.);
\draw (2.5, 0) -- (2.9, 0);
\draw (3., -0.3) node {$\alpha _{n-3}$};
\draw (3.1, 0.) -- (3.9, 0.);
\draw (4., -0.3) node {$\alpha _{n-2}$};
\draw (4.1, 0.) -- (4.9, 0.);
\draw (5., -0.3) node {$\alpha _n$};
\draw (4., 0.1) -- (4., 0.9);
\draw (4., 1.3) node {$\alpha _{n-1}$};
\draw (0, -0.3) node {$\lambda$};
\end{tikzpicture}
\end{equation*}
be the Dynkin diagram of $\Sigma_{\lambda, \Phi}$. Then, all the roots in $I_\Phi^\lambda$ have the same length and
\begin{equation}\label{string:dn:dn+1}
I_{\Phi}^{\lambda} =  \left\{ \lambda, \, \lambda + \sum_{\substack{i=1 \\ i \neq n-1}}^n \alpha_i  \right\} \cup \left\{\,\lambda + \sum_{i=1}^l \alpha_i,  \,  \lambda + \sum_{i=1}^{n} \alpha_i +\sum_{j=2}^{k} \alpha_{n-j} \, :  \, \substack{1 \leq l \leq n \\  2 \leq k \leq n-1}  \right\}.
\end{equation}

\end{proposition}

\begin{proof}
The number of roots in the $\Phi$-string of $\lambda$ will be at most: the number of positive roots of the root system $\Sigma_{\lambda, \Phi} \cong D_{n+1}$, minus the number of positive roots of $\Sigma_\Phi \cong D_n$. Hence $|I_{\Phi}^{\lambda}| \leq 2n$ (see~\cite[p.~685]{K}).

Since $\lambda$ is of minimum level in its $\Phi$-string, it is also of minimum level in its $\Phi \backslash \{ \alpha_n\}$ and its $\Phi \backslash \{ \alpha_{n-1}\}$-string. Note that both $\Sigma_{\lambda, \Phi \backslash \{ \alpha_n\}}$ and $\Sigma_{\lambda, \Phi \backslash \{ \alpha_{n-1}\}}$ are $A_n$ root systems. Thus, using Proposition~\ref{proposition:a:string:type}, we deduce that
\[
\lambda_n = \lambda + \sum_{i=1}^{n-2} \alpha_i + \alpha_n \quad  \text{and} \quad  \lambda_l = \lambda + \sum_{i=1}^{l} \alpha_i, \, \, \text{for each $l$ in $\{ 1, \dots ,n-1 \}$}, 
\]
are roots in the $\Phi$-string of $\lambda$. Moreover, $A_{\alpha_n, \lambda_{n-1}}=-1$ and hence $\gamma = \lambda_{n-1} + \alpha_n$ is a root in the $\Phi$-string of $\lambda$ by virtue of Proposition~\ref{proposition:strings}~(\ref{proposition:strings:4}). Now, $\Phi \backslash \{ \alpha_{n-1}, \alpha_n \}$ is a connected subset of $\Phi$ and $\gamma$ is of minimum level in its $\Phi \backslash \{ \alpha_{n-1}, \alpha_n \}$-string, as follows from Proposition~\ref{proposition:root:minimum:level} and the fact that $A_{\alpha_{n-2}, \gamma}=-1$ and $A_{\alpha_k, \gamma} = 0$, with $1 \leq k \leq n-3$. Now, $\Sigma_{\gamma, \Phi \backslash \{ \alpha_{n-1}, \alpha_n \}}$ is a root system of type $A_{n-1}$. Hence, using again Proposition~\ref{proposition:a:string:type} we deduce that 
\[
\gamma_k = \gamma + \sum_{j=2}^{k} \alpha_{n-j}= \lambda + \sum_{i=1}^n  \alpha_i  + \sum_{j=2}^{k} \alpha_{n-j}, \, \, \text{for each $k$ in $\{ 2, \dots, n-1 \}$},
\]
is a root in the $\Phi$-string of $\lambda$. Thus, $\lambda$, $\lambda_l$ with $l \in \{ 1, \dots, n \}$, $\gamma$, and $\gamma_k$ with $k \in \{ 2, \dots, n-1 \}$, they are $2n \geq |I_\Phi^\lambda|$ roots in the $\Phi$-string of $\lambda$, so they are all. They all have the same length since in the root system $\Sigma_{\lambda, \Phi} \cong D_{n+1}$ all the roots~do.
\end{proof}

\section{Exceptional root systems}\label{section:exceptional:root:system}
As usual in this paper, let $\Phi$ be a proper connected subset of the set $\Pi$ of simple roots, and let $\lambda \in \Sigma^\Phi$ be the root of minimum level in its $\Phi$-string. Recall that under these circumstances, $\Sigma_{\lambda, \Phi}$ is a root subsystem of $\Sigma$ for which $\Pi_{\lambda, \Phi} = \{\lambda\} \cup \Phi$ is a simple system, by virtue of Proposition~\ref{proposition:root:system:lambda:phi}. This section is completely devoted to determine the $\Phi$-string of $\lambda$ explicitly, whenever $\Sigma_{\lambda, \Phi}$ gives rise to an exceptional root system.

In this line, in Subsections~\ref{subsection:a:e} and~\ref{subsection:d:e} we address the cases $(\Sigma_\Phi, \Sigma_{\lambda, \Phi}) \cong (A_n, E_{n+1})$  and $(\Sigma_\Phi, \Sigma_{\lambda, \Phi}) \cong (D_n, E_{n+1})$, respectively, with $n \in \{5,6,7\}$. Then, in Subsection~\ref{subsection:e:e}, we analyze the case $(\Sigma_\Phi, \Sigma_{\lambda, \Phi}) \cong (E_n, E_{n+1})$, with $n \in \{6,7\}$. Finally, in Subsections~\ref{subsection:b:f} and~\ref{subsection:c:f}, we focus on the cases $(\Sigma_\Phi, \Sigma_{\lambda, \Phi}) \cong (B_3, F_4)$ and $(\Sigma_\Phi, \Sigma_{\lambda, \Phi}) \cong (C_3, F_4)$, respectively. In order to do so, we will follow the notation of~\cite{K} to express the roots in $E_8$ and $F_4$ root systems. More precisely, let 
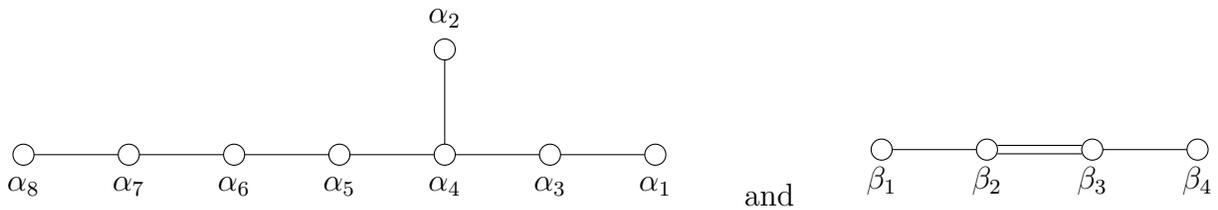
\begin{equation*}\label{string:e6:e7}
\begin{tikzpicture}[scale=1.4]
\draw (0, 0) circle (0.1);
\draw (1, 0) circle (0.1);
\draw (2, 0) circle (0.1);
\draw (3, 0) circle (0.1);
\draw (4, 0) circle (0.1);
\draw (5, 0) circle (0.1);
\draw (6, 0) circle (0.1);
\draw (4, 1) circle (0.1);
\draw (0.1, 0.) -- (0.9, 0.);
\draw (1., -0.3) node {$\alpha _7$};
\draw (1.1, 0.) -- (1.9, 0.);
\draw (2., -0.3) node {$\alpha _6$};
\draw (2.1, 0.) -- (2.9, 0.);
\draw (3., -0.3) node {$\alpha _5$};
\draw (3.1, 0.) -- (3.9, 0.);
\draw (4., -0.3) node {$\alpha _4$};
\draw (4.1, 0.) -- (4.9, 0.);
\draw (5., -0.3) node {$\alpha _3$};
\draw (5.1, 0.) -- (5.9, 0.);
\draw (6., -0.3) node {$\alpha _1$};
\draw (4., 0.1) -- (4., 0.9);
\draw (4., 1.3) node {$\alpha _2$};
\draw (0, -0.3) node {$\alpha _8$};
\end{tikzpicture}
\qquad \text{and} \qquad
\begin{tikzpicture}[scale=1.4]
\draw (0, 0) circle (0.1);
\draw (1, 0) circle (0.1);
\draw (2, 0) circle (0.1);
\draw (3, 0) circle (0.1);
\draw (0.1, 0.) -- (0.9, 0.);
\draw (1., -0.3) node {$\beta_2$};
\draw (1.1, -0.04) -- (1.9, -0.04);
\draw (1.1, 0.04) -- (1.9, 0.04);
\draw (2., -0.3) node {$\beta_3$};
\draw (2.1, 0.) -- (2.9, 0.);
\draw (3., -0.3) node {$\beta_4$};
\draw (0, -0.3) node {$\beta_1$};
\end{tikzpicture}
\end{equation*}
be the Dynkin diagrams of an $E_8$ and an $F_4$ root system, respectively. Then
\[
\left(
\begin{array}{@{}c@{}c@{}c@{}c@{}c@{}c@{}c@{}}
	&  &  &  & a_2 &  & \\
	a_8 \, & a_7 \, & a_6 \, & a_5 \, & a_4 \,& a_3 \, & a_1 \,
\end{array}
\right)
\qquad
\text{and}
\qquad
(b_1 \, b_2 \,b_3 \,b_4)
\]
denote the root $\sum_{i=1}^8 a_i \alpha_i$ and the root $\sum_{j=1}^4 b_j \beta_j$ of the root systems $E_8$ and $F_4$, respectively. Note that this notation allows us also to express roots in both $E_6$ or $E_7$ root systems, just by taking $a_8 = a_7 = 0$ or $a_8 = 0$, respectively.

As explained in Subsection~\ref{subsection:strategy}, the approach when $\Sigma_{\lambda, \Phi}$ is an exceptional root system is slightly different from the case when it is a classical root system. We will still start by calculating the number $m$ of roots in the $\Phi$-string of $\lambda$. However, after that we will just give a list of $m$ roots in $\Sigma_{\lambda, \Phi}^{+}$ which are at same time in the $\Phi$-string of $\lambda$ and in the corresponding list of positive roots of $E_6$ (\cite[p.~687]{K}), $E_7$ (\cite[p.~688]{K}), $E_8$ (\cite[p.~689--690]{K}) or $F_4$ (\cite[p.~691]{K}) root systems.

In Subsections~\ref{subsection:a:e}, \ref{subsection:d:e} and~\ref{subsection:e:e}, the root system $\Sigma_{\lambda, \Phi}$ will be always of type $E$. This implies that all the roots of the $\Phi$-string of $\lambda$ will have the same length (as $\lambda$). Therefore, we omit this argument in each proof for avoiding repetition.

\subsection{$\Sigma_\Phi$ is an $A_n$ root subsystem of $\Sigma$}\label{subsection:a:e}

Recall that $\Phi = \{ \alpha_1, \dots, \alpha_n \}$ is a proper connected subset of $\Pi$, $\lambda \in \Sigma^{\Phi}$ is the root of minimum level in its $\Phi$-string, and $\Sigma_{\lambda, \Phi}$ denotes the root subsystem of $\Sigma$ for which $\Pi_{\lambda, \Phi} = \{ \lambda \} \cup \Phi$ is a simple system.
\begin{proposition}\label{proposition:a:e:string:type}
Let $ (\Sigma_\Phi, \Sigma_{\lambda, \Phi}) \cong (A_n, E_{n+1})$, with $n$ in $\{5, 6, 7\}$, and let 
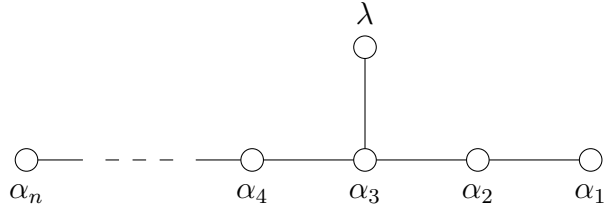
\begin{equation*}\label{string:a:e}
\begin{tikzpicture}[scale=1.5]
\draw (1, 0) circle (0.1);
\draw (3, 0) circle (0.1);
\draw (4, 0) circle (0.1);
\draw (5, 0) circle (0.1);
\draw (4, 1) circle (0.1);
\draw (6, 0) circle (0.1);
\draw (4., 0.1) -- (4., 0.9);
\draw (4., 1.3) node {$\lambda$};
\draw (3.1, 0.) -- (3.9, 0.);
\draw (4., -0.3) node {$\alpha _3$};
\draw (4.1, 0.) -- (4.9, 0.);
\draw (5., -0.3) node {$\alpha _2$};
\draw (5.1, 0.) -- (5.9, 0.);
\draw (6., -0.3) node {$\alpha _1$};
\draw (1.1, 0.) -- (1.5, 0.);
\draw (1.7, 0.) -- (1.82, 0.);
\draw (1.94, 0.) -- (2.06, 0.);
\draw (2.18, 0.) -- (2.3, 0.);
\draw (2.5, 0) -- (2.9, 0);
\draw (3., -0.3) node {$\alpha _4$};
\draw (1, -0.3) node {$\alpha _n$};
\end{tikzpicture}
\end{equation*}
be the Dynkin diagram of $\Sigma_{\lambda, \Phi}$. Then, the $\Phi$-string of $\lambda$ consists of the first 20 roots below if $n = 5$, the first 35 below if $n=6$, and all of them otherwise, and they all have the same length:
\begin{align*}
&\left(
\begin{array}{@{}c@{}c@{}c@{}c@{}c@{}c@{}c@{}}
	&  &  &  & 1 &  & \\
	0 & 0 & 0 & 0 & 0 & 0 & 0
\end{array}
\right),
\left(
\begin{array}{@{}c@{}c@{}c@{}c@{}c@{}c@{}c@{}}
	&  &  &  & 1 &  & \\
	0 & 0 & 0 & 0 & 1 & 0 & 0
\end{array}
\right),
\left(
\begin{array}{@{}c@{}c@{}c@{}c@{}c@{}c@{}c@{}}
	&  &  &  & 1 &  & \\
	0 & 0 & 0 & 1 & 1 & 0 & 0
\end{array}
\right),
\left(
\begin{array}{@{}c@{}c@{}c@{}c@{}c@{}c@{}c@{}}
	&  &  &  & 1 &  & \\
	0 & 0 & 0 & 0 & 1 & 1 & 0
\end{array}
\right),
\left(
\begin{array}{@{}c@{}c@{}c@{}c@{}c@{}c@{}c@{}}
	&  &  &  & 1 &  & \\
	0 & 0 & 1 & 1 & 1 & 0 & 0
\end{array}
\right),
\left(
\begin{array}{@{}c@{}c@{}c@{}c@{}c@{}c@{}c@{}}
	&  &  &  & 1 &  & \\
	0 & 0 & 0 & 1 & 1 & 1 & 0
\end{array}
\right),
\left(
\begin{array}{@{}c@{}c@{}c@{}c@{}c@{}c@{}c@{}}
	&  &  &  & 1 &  & \\
	0 & 0 & 0 & 0 & 1 & 1 & 1
\end{array}
\right), 
\\&
\left(
\begin{array}{@{}c@{}c@{}c@{}c@{}c@{}c@{}c@{}}
	&  &  &  & 1 &  & \\
	0 & 0 & 1 & 1 & 1 & 1 & 0
\end{array}
\right),
\left(
\begin{array}{@{}c@{}c@{}c@{}c@{}c@{}c@{}c@{}}
	&  &  &  & 1 &  & \\
	0 & 0 & 0 & 1 & 2 & 1 & 0
\end{array}
\right),
\left(
\begin{array}{@{}c@{}c@{}c@{}c@{}c@{}c@{}c@{}}
	&  &  &  & 1 &  & \\
	0 & 0 & 0 & 1 & 1 & 1 & 1
\end{array}
\right),
\left(
\begin{array}{@{}c@{}c@{}c@{}c@{}c@{}c@{}c@{}}
	&  &  &  & 1 &  & \\
	0 & 0 & 1 & 1 & 2 & 1 & 0
\end{array}
\right),
\left(
\begin{array}{@{}c@{}c@{}c@{}c@{}c@{}c@{}c@{}}
	&  &  &  & 1 &  & \\
	0 & 0 & 1 & 1 & 1 & 1 & 1
\end{array}
\right),
\left(
\begin{array}{@{}c@{}c@{}c@{}c@{}c@{}c@{}c@{}}
	&  &  &  & 1 &  & \\
	0 & 0 & 0 & 1 & 2 & 1 & 1
\end{array}
\right),
\left(
\begin{array}{@{}c@{}c@{}c@{}c@{}c@{}c@{}c@{}}
	&  &  &  & 1 &  & \\
	0 & 0 & 1 & 2 & 2 & 1 & 0
\end{array}
\right), 
\\&
\left(
\begin{array}{@{}c@{}c@{}c@{}c@{}c@{}c@{}c@{}}
	&  &  &  & 1 &  & \\
	0 & 0 & 1 & 1 & 2 & 1 & 1
\end{array}
\right),
\left(
\begin{array}{@{}c@{}c@{}c@{}c@{}c@{}c@{}c@{}}
	&  &  &  & 1 &  & \\
	0 & 0 & 0 & 1 & 2 & 2 & 1
\end{array}
\right),
\left(
\begin{array}{@{}c@{}c@{}c@{}c@{}c@{}c@{}c@{}}
	&  &  &  & 1 &  & \\
	0 & 0 & 1 & 2 & 2 & 1 & 1
\end{array}
\right),
\left(
\begin{array}{@{}c@{}c@{}c@{}c@{}c@{}c@{}c@{}}
	&  &  &  & 1 &  & \\
	0 & 0 & 1 & 1 & 2 & 2 & 1
\end{array}
\right),
\left(
\begin{array}{@{}c@{}c@{}c@{}c@{}c@{}c@{}c@{}}
	&  &  &  & 1 &  & \\
	0 & 0 & 1 & 2 & 2 & 2 & 1
\end{array}
\right),
\left(
\begin{array}{@{}c@{}c@{}c@{}c@{}c@{}c@{}c@{}}
	&  &  &  & 1 &  & \\
	0 & 0 & 1 & 2 & 3 & 2 & 1
\end{array}
\right),
\left(
\begin{array}{@{}c@{}c@{}c@{}c@{}c@{}c@{}c@{}}
	&  &  &  & 1 &  & \\
	0 & 1 & 1 & 1 & 1 & 0 & 0
\end{array}
\right),
\\&
\left(
\begin{array}{@{}c@{}c@{}c@{}c@{}c@{}c@{}c@{}}
	&  &  &  & 1 &  & \\
	0 & 1 & 1 & 1 & 1 & 1 & 0
\end{array}
\right),
\left(
\begin{array}{@{}c@{}c@{}c@{}c@{}c@{}c@{}c@{}}
	&  &  &  & 1 &  & \\
	0 & 1 & 1 & 1 & 2 & 1 & 0
\end{array}
\right),
\left(
\begin{array}{@{}c@{}c@{}c@{}c@{}c@{}c@{}c@{}}
	&  &  &  & 1 &  & \\
	0 & 1 & 1 & 2 & 2 & 1 & 0
\end{array}
\right),
\left(
\begin{array}{@{}c@{}c@{}c@{}c@{}c@{}c@{}c@{}}
	&  &  &  & 1 &  & \\
	0 & 1 & 2 & 2 & 2 & 1 & 0
\end{array}
\right),
\left(
\begin{array}{@{}c@{}c@{}c@{}c@{}c@{}c@{}c@{}}
	&  &  &  & 1 &  & \\
	0 & 1 & 1 & 1 & 1 & 1 & 1
\end{array}
\right),
\left(
\begin{array}{@{}c@{}c@{}c@{}c@{}c@{}c@{}c@{}}
	&  &  &  & 1 &  & \\
	0 & 1 & 1 & 1 & 2 & 1 & 1
\end{array}
\right),
\left(
\begin{array}{@{}c@{}c@{}c@{}c@{}c@{}c@{}c@{}}
	&  &  &  & 1 &  & \\
	0 & 1 & 1 & 2 & 2 & 1 & 1
\end{array}
\right), 
\end{align*}
\begin{align*}
&
\left(
\begin{array}{@{}c@{}c@{}c@{}c@{}c@{}c@{}c@{}}
	&  &  &  & 1 &  & \\
	0 & 1 & 2 & 2 & 2 & 1 & 1
\end{array}
\right),
\left(
\begin{array}{@{}c@{}c@{}c@{}c@{}c@{}c@{}c@{}}
	&  &  &  & 1 &  & \\
	0 & 1 & 1 & 1 & 2 & 2 & 1
\end{array}
\right),
\left(
\begin{array}{@{}c@{}c@{}c@{}c@{}c@{}c@{}c@{}}
	&  &  &  & 1 &  & \\
	0 & 1 & 1 & 2 & 2 & 2 & 1
\end{array}
\right),
\left(
\begin{array}{@{}c@{}c@{}c@{}c@{}c@{}c@{}c@{}}
	&  &  &  & 1 &  & \\
	0 & 1 & 2 & 2 & 2 & 2 & 1
\end{array}
\right),
\left(
\begin{array}{@{}c@{}c@{}c@{}c@{}c@{}c@{}c@{}}
	&  &  &  & 1 &  & \\
	0 & 1 & 1 & 2 & 3 & 2 & 1
\end{array}
\right),
\left(
\begin{array}{@{}c@{}c@{}c@{}c@{}c@{}c@{}c@{}}
	&  &  &  & 1 &  & \\
	0 & 1 & 2 & 2 & 3 & 2 & 1
\end{array}
\right),
\left(
\begin{array}{@{}c@{}c@{}c@{}c@{}c@{}c@{}c@{}}
	&  &  &  & 1 &  & \\
	0 & 1 & 2 & 3 & 3 & 2 & 1
\end{array}
\right), \\&
\left(
\begin{array}{@{}c@{}c@{}c@{}c@{}c@{}c@{}c@{}}
	&  &  &  & 1 &  & \\
	1 & 1 & 1 & 1 & 1 & 0 & 0
\end{array}
\right),
\left(
\begin{array}{@{}c@{}c@{}c@{}c@{}c@{}c@{}c@{}}
	&  &  &  & 1 &  & \\
	1 & 1 & 1 & 1 & 1 & 1 & 0
\end{array}
\right),
\left(
\begin{array}{@{}c@{}c@{}c@{}c@{}c@{}c@{}c@{}}
	&  &  &  & 1 &  & \\
	1 & 1 & 1 & 1 & 2 & 1 & 0
\end{array}
\right),
\left(
\begin{array}{@{}c@{}c@{}c@{}c@{}c@{}c@{}c@{}}
	&  &  &  & 1 &  & \\
	1 & 1 & 1 & 1 & 1 & 1 & 1
\end{array}
\right),
\left(
\begin{array}{@{}c@{}c@{}c@{}c@{}c@{}c@{}c@{}}
	&  &  &  & 1 &  & \\
	1 & 1 & 1 & 2 & 2 & 1 & 0
\end{array}
\right),
\left(
\begin{array}{@{}c@{}c@{}c@{}c@{}c@{}c@{}c@{}}
	&  &  &  & 1 &  & \\
	1 & 1 & 1 & 1 & 2 & 1 & 1
\end{array}
\right),
\left(
\begin{array}{@{}c@{}c@{}c@{}c@{}c@{}c@{}c@{}}
	&  &  &  & 1 &  & \\
	1 & 1 & 2 & 2 & 2 & 1 & 0
\end{array}
\right), \\&
\left(
\begin{array}{@{}c@{}c@{}c@{}c@{}c@{}c@{}c@{}}
	&  &  &  & 1 &  & \\
	1 & 1 & 1 & 2 & 2 & 1 & 1
\end{array}
\right),
\left(
\begin{array}{@{}c@{}c@{}c@{}c@{}c@{}c@{}c@{}}
	&  &  &  & 1 &  & \\
	1 & 1 & 1 & 1 & 2 & 2 & 1
\end{array}
\right),
\left(
\begin{array}{@{}c@{}c@{}c@{}c@{}c@{}c@{}c@{}}
	&  &  &  & 1 &  & \\
	1 & 2 & 2 & 2 & 2 & 1 & 0
\end{array}
\right),
\left(
\begin{array}{@{}c@{}c@{}c@{}c@{}c@{}c@{}c@{}}
	&  &  &  & 1 &  & \\
	1 & 1 & 2 & 2 & 2 & 1 & 1
\end{array}
\right),
\left(
\begin{array}{@{}c@{}c@{}c@{}c@{}c@{}c@{}c@{}}
	&  &  &  & 1 &  & \\
	1 & 1 & 1 & 2 & 2 & 2 & 1
\end{array}
\right),
\left(
\begin{array}{@{}c@{}c@{}c@{}c@{}c@{}c@{}c@{}}
	&  &  &  & 1 &  & \\
	1 & 2 & 2 & 2 & 2 & 1 & 1
\end{array}
\right),
\left(
\begin{array}{@{}c@{}c@{}c@{}c@{}c@{}c@{}c@{}}
	&  &  &  & 1 &  & \\
	1 & 1 & 2 & 2 & 2 & 2 & 1
\end{array}
\right), \\&
\left(
\begin{array}{@{}c@{}c@{}c@{}c@{}c@{}c@{}c@{}}
	&  &  &  & 1 &  & \\
	1 & 1 & 1 & 2 & 3 & 2 & 1
\end{array}
\right),
\left(
\begin{array}{@{}c@{}c@{}c@{}c@{}c@{}c@{}c@{}}
	&  &  &  & 1 &  & \\
	1 & 2 & 2 & 2 & 2 & 2 & 1
\end{array}
\right),
\left(
\begin{array}{@{}c@{}c@{}c@{}c@{}c@{}c@{}c@{}}
	&  &  &  & 1 &  & \\
	1 & 1 & 2 & 2 & 3 & 2 & 1
\end{array}
\right),
\left(
\begin{array}{@{}c@{}c@{}c@{}c@{}c@{}c@{}c@{}}
	&  &  &  & 1 &  & \\
	1 & 2 & 2 & 2 & 3 & 2 & 1
\end{array}
\right),
\left(
\begin{array}{@{}c@{}c@{}c@{}c@{}c@{}c@{}c@{}}
	&  &  &  & 1 &  & \\
	1 & 1 & 2 & 3 & 3 & 2 & 1
\end{array}
\right),
\left(
\begin{array}{@{}c@{}c@{}c@{}c@{}c@{}c@{}c@{}}
	&  &  &  & 1 &  & \\
	1 & 2 & 2 & 3 & 3 & 2 & 1
\end{array}
\right),
\left(
\begin{array}{@{}c@{}c@{}c@{}c@{}c@{}c@{}c@{}}
	&  &  &  & 1 &  & \\
	1 & 2 & 3 & 3 & 3 & 2 & 1
\end{array}
\right).
\end{align*}
\end{proposition}

\begin{proof}
The number of roots in the $\Phi$-string of $\lambda$ will be: the number of positive roots of the root system $\Sigma_{\lambda, \Phi} \cong E_{n+1}$, minus the number of positive roots of $\Sigma_\Phi \cong A_n$, and minus the number of (positive) roots of $\Sigma_{\lambda, \Phi}$ whose coefficient corresponding to $\lambda$ when expressed with respect to  $\Pi_{\lambda, \Phi}$ is greater than one, where $n \in \{5, 6, 7\}$. According to~\cite[p.~687--690]{K}, we have that $|I_\Phi^\lambda|$ is 20 if $n =5$, 35 if $n = 6$ and 56 if $n = 7$.

Now, one can check that the first 20 roots in the statement of this result are positive roots in an $E_6$ root system, the first 35 are positive roots in an $E_7$ root system, and all of them, 56, are positive roots of an $E_8$ root system (see~\cite[p.~687--690]{K}). They all belong to $I_\Phi^\lambda$ as their coefficient corresponding to $\lambda$ with respect to $\Pi_{\lambda, \Phi}$ is one.
\end{proof}

\subsection{$\Sigma_\Phi$ is a $D_n$ root subsystem of $\Sigma$}\label{subsection:d:e}

\begin{proposition}\label{proposition:d:e:string:type}
Let $ (\Sigma_\Phi, \Sigma_{\lambda, \Phi}) \cong (D_n, E_{n+1})$, with $n$ in $\{5, 6, 7\}$, and let 
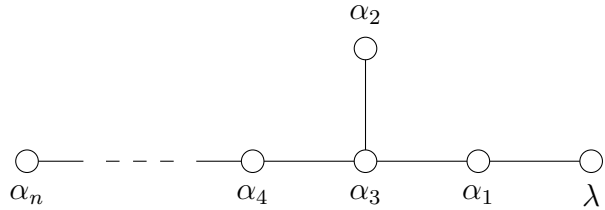
\begin{equation*}\label{string:d:e}
\begin{tikzpicture}[scale=1.5]
\draw (1, 0) circle (0.1);
\draw (3, 0) circle (0.1);
\draw (4, 0) circle (0.1);
\draw (5, 0) circle (0.1);
\draw (4, 1) circle (0.1);
\draw (6, 0) circle (0.1);
\draw (4., 0.1) -- (4., 0.9);
\draw (4., 1.3) node {$\alpha _2$};
\draw (3.1, 0.) -- (3.9, 0.);
\draw (4., -0.3) node {$\alpha _3$};
\draw (4.1, 0.) -- (4.9, 0.);
\draw (5., -0.3) node {$\alpha _1$};
\draw (5.1, 0.) -- (5.9, 0.);
\draw (6., -0.3) node {$\lambda$};
\draw (1.1, 0.) -- (1.5, 0.);
\draw (1.7, 0.) -- (1.82, 0.);
\draw (1.94, 0.) -- (2.06, 0.);
\draw (2.18, 0.) -- (2.3, 0.);
\draw (2.5, 0) -- (2.9, 0);
\draw (3., -0.3) node {$\alpha _4$};
\draw (1, -0.3) node {$\alpha _n$};
\end{tikzpicture}
\end{equation*}
be the Dynkin diagram of $\Sigma_{\lambda, \Phi}$. Then, the $\Phi$-string of $\lambda$ consists of the first 16 roots below if $n = 5$, the first 32 roots below if $n=6$, and all of them otherwise, and they all have the same length:
\begin{align*}
	&\left(
	\begin{array}{@{}c@{}c@{}c@{}c@{}c@{}c@{}c@{}}
		&  &  &  & 0 &  & \\
		0 & 0 & 0 & 0 & 0 & 0 & 1
	\end{array}
	\right),
	\left(
	\begin{array}{@{}c@{}c@{}c@{}c@{}c@{}c@{}c@{}}
		&  &  &  & 0 &  & \\
		0 & 0 & 0 & 0 & 0 & 1 & 1
	\end{array}
	\right),
	\left(
	\begin{array}{@{}c@{}c@{}c@{}c@{}c@{}c@{}c@{}}
		&  &  &  & 0 &  & \\
		0 & 0 & 0 & 0 & 1 & 1 & 1
	\end{array}
	\right),
	\left(
	\begin{array}{@{}c@{}c@{}c@{}c@{}c@{}c@{}c@{}}
		&  &  &  & 0 &  & \\
		0 & 0 & 0 & 1 & 1 & 1 & 1
	\end{array}
	\right),
	\left(
	\begin{array}{@{}c@{}c@{}c@{}c@{}c@{}c@{}c@{}}
		&  &  &  & 1 &  & \\
		0 & 0 & 0 & 0 & 1 & 1 & 1
	\end{array}
	\right),
	\left(
	\begin{array}{@{}c@{}c@{}c@{}c@{}c@{}c@{}c@{}}
		&  &  &  & 0 &  & \\
		0 & 0 & 1 & 1 & 1 & 1 & 1
	\end{array}
	\right),
	\left(
	\begin{array}{@{}c@{}c@{}c@{}c@{}c@{}c@{}c@{}}
		&  &  &  & 1 &  & \\
		0 & 0 & 0 & 1 & 1 & 1 & 1
	\end{array}
	\right), \\&
	\left(
	\begin{array}{@{}c@{}c@{}c@{}c@{}c@{}c@{}c@{}}
		&  &  &  & 1 &  & \\
		0 & 0 & 1 & 1 & 1 & 1 & 1
	\end{array}
	\right),
	\left(
	\begin{array}{@{}c@{}c@{}c@{}c@{}c@{}c@{}c@{}}
		&  &  &  & 1 &  & \\
		0 & 0 & 0 & 1 & 2 & 1 & 1
	\end{array}
	\right),
	\left(
	\begin{array}{@{}c@{}c@{}c@{}c@{}c@{}c@{}c@{}}
		&  &  &  & 1 &  & \\
		0 & 0 & 1 & 1 & 2 & 1 & 1
	\end{array}
	\right),
	\left(
	\begin{array}{@{}c@{}c@{}c@{}c@{}c@{}c@{}c@{}}
		&  &  &  & 1 &  & \\
		0 & 0 & 0 & 1 & 2 & 2 & 1
	\end{array}
	\right),
	\left(
	\begin{array}{@{}c@{}c@{}c@{}c@{}c@{}c@{}c@{}}
		&  &  &  & 1 &  & \\
		0 & 0 & 1 & 2 & 2 & 1 & 1
	\end{array}
	\right),
	\left(
	\begin{array}{@{}c@{}c@{}c@{}c@{}c@{}c@{}c@{}}
		&  &  &  & 1 &  & \\
		0 & 0 & 1 & 1 & 2 & 2 & 1
	\end{array}
	\right),
	\left(
	\begin{array}{@{}c@{}c@{}c@{}c@{}c@{}c@{}c@{}}
		&  &  &  & 1 &  & \\
		0 & 0 & 1 & 2 & 2 & 2 & 1
	\end{array}
	\right), \\&
	\left(
	\begin{array}{@{}c@{}c@{}c@{}c@{}c@{}c@{}c@{}}
		&  &  &  & 1 &  & \\
		0 & 0 & 1 & 2 & 3 & 2 & 1
	\end{array}
	\right),
	\left(
	\begin{array}{@{}c@{}c@{}c@{}c@{}c@{}c@{}c@{}}
		&  &  &  & 2 &  & \\
		0 & 0 & 1 & 2 & 3 & 2 & 1
	\end{array}
	\right),
	\left(
	\begin{array}{@{}c@{}c@{}c@{}c@{}c@{}c@{}c@{}}
		&  &  &  & 0 &  & \\
		0 & 1 & 1 & 1 & 1 & 1 & 1
	\end{array}
	\right),
	\left(
	\begin{array}{@{}c@{}c@{}c@{}c@{}c@{}c@{}c@{}}
		&  &  &  & 1 &  & \\
		0 & 1 & 1 & 1 & 1 & 1 & 1
	\end{array}
	\right),
	\left(
	\begin{array}{@{}c@{}c@{}c@{}c@{}c@{}c@{}c@{}}
		&  &  &  & 1 &  & \\
		0 & 1 & 1 & 1 & 2 & 1 & 1
	\end{array}
	\right),
	\left(
	\begin{array}{@{}c@{}c@{}c@{}c@{}c@{}c@{}c@{}}
		&  &  &  & 1 &  & \\
		0 & 1 & 1 & 2 & 2 & 1 & 1
	\end{array}
	\right),
	\left(
	\begin{array}{@{}c@{}c@{}c@{}c@{}c@{}c@{}c@{}}
		&  &  &  & 1 &  & \\
		0 & 1 & 1 & 1 & 2 & 2 & 1
	\end{array}
	\right),
\\&
\left(
	\begin{array}{@{}c@{}c@{}c@{}c@{}c@{}c@{}c@{}}
		&  &  &  & 1 &  & \\
		0 & 1 & 2 & 2 & 2 & 1 & 1
	\end{array}
	\right),
	\left(
	\begin{array}{@{}c@{}c@{}c@{}c@{}c@{}c@{}c@{}}
		&  &  &  & 1 &  & \\
		0 & 1 & 1 & 2 & 2 & 2 & 1
	\end{array}
	\right),
	\left(
	\begin{array}{@{}c@{}c@{}c@{}c@{}c@{}c@{}c@{}}
		&  &  &  & 1 &  & \\
		0 & 1 & 2 & 2 & 2 & 2 & 1
	\end{array}
	\right),
	\left(
	\begin{array}{@{}c@{}c@{}c@{}c@{}c@{}c@{}c@{}}
		&  &  &  & 1 &  & \\
		0 & 1 & 1 & 2 & 3 & 2 & 1
	\end{array}
	\right),
	\left(
	\begin{array}{@{}c@{}c@{}c@{}c@{}c@{}c@{}c@{}}
		&  &  &  & 1 &  & \\
		0 & 1 & 2 & 2 & 3 & 2 & 1
	\end{array}
	\right),
	\left(
	\begin{array}{@{}c@{}c@{}c@{}c@{}c@{}c@{}c@{}}
		&  &  &  & 2 &  & \\
		0 & 1 & 1 & 2 & 3 & 2 & 1
	\end{array}
	\right),
	\left(
	\begin{array}{@{}c@{}c@{}c@{}c@{}c@{}c@{}c@{}}
		&  &  &  & 1 &  & \\
		0 & 1 & 2 & 3 & 3 & 2 & 1
	\end{array}
	\right),
\end{align*}
\begin{align*}
&
	\left(
	\begin{array}{@{}c@{}c@{}c@{}c@{}c@{}c@{}c@{}}
		&  &  &  & 2 &  & \\
		0 & 1 & 2 & 2 & 3 & 2 & 1
	\end{array}
	\right), 
	\left(
	\begin{array}{@{}c@{}c@{}c@{}c@{}c@{}c@{}c@{}}
		&  &  &  & 2 &  & \\
		0 & 1 & 2 & 3 & 3 & 2 & 1
	\end{array}
	\right),
	\left(
	\begin{array}{@{}c@{}c@{}c@{}c@{}c@{}c@{}c@{}}
		&  &  &  & 2 &  & \\
		0 & 1 & 2 & 3 & 4 & 2 & 1
	\end{array}
	\right),
	\left(
	\begin{array}{@{}c@{}c@{}c@{}c@{}c@{}c@{}c@{}}
		&  &  &  & 2 &  & \\
		0 & 1 & 2 & 3 & 4 & 3 & 1
	\end{array}
	\right),
	\left(
	\begin{array}{@{}c@{}c@{}c@{}c@{}c@{}c@{}c@{}}
		&  &  &  & 0 &  & \\
		1 & 1 & 1 & 1 & 1 & 1 & 1
	\end{array}
	\right),
	\left(
	\begin{array}{@{}c@{}c@{}c@{}c@{}c@{}c@{}c@{}}
		&  &  &  & 1 &  & \\
		1 & 1 & 1 & 1 & 1 & 1 & 1
	\end{array}
	\right),
	\left(
	\begin{array}{@{}c@{}c@{}c@{}c@{}c@{}c@{}c@{}}
		&  &  &  & 1 &  & \\
		1 & 1 & 1 & 1 & 2 & 1 & 1
	\end{array}
	\right),\\&
	\left(
	\begin{array}{@{}c@{}c@{}c@{}c@{}c@{}c@{}c@{}}
		&  &  &  & 1 &  & \\
		1 & 1 & 1 & 2 & 2 & 1 & 1
	\end{array}
	\right),
	\left(
	\begin{array}{@{}c@{}c@{}c@{}c@{}c@{}c@{}c@{}}
		&  &  &  & 1 &  & \\
		1 & 1 & 1 & 1 & 2 & 2 & 1
	\end{array}
	\right),
	\left(
	\begin{array}{@{}c@{}c@{}c@{}c@{}c@{}c@{}c@{}}
		&  &  &  & 1 &  & \\
		1 & 1 & 2 & 2 & 2 & 1 & 1
	\end{array}
	\right),
	\left(
	\begin{array}{@{}c@{}c@{}c@{}c@{}c@{}c@{}c@{}}
		&  &  &  & 1 &  & \\
		1 & 1 & 1 & 2 & 2 & 2 & 1
	\end{array}
	\right),
	\left(
	\begin{array}{@{}c@{}c@{}c@{}c@{}c@{}c@{}c@{}}
		&  &  &  & 1 &  & \\
		1 & 2 & 2 & 2 & 2 & 1 & 1
	\end{array}
	\right),
	\left(
	\begin{array}{@{}c@{}c@{}c@{}c@{}c@{}c@{}c@{}}
		&  &  &  & 1 &  & \\
		1 & 1 & 2 & 2 & 2 & 2 & 1
	\end{array}
	\right),
	\left(
	\begin{array}{@{}c@{}c@{}c@{}c@{}c@{}c@{}c@{}}
		&  &  &  & 1 &  & \\
		1 & 1 & 1 & 2 & 3 & 2 & 1
	\end{array}
	\right), \\&
	\left(
	\begin{array}{@{}c@{}c@{}c@{}c@{}c@{}c@{}c@{}}
		&  &  &  & 1 &  & \\
		1 & 2 & 2 & 2 & 2 & 2 & 1
	\end{array}
	\right),
	\left(
	\begin{array}{@{}c@{}c@{}c@{}c@{}c@{}c@{}c@{}}
		&  &  &  & 1 &  & \\
		1 & 1 & 2 & 2 & 3 & 2 & 1
	\end{array}
	\right),
	\left(
	\begin{array}{@{}c@{}c@{}c@{}c@{}c@{}c@{}c@{}}
		&  &  &  & 2 &  & \\
		1 & 1 & 1 & 2 & 3 & 2 & 1
	\end{array}
	\right),
	\left(
	\begin{array}{@{}c@{}c@{}c@{}c@{}c@{}c@{}c@{}}
		&  &  &  & 1 &  & \\
		1 & 2 & 2 & 2 & 3 & 2 & 1
	\end{array}
	\right),
	\left(
	\begin{array}{@{}c@{}c@{}c@{}c@{}c@{}c@{}c@{}}
		&  &  &  & 1 &  & \\
		1 & 1 & 2 & 3 & 3 & 2 & 1
	\end{array}
	\right),
	\left(
	\begin{array}{@{}c@{}c@{}c@{}c@{}c@{}c@{}c@{}}
		&  &  &  & 2 &  & \\
		1 & 1 & 2 & 2 & 3 & 2 & 1
	\end{array}
	\right),
	\left(
	\begin{array}{@{}c@{}c@{}c@{}c@{}c@{}c@{}c@{}}
		&  &  &  & 1 &  & \\
		1 & 2 & 2 & 3 & 3 & 2 & 1
	\end{array}
	\right), \\&
	\left(
	\begin{array}{@{}c@{}c@{}c@{}c@{}c@{}c@{}c@{}}
		&  &  &  & 2 &  & \\
		1 & 2 & 2 & 2 & 3 & 2 & 1
	\end{array}
	\right),
	\left(
	\begin{array}{@{}c@{}c@{}c@{}c@{}c@{}c@{}c@{}}
		&  &  &  & 2 &  & \\
		1 & 1 & 2 & 3 & 3 & 2 & 1
	\end{array}
	\right),
	\left(
	\begin{array}{@{}c@{}c@{}c@{}c@{}c@{}c@{}c@{}}
		&  &  &  & 1 &  & \\
		1 & 2 & 3 & 3 & 3 & 2 & 1
	\end{array}
	\right),
	\left(
	\begin{array}{@{}c@{}c@{}c@{}c@{}c@{}c@{}c@{}}
		&  &  &  & 2 &  & \\
		1 & 2 & 2 & 3 & 3 & 2 & 1
	\end{array}
	\right),
	\left(
	\begin{array}{@{}c@{}c@{}c@{}c@{}c@{}c@{}c@{}}
		&  &  &  & 2 &  & \\
		1 & 1 & 2 & 3 & 4 & 2 & 1
	\end{array}
	\right),
	\left(
	\begin{array}{@{}c@{}c@{}c@{}c@{}c@{}c@{}c@{}}
		&  &  &  & 2 &  & \\
		1 & 2 & 3 & 3 & 3 & 2 & 1
	\end{array}
	\right),
	\left(
	\begin{array}{@{}c@{}c@{}c@{}c@{}c@{}c@{}c@{}}
		&  &  &  & 2 &  & \\
		1 & 2 & 2 & 3 & 4 & 2 & 1
	\end{array}
	\right), \\&
	\left(
	\begin{array}{@{}c@{}c@{}c@{}c@{}c@{}c@{}c@{}}
		&  &  &  & 2 &  & \\
		1 & 1 & 2 & 3 & 4 & 3 & 1
	\end{array}
	\right),
	\left(
	\begin{array}{@{}c@{}c@{}c@{}c@{}c@{}c@{}c@{}}
		&  &  &  & 2 &  & \\
		1 & 2 & 3 & 3 & 4 & 2 & 1
	\end{array}
	\right),
	\left(
	\begin{array}{@{}c@{}c@{}c@{}c@{}c@{}c@{}c@{}}
		&  &  &  & 2 &  & \\
		1 & 2 & 2 & 3 & 4 & 3 & 1
	\end{array}
	\right),
	\left(
	\begin{array}{@{}c@{}c@{}c@{}c@{}c@{}c@{}c@{}}
		&  &  &  & 2 &  & \\
		1 & 2 & 3 & 4 & 4 & 2 & 1
	\end{array}
	\right),
	\left(
	\begin{array}{@{}c@{}c@{}c@{}c@{}c@{}c@{}c@{}}
		&  &  &  & 2 &  & \\
		1 & 2 & 3 & 3 & 4 & 3 & 1
	\end{array}
	\right),
	\left(
	\begin{array}{@{}c@{}c@{}c@{}c@{}c@{}c@{}c@{}}
		&  &  &  & 2 &  & \\
		1 & 2 & 3 & 4 & 4 & 3 & 1
	\end{array}
	\right),
	\left(
	\begin{array}{@{}c@{}c@{}c@{}c@{}c@{}c@{}c@{}}
		&  &  &  & 2 &  & \\
		1 & 2 & 3 & 4 & 5 & 3 & 1
	\end{array}
	\right), \\&
	\left(
	\begin{array}{@{}c@{}c@{}c@{}c@{}c@{}c@{}c@{}}
		&  &  &  & 3 &  & \\
		1 & 2 & 3 & 4 & 5 & 3 & 1
	\end{array}
	\right).
\end{align*}

\end{proposition}
\begin{proof}
The number of roots in the $\Phi$-string of $\lambda$ will be: the number of positive roots of the root system $\Sigma_{\lambda, \Phi} \cong E_{n+1}$, minus the number of positive roots of $\Sigma_\Phi \cong D_n$, and minus the number of (positive) roots in $\Sigma_{\lambda, \Phi}$ whose coefficient corresponding to $\lambda$ when expressed with respect to  $\Pi_{\lambda, \Phi}$ is greater than one, where $n \in \{5, 6, 7\}$. According to~\cite[p.~687--690]{K}, we have that $|I_\Phi^\lambda|$ is 16 if $n =5$, 32 if $n = 6$ and 64 if $n = 7$.

Now, one can check that the first 16 roots in the statement of this result are positive roots in an $E_6$ root system, the first 32 are positive roots in an $E_7$ root system, and all of them, 64, are positive roots in an $E_8$ root system (see~\cite[p.~687--690]{K}). They all belong to $I_\Phi^\lambda$ as their coefficient corresponding to $\lambda$ with respect to $\Pi_{\lambda, \Phi}$ is one.
\end{proof}

\subsection{$\Sigma_\Phi$ is an $E_n$ root subsystem of $\Sigma$}\label{subsection:e:e}

\begin{proposition}\label{proposition:e6:e7:string:type}
Let $ (\Sigma_\Phi, \Sigma_{\lambda, \Phi}) \cong (E_6, E_7)$ and let 
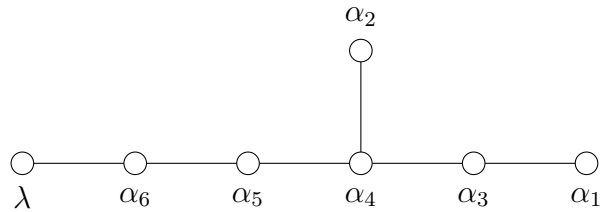
\begin{equation*}
\begin{tikzpicture}[scale=1.5]
\draw (1, 0) circle (0.1);
\draw (2, 0) circle (0.1);
\draw (3, 0) circle (0.1);
\draw (4, 0) circle (0.1);
\draw (5, 0) circle (0.1);
\draw (6, 0) circle (0.1);
\draw (4, 1) circle (0.1);
\draw (1.1, 0.) -- (1.9, 0.);
\draw (2., -0.3) node {$\alpha _6$};
\draw (2.1, 0.) -- (2.9, 0.);
\draw (3., -0.3) node {$\alpha _5$};
\draw (3.1, 0.) -- (3.9, 0.);
\draw (4., -0.3) node {$\alpha _4$};
\draw (4.1, 0.) -- (4.9, 0.);
\draw (5., -0.3) node {$\alpha _3$};
\draw (5.1, 0.) -- (5.9, 0.);
\draw (6., -0.3) node {$\alpha _1$};
\draw (4., 0.1) -- (4., 0.9);
\draw (4., 1.3) node {$\alpha _2$};
\draw (1, -0.3) node {$\lambda$};
\end{tikzpicture}
\end{equation*}
be the Dynkin diagram of $\Sigma_{\lambda, \Phi}$. Then, the $\Phi$-string of $\lambda$ consists of the following roots, and they all have the same length:

\begin{align*}
&\left(
\begin{array}{@{}c@{}c@{}c@{}c@{}c@{}c@{}c@{}}
	&  &  &  & 0 &  & \\
	0 & 1 & 0 & 0 & 0 & 0 & 0
\end{array}
\right),
\left(
\begin{array}{@{}c@{}c@{}c@{}c@{}c@{}c@{}c@{}}
	&  &  &  & 0 &  & \\
	0 & 1 & 1 & 0 & 0 & 0 & 0
\end{array}
\right),
\left(
\begin{array}{@{}c@{}c@{}c@{}c@{}c@{}c@{}c@{}}
	&  &  &  & 0 &  & \\
	0 & 1 & 1 & 1 & 0 & 0 & 0
\end{array}
\right),
\left(
\begin{array}{@{}c@{}c@{}c@{}c@{}c@{}c@{}c@{}}
	&  &  &  & 0 &  & \\
	0 & 1 & 1 & 1 & 1 & 0 & 0
\end{array}
\right),
\left(
\begin{array}{@{}c@{}c@{}c@{}c@{}c@{}c@{}c@{}}
	&  &  &  & 0 &  & \\
	0 & 1 & 1 & 1 & 1 & 1 & 0
\end{array}
\right),
\left(
\begin{array}{@{}c@{}c@{}c@{}c@{}c@{}c@{}c@{}}
	&  &  &  & 1 &  & \\
	0 & 1 & 1 & 1 & 1 & 0 & 0
\end{array}
\right),
\left(
\begin{array}{@{}c@{}c@{}c@{}c@{}c@{}c@{}c@{}}
	&  &  &  & 1 &  & \\
	0 & 1 & 1 & 1 & 1 & 1 & 0
\end{array}
\right), \\&
\left(
\begin{array}{@{}c@{}c@{}c@{}c@{}c@{}c@{}c@{}}
	&  &  &  & 0 &  & \\
	0 & 1 & 1 & 1 & 1 & 1 & 1
\end{array}
\right),
\left(
\begin{array}{@{}c@{}c@{}c@{}c@{}c@{}c@{}c@{}}
	&  &  &  & 1 &  & \\
	0 & 1 & 1 & 1 & 2 & 1 & 0
\end{array}
\right),
\left(
\begin{array}{@{}c@{}c@{}c@{}c@{}c@{}c@{}c@{}}
	&  &  &  & 1 &  & \\
	0 & 1 & 1 & 1 & 1 & 1 & 1
\end{array}
\right),
\left(
\begin{array}{@{}c@{}c@{}c@{}c@{}c@{}c@{}c@{}}
	&  &  &  & 1 &  & \\
	0 & 1 & 1 & 2 & 2 & 1 & 0
\end{array}
\right),
\left(
\begin{array}{@{}c@{}c@{}c@{}c@{}c@{}c@{}c@{}}
	&  &  &  & 1 &  & \\
	0 & 1 & 1 & 1 & 2 & 1 & 1
\end{array}
\right),
\left(
\begin{array}{@{}c@{}c@{}c@{}c@{}c@{}c@{}c@{}}
	&  &  &  & 1 &  & \\
	0 & 1 & 2 & 2 & 2 & 1 & 0
\end{array}
\right),
\left(
\begin{array}{@{}c@{}c@{}c@{}c@{}c@{}c@{}c@{}}
	&  &  &  & 1 &  & \\
	0 & 1 & 1 & 2 & 2 & 1 & 1
\end{array}
\right), 
\end{align*}
\begin{align*}
&
\left(
\begin{array}{@{}c@{}c@{}c@{}c@{}c@{}c@{}c@{}}
	&  &  &  & 1 &  & \\
	0 & 1 & 1 & 1 & 2 & 2 & 1
\end{array}
\right),
\left(
\begin{array}{@{}c@{}c@{}c@{}c@{}c@{}c@{}c@{}}
	&  &  &  & 1 &  & \\
	0 & 1 & 2 & 2 & 2 & 1 & 1
\end{array}
\right),
\left(
\begin{array}{@{}c@{}c@{}c@{}c@{}c@{}c@{}c@{}}
	&  &  &  & 1 &  & \\
	0 & 1 & 1 & 2 & 2 & 2 & 1
\end{array}
\right),
\left(
\begin{array}{@{}c@{}c@{}c@{}c@{}c@{}c@{}c@{}}
	&  &  &  & 1 &  & \\
	0 & 1 & 2 & 2 & 2 & 2 & 1
\end{array}
\right),
\left(
\begin{array}{@{}c@{}c@{}c@{}c@{}c@{}c@{}c@{}}
	&  &  &  & 1 &  & \\
	0 & 1 & 1 & 2 & 3 & 2 & 1
\end{array}
\right),
\left(
\begin{array}{@{}c@{}c@{}c@{}c@{}c@{}c@{}c@{}}
	&  &  &  & 1 &  & \\
	0 & 1 & 2 & 2 & 3 & 2 & 1
\end{array}
\right),
\left(
\begin{array}{@{}c@{}c@{}c@{}c@{}c@{}c@{}c@{}}
	&  &  &  & 2 &  & \\
	0 & 1 & 1 & 2 & 3 & 2 & 1
\end{array}
\right), \\&
\left(
\begin{array}{@{}c@{}c@{}c@{}c@{}c@{}c@{}c@{}}
	&  &  &  & 1 &  & \\
	0 & 1 & 2 & 3 & 3 & 2 & 1
\end{array}
\right),
\left(
\begin{array}{@{}c@{}c@{}c@{}c@{}c@{}c@{}c@{}}
	&  &  &  & 2 &  & \\
	0 & 1 & 2 & 2 & 3 & 2 & 1
\end{array}
\right),
\left(
\begin{array}{@{}c@{}c@{}c@{}c@{}c@{}c@{}c@{}}
	&  &  &  & 2 &  & \\
	0 & 1 & 2 & 3 & 3 & 2 & 1
\end{array}
\right),
\left(
\begin{array}{@{}c@{}c@{}c@{}c@{}c@{}c@{}c@{}}
	&  &  &  & 2 &  & \\
	0 & 1 & 2 & 3 & 4 & 2 & 1
\end{array}
\right),
\left(
\begin{array}{@{}c@{}c@{}c@{}c@{}c@{}c@{}c@{}}
	&  &  &  & 2 &  & \\
	0 & 1 & 2 & 3 & 4 & 3 & 1
\end{array}
\right),
\left(
\begin{array}{@{}c@{}c@{}c@{}c@{}c@{}c@{}c@{}}
	&  &  &  & 2 &  & \\
	0 & 1 & 2 & 3 & 4 & 3 & 2
\end{array}
\right).
\end{align*}

\end{proposition}

\begin{proof}
The number of roots in the $\Phi$-string of $\lambda$ will be: the number of positive roots of the root system $\Sigma_{\lambda, \Phi} \cong E_7$, minus the number of positive roots of $\Sigma_\Phi \cong E_6$, and minus the number of (positive) roots of $\Sigma_{\lambda, \Phi}$ whose coefficient corresponding to $\lambda$ when expressed with respect to $\Pi_{\lambda, \Phi}$ is greater than one. According to~\cite[p.~687--690]{K}, we have $|I_\Phi^\lambda| = 63 - 36 =27$.

Now, on the one hand, one can check that all the roots in the statement of this result are positive roots in an $E_7$ root system (see~\cite[p.~687--688]{K}). On the other hand, they all belong to $I_\Phi^\lambda$ as their coefficient corresponding to $\lambda$ with respect to $\Pi_{\lambda, \Phi}$ is one. 
\end{proof}

\begin{proposition}\label{proposition:e7:e8:string:type}
Let $ (\Sigma_\Phi, \Sigma_{\lambda, \Phi}) \cong (E_7, E_8)$ and let 
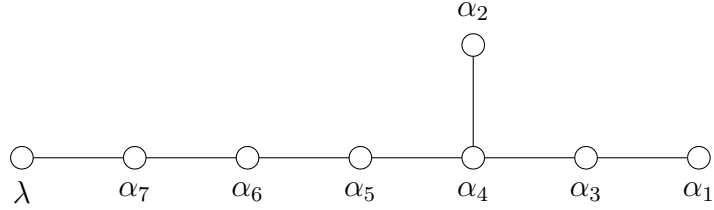
\begin{equation*}\label{string:e7:e8}
\begin{tikzpicture}[scale=1.5]
\draw (0, 0) circle (0.1);
\draw (1, 0) circle (0.1);
\draw (2, 0) circle (0.1);
\draw (3, 0) circle (0.1);
\draw (4, 0) circle (0.1);
\draw (5, 0) circle (0.1);
\draw (6, 0) circle (0.1);
\draw (4, 1) circle (0.1);
\draw (0.1, 0.) -- (0.9, 0.);
\draw (1., -0.3) node {$\alpha _7$};
\draw (1.1, 0.) -- (1.9, 0.);
\draw (2., -0.3) node {$\alpha _6$};
\draw (2.1, 0.) -- (2.9, 0.);
\draw (3., -0.3) node {$\alpha _5$};
\draw (3.1, 0.) -- (3.9, 0.);
\draw (4., -0.3) node {$\alpha _4$};
\draw (4.1, 0.) -- (4.9, 0.);
\draw (5., -0.3) node {$\alpha _3$};
\draw (5.1, 0.) -- (5.9, 0.);
\draw (6., -0.3) node {$\alpha _1$};
\draw (4., 0.1) -- (4., 0.9);
\draw (4., 1.3) node {$\alpha _2$};
\draw (0, -0.3) node {$\lambda$};
\end{tikzpicture}
\end{equation*}
be the Dynkin diagram of $\Sigma_{\lambda, \Phi}$. Then, the $\Phi$-string of $\lambda$ consists of the following roots, and they all have the same length:
\begin{align*}
&\left(
\begin{array}{@{}c@{}c@{}c@{}c@{}c@{}c@{}c@{}}
	&  &  &  & 0 &  & \\
	1 & 0 & 0 & 0 & 0 & 0 & 0
\end{array}
\right),
\left(
\begin{array}{@{}c@{}c@{}c@{}c@{}c@{}c@{}c@{}}
	&  &  &  & 0 &  & \\
	1 & 1 & 0 & 0 & 0 & 0 & 0
\end{array}
\right),
\left(
\begin{array}{@{}c@{}c@{}c@{}c@{}c@{}c@{}c@{}}
	&  &  &  & 0 &  & \\
	1 & 1 & 1 & 0 & 0 & 0 & 0
\end{array}
\right),
\left(
\begin{array}{@{}c@{}c@{}c@{}c@{}c@{}c@{}c@{}}
	&  &  &  & 0 &  & \\
	1 & 1 & 1 & 1 & 0 & 0 & 0
\end{array}
\right),
\left(
\begin{array}{@{}c@{}c@{}c@{}c@{}c@{}c@{}c@{}}
	&  &  &  & 0 &  & \\
	1 & 1 & 1 & 1 & 1 & 0 & 0
\end{array}
\right),
\left(
\begin{array}{@{}c@{}c@{}c@{}c@{}c@{}c@{}c@{}}
	&  &  &  & 0 &  & \\
	1 & 1 & 1 & 1 & 1 & 1 & 0
\end{array}
\right),
\left(
\begin{array}{@{}c@{}c@{}c@{}c@{}c@{}c@{}c@{}}
	&  &  &  & 1 &  & \\
	1 & 1 & 1 & 1 & 1 & 0 & 0
\end{array}
\right), \\&
\left(
\begin{array}{@{}c@{}c@{}c@{}c@{}c@{}c@{}c@{}}
	&  &  &  & 1 &  & \\
	1 & 1 & 1 & 1 & 1 & 1 & 0
\end{array}
\right),
\left(
\begin{array}{@{}c@{}c@{}c@{}c@{}c@{}c@{}c@{}}
	&  &  &  & 0 &  & \\
	1 & 1 & 1 & 1 & 1 & 1 & 1
\end{array}
\right),
\left(
\begin{array}{@{}c@{}c@{}c@{}c@{}c@{}c@{}c@{}}
	&  &  &  & 1 &  & \\
	1 & 1 & 1 & 1 & 2 & 1 & 0
\end{array}
\right),
\left(
\begin{array}{@{}c@{}c@{}c@{}c@{}c@{}c@{}c@{}}
	&  &  &  & 1 &  & \\
	1 & 1 & 1 & 1 & 1 & 1 & 1
\end{array}
\right),
\left(
\begin{array}{@{}c@{}c@{}c@{}c@{}c@{}c@{}c@{}}
	&  &  &  & 1 &  & \\
	1 & 1 & 1 & 2 & 2 & 1 & 0
\end{array}
\right),
\left(
\begin{array}{@{}c@{}c@{}c@{}c@{}c@{}c@{}c@{}}
	&  &  &  & 1 &  & \\
	1 & 1 & 1 & 1 & 2 & 1 & 1
\end{array}
\right),
\left(
\begin{array}{@{}c@{}c@{}c@{}c@{}c@{}c@{}c@{}}
	&  &  &  & 1 &  & \\
	1 & 1 & 2 & 2 & 2 & 1 & 0
\end{array}
\right), \\&
\left(
\begin{array}{@{}c@{}c@{}c@{}c@{}c@{}c@{}c@{}}
	&  &  &  & 1 &  & \\
	1 & 1 & 1 & 2 & 2 & 1 & 1
\end{array}
\right),
\left(
\begin{array}{@{}c@{}c@{}c@{}c@{}c@{}c@{}c@{}}
	&  &  &  & 1 &  & \\
	1 & 1 & 1 & 1 & 2 & 2 & 1
\end{array}
\right),
\left(
\begin{array}{@{}c@{}c@{}c@{}c@{}c@{}c@{}c@{}}
	&  &  &  & 1 &  & \\
	1 & 2 & 2 & 2 & 2 & 1 & 0
\end{array}
\right),
\left(
\begin{array}{@{}c@{}c@{}c@{}c@{}c@{}c@{}c@{}}
	&  &  &  & 1 &  & \\
	1 & 1 & 2 & 2 & 2 & 1 & 1
\end{array}
\right),
\left(
\begin{array}{@{}c@{}c@{}c@{}c@{}c@{}c@{}c@{}}
	&  &  &  & 1 &  & \\
	1 & 1 & 1 & 2 & 2 & 2 & 1
\end{array}
\right),
\left(
\begin{array}{@{}c@{}c@{}c@{}c@{}c@{}c@{}c@{}}
	&  &  &  & 1 &  & \\
	1 & 2 & 2 & 2 & 2 & 1 & 1
\end{array}
\right),
\left(
\begin{array}{@{}c@{}c@{}c@{}c@{}c@{}c@{}c@{}}
	&  &  &  & 1 &  & \\
	1 & 1 & 2 & 2 & 2 & 2 & 1
\end{array}
\right), \\&
\left(
\begin{array}{@{}c@{}c@{}c@{}c@{}c@{}c@{}c@{}}
	&  &  &  & 1 &  & \\
	1 & 1 & 1 & 2 & 3 & 2 & 1
\end{array}
\right),
\left(
\begin{array}{@{}c@{}c@{}c@{}c@{}c@{}c@{}c@{}}
	&  &  &  & 1 &  & \\
	1 & 2 & 2 & 2 & 2 & 2 & 1
\end{array}
\right),
\left(
\begin{array}{@{}c@{}c@{}c@{}c@{}c@{}c@{}c@{}}
	&  &  &  & 1 &  & \\
	1 & 1 & 2 & 2 & 3 & 2 & 1
\end{array}
\right),
\left(
\begin{array}{@{}c@{}c@{}c@{}c@{}c@{}c@{}c@{}}
	&  &  &  & 2 &  & \\
	1 & 1 & 1 & 2 & 3 & 2 & 1
\end{array}
\right),
\left(
\begin{array}{@{}c@{}c@{}c@{}c@{}c@{}c@{}c@{}}
	&  &  &  & 1 &  & \\
	1 & 2 & 2 & 2 & 3 & 2 & 1
\end{array}
\right),
\left(
\begin{array}{@{}c@{}c@{}c@{}c@{}c@{}c@{}c@{}}
	&  &  &  & 1 &  & \\
	1 & 1 & 2 & 3 & 3 & 2 & 1
\end{array}
\right),
\left(
\begin{array}{@{}c@{}c@{}c@{}c@{}c@{}c@{}c@{}}
	&  &  &  & 2 &  & \\
	1 & 1 & 2 & 2 & 3 & 2 & 1
\end{array}
\right), 
\end{align*}
\begin{align*}
&
\left(
\begin{array}{@{}c@{}c@{}c@{}c@{}c@{}c@{}c@{}}
	&  &  &  & 1 &  & \\
	1 & 2 & 2 & 3 & 3 & 2 & 1
\end{array}
\right),
\left(
\begin{array}{@{}c@{}c@{}c@{}c@{}c@{}c@{}c@{}}
	&  &  &  & 2 &  & \\
	1 & 2 & 2 & 2 & 3 & 2 & 1
\end{array}
\right),
\left(
\begin{array}{@{}c@{}c@{}c@{}c@{}c@{}c@{}c@{}}
	&  &  &  & 2 &  & \\
	1 & 1 & 2 & 3 & 3 & 2 & 1
\end{array}
\right),
\left(
\begin{array}{@{}c@{}c@{}c@{}c@{}c@{}c@{}c@{}}
	&  &  &  & 1 &  & \\
	1 & 2 & 3 & 3 & 3 & 2 & 1
\end{array}
\right),
\left(
\begin{array}{@{}c@{}c@{}c@{}c@{}c@{}c@{}c@{}}
	&  &  &  & 2 &  & \\
	1 & 2 & 2 & 3 & 3 & 2 & 1
\end{array}
\right),
\left(
\begin{array}{@{}c@{}c@{}c@{}c@{}c@{}c@{}c@{}}
	&  &  &  & 2 &  & \\
	1 & 1 & 2 & 3 & 4 & 2 & 1
\end{array}
\right),
\left(
\begin{array}{@{}c@{}c@{}c@{}c@{}c@{}c@{}c@{}}
	&  &  &  & 2 &  & \\
	1 & 2 & 3 & 3 & 3 & 2 & 1
\end{array}
\right),\\&
\left(
\begin{array}{@{}c@{}c@{}c@{}c@{}c@{}c@{}c@{}}
	&  &  &  & 2 &  & \\
	1 & 2 & 2 & 3 & 4 & 2 & 1
\end{array}
\right),
\left(
\begin{array}{@{}c@{}c@{}c@{}c@{}c@{}c@{}c@{}}
	&  &  &  & 2 &  & \\
	1 & 1 & 2 & 3 & 4 & 3 & 1
\end{array}
\right),
\left(
\begin{array}{@{}c@{}c@{}c@{}c@{}c@{}c@{}c@{}}
	&  &  &  & 2 &  & \\
	1 & 2 & 3 & 3 & 4 & 2 & 1
\end{array}
\right),
\left(
\begin{array}{@{}c@{}c@{}c@{}c@{}c@{}c@{}c@{}}
	&  &  &  & 2 &  & \\
	1 & 2 & 2 & 3 & 4 & 3 & 1
\end{array}
\right),
\left(
\begin{array}{@{}c@{}c@{}c@{}c@{}c@{}c@{}c@{}}
	&  &  &  & 2 &  & \\
	1 & 1 & 2 & 3 & 4 & 3 & 2
\end{array}
\right),
\left(
\begin{array}{@{}c@{}c@{}c@{}c@{}c@{}c@{}c@{}}
	&  &  &  & 2 &  & \\
	1 & 2 & 3 & 4 & 4 & 2 & 1
\end{array}
\right), 
\left(
\begin{array}{@{}c@{}c@{}c@{}c@{}c@{}c@{}c@{}}
	&  &  &  & 2 &  & \\
	1 & 2 & 3 & 3 & 4 & 3 & 1
\end{array}
\right), \\ 
&\left(
\begin{array}{@{}c@{}c@{}c@{}c@{}c@{}c@{}c@{}}
	&  &  &  & 2 &  & \\
	1 & 2 & 2 & 3 & 4 & 3 & 2
\end{array}
\right),
\left(
\begin{array}{@{}c@{}c@{}c@{}c@{}c@{}c@{}c@{}}
	&  &  &  & 2 &  & \\
	1 & 2 & 3 & 4 & 4 & 3 & 1
\end{array}
\right),
\left(
\begin{array}{@{}c@{}c@{}c@{}c@{}c@{}c@{}c@{}}
	&  &  &  & 2 &  & \\
	1 & 2 & 3 & 3 & 4 & 3 & 2
\end{array}
\right),
\left(
\begin{array}{@{}c@{}c@{}c@{}c@{}c@{}c@{}c@{}}
	&  &  &  & 2 &  & \\
	1 & 2 & 3 & 4 & 5 & 3 & 1
\end{array}
\right),
\left(
\begin{array}{@{}c@{}c@{}c@{}c@{}c@{}c@{}c@{}}
	&  &  &  & 2 &  & \\
	1 & 2 & 3 & 4 & 4 & 3 & 2
\end{array}
\right),
\left(
\begin{array}{@{}c@{}c@{}c@{}c@{}c@{}c@{}c@{}}
	&  &  &  & 3 &  & \\
	1 & 2 & 3 & 4 & 5 & 3 & 1
\end{array}
\right),
\left(
\begin{array}{@{}c@{}c@{}c@{}c@{}c@{}c@{}c@{}}
	&  &  &  & 2 &  & \\
	1 & 2 & 3 & 4 & 5 & 3 & 2
\end{array}
\right), \\&
\left(
\begin{array}{@{}c@{}c@{}c@{}c@{}c@{}c@{}c@{}}
	&  &  &  & 2 &  & \\
	1 & 2 & 3 & 4 & 5 & 4 & 2
\end{array}
\right),
\left(
\begin{array}{@{}c@{}c@{}c@{}c@{}c@{}c@{}c@{}}
	&  &  &  & 3 &  & \\
	1 & 2 & 3 & 4 & 5 & 3 & 2
\end{array}
\right),
\left(
\begin{array}{@{}c@{}c@{}c@{}c@{}c@{}c@{}c@{}}
	&  &  &  & 3 &  & \\
	1 & 2 & 3 & 4 & 5 & 4 & 2
\end{array}
\right),
\left(
\begin{array}{@{}c@{}c@{}c@{}c@{}c@{}c@{}c@{}}
	&  &  &  & 3 &  & \\
	1 & 2 & 3 & 4 & 6 & 4 & 2
\end{array}
\right),
\left(
\begin{array}{@{}c@{}c@{}c@{}c@{}c@{}c@{}c@{}}
	&  &  &  & 3 &  & \\
	1 & 2 & 3 & 5 & 6 & 4 & 2
\end{array}
\right),
\left(
\begin{array}{@{}c@{}c@{}c@{}c@{}c@{}c@{}c@{}}
	&  &  &  & 3 &  & \\
	1 & 2 & 4 & 5 & 6 & 4 & 2
\end{array}
\right),
\left(
\begin{array}{@{}c@{}c@{}c@{}c@{}c@{}c@{}c@{}}
	&  &  &  & 3 &  & \\
	1 & 3 & 4 & 5 & 6 & 4 & 2
\end{array}
\right).
\end{align*}
\end{proposition}

\begin{proof}
The number of roots in the $\Phi$-string of $\lambda$ will be: the number of positive roots of the root system $\Sigma_{\lambda, \Phi} \cong E_{8}$, minus the number of positive roots of $\Sigma_\Phi \cong E_7$, and minus the number of (positive) roots of $\Sigma_{\lambda, \Phi}$ whose coefficient corresponding to $\lambda$ when expressed with respect to $\Pi_{\lambda, \Phi}$ is greater than one. According to~\cite[p.~688--690]{K}, we have that $|I_\Phi^\lambda| = 120 - 63 -1 =56$.
	
Now, on the one hand, one can check that all the roots in the statement of this result are positive roots in an $E_8$ root system (see~\cite[p.~688--690]{K}). On the other hand, they all belong to $I_\Phi^\lambda$ as their coefficient corresponding to $\lambda$ with respect to $\Pi_{\lambda, \Phi}$ is one.
\end{proof}

\subsection{$\Sigma_\Phi$ is a $B_n$ root subsystem of $\Sigma$}\label{subsection:b:f}

\begin{proposition}\label{proposition:b:f:string:type}
Let $ (\Sigma_\Phi, \Sigma_{\lambda, \Phi}) \cong (B_3, F_4)$ and let 
\\
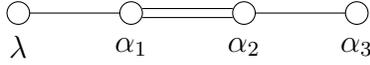
\begin{equation*}
\begin{tikzpicture}[scale=1.5]
\draw (0, 0) circle (0.1);
\draw (1, 0) circle (0.1);
\draw (2, 0) circle (0.1);
\draw (3, 0) circle (0.1);
\draw (0.1, 0.) -- (0.9, 0.);
\draw (1., -0.3) node {$\alpha _1$};
\draw (1.1, -0.04) -- (1.9, -0.04);
\draw (1.1, 0.04) -- (1.9, 0.04);
\draw (2., -0.3) node {$\alpha _2$};
\draw (2.1, 0.) -- (2.9, 0.);
\draw (3., -0.3) node {$\alpha_3$};
\draw (0, -0.3) node {$\lambda$};
\end{tikzpicture}
\end{equation*}
be the Dynkin diagram of $\Sigma_{\lambda, \Phi}$. Then, the $\Phi$-string of $\lambda$ consists of the following roots, and they all have the same length:
\begin{align*}
(1 0 0 0), (1 1 0 0), (1 1 1 0), (1 1 1 1), (1 2 1 0), (1 2 1 1), (1 2 2 1), (1 3 2 1).
\end{align*}
\end{proposition}

\begin{proof}
The number of roots in the $\Phi$-string of $\lambda$ will be: the number of positive roots of the root system $\Sigma_{\lambda, \Phi} \cong F_4$, minus the number of positive roots of $\Sigma_\Phi \cong B_3$, and minus the number of (positive) roots of $\Sigma_{\lambda, \Phi}$ whose coefficient corresponding to $\lambda$ when expressed with respect to $\Pi_{\lambda, \Phi}$ is greater than one. According to~\cite[p.~691]{K}, we have $|I_\Phi^\lambda| = 24 - 9 -7 =8$.

Now, on the one hand, one can check that all the roots in the statement of this result are positive roots in an $F_4$ root system (see~\cite[p.~691]{K}). On the other hand, they all belong to $I_\Phi^\lambda$ as their coefficient corresponding to $\lambda$ with respect to $\Pi_{\lambda, \Phi}$ is one. The claim concerning the lengths follows from a straightforward computation using the Cartan~integers.
\end{proof}

\subsection{$\Sigma_\Phi$ is a $C_n$ root subsystem of $\Sigma$}\label{subsection:c:f}

\begin{proposition}\label{proposition:c:f:string:type}
Let $ (\Sigma_\Phi, \Sigma_{\lambda, \Phi}) \cong (C_3, F_4)$ and let 
\\
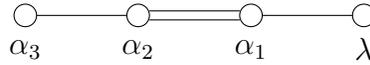
\begin{equation*}\label{string:b:f}
\begin{tikzpicture}[scale=1.5]
\draw (0, 0) circle (0.1);
\draw (1, 0) circle (0.1);
\draw (2, 0) circle (0.1);
\draw (3, 0) circle (0.1);
\draw (0.1, 0.) -- (0.9, 0.);
\draw (1., -0.3) node {$\alpha _2$};
\draw (1.1, -0.04) -- (1.9, -0.04);
\draw (1.1, 0.04) -- (1.9, 0.04);
\draw (2., -0.3) node {$\alpha _1$};
\draw (2.1, 0.) -- (2.9, 0.);
\draw (3., -0.3) node {$\lambda$};
\draw (0, -0.3) node {$\alpha _3$};
\end{tikzpicture}
\end{equation*}
be the Dynkin diagram of $\Sigma_{\lambda, \Phi}$. Then, the $\Phi$-string of $\lambda$ consists of the following roots, and these expressed with single parenthesis have twice the squared norm of these with double:
\begin{align*}
& (0 0 0 1), (0 0 1 1), ((0 1 1 1)), (0 2 1 1), ((1 1 1 1)), (0 2 2 1), ((1 2 1 1)),\\& ((1 2 2 1)), (2 2 1 1), ((1 3 2 1)), (2 2 2 1), ((2 3 2 1)), (2 4 2 1), (2 4 3 1).
\end{align*}
	
\end{proposition}

\begin{proof}
The number of roots in the $\Phi$-string of $\lambda$ will be: the number of positive roots of the root system $\Sigma_{\lambda, \Phi} \cong F_4$, minus the number of positive roots of $\Sigma_\Phi \cong C_3$, and minus the number of (positive) roots of $\Sigma_{\lambda, \Phi}$ whose coefficient corresponding to $\lambda$ when expressed with respect to $\Pi_{\lambda, \Phi}$ is greater than one. According to~\cite[p.~691]{K}, we have that $|I_\Phi^\lambda| = 24 - 9 -1 =14$.

Now, on the one hand, one can check that all the elements in the statement of this result are positive roots in an $F_4$ root system (see~\cite[p.~691]{K}). On the other hand, they all belong to $I_\Phi^\lambda$ as their coefficient corresponding to $\lambda$ with respect to $\Pi_{\lambda, \Phi}$ is one. The claim concerning the lengths follows from a straightforward computation using the Cartan~integers.
\end{proof}

\section{Non-connected case}\label{section:non:connected:phi}

So far, we have been analyzing strings for connected subsets of the set of simple roots. In this section, we deal with a proper non-necessarily connected subset $\Phi$ of $\Pi$ and generalize Proposition~\ref{proposition:root:minimum:level}.

\begin{proposition}\label{proposition:phi:non-connected}
Let $\Phi_1, \dots, \Phi_n$ be connected and mutually orthogonal subsets of the set $\Pi$ of simple roots, put $\Phi = \Phi_1 \cup \dots \cup \Phi_n$, with $n \geq 1$, and assume that $\Phi$ is a proper subset of $\Pi$. Let $\lambda \in \Sigma^{\Phi}$ be the root of minimum level in its $\Phi$-string. Then:
\begin{enumerate}[{\rm (i)}]
\item $I_\Phi^\lambda = \{ \lambda + \sum_{k = 1}^n  \gamma_k \, : \, \gamma_k \in \spann \Phi_k \, \, \, \text{and} \, \, \, \lambda + \gamma_k \in I_{\Phi_k}^\lambda \subset \Sigma^+,  \, \text{for all} \, \, \, k = 1, \dots, n    \}$. \label{proposition:phi:non-connected:1}
\item The root $\gamma \in \Sigma^\Phi$ is of minimum level in its $\Phi$-string if and only if it is of minimum level in its $\Phi_k$-string, for each $k \in \{1, \dots, n\}$. \label{proposition:phi:non-connected:2}
\end{enumerate}
\end{proposition}

\begin{proof}
(\ref{proposition:phi:non-connected:1}): Let $\gamma = \lambda + \sum_{k = 1}^n  \gamma_k$ be a generic root in the $\Phi$-string of $\lambda$, with $\gamma_k$ in $\spann \Phi_k$ for each $k \in \{1, \dots, n\}$. We start by proving the relation ``$\subset$'' in the equality~(\ref{proposition:phi:non-connected:1}) by induction on the level of $\gamma$ with respect to the simple system $\Pi_{\lambda, \Phi}$. If $l(\gamma) = 1$, then $\gamma = \lambda$ and the result is trivial. Assume that our claim holds for roots in the $\Phi$-string of $\lambda$ with less level than $n$ with respect to $\Pi_{\lambda, \Phi}$, and put $l(\gamma)=n > 1$. From Proposition~\ref{proposition:string:simple:rest}, there exists $\alpha_l \in \Phi_l \subset \Phi$, for a certain $l \in \{1, \dots ,n \}$, such that $\gamma - \alpha_l$ is a root in the $\Phi$-string of $\lambda$ (of level $n-1$ with respect to $\Pi_{\lambda, \Phi}$). By applying the induction hypothesis we deduce that $\lambda + \gamma_l - \alpha_l$ and $\lambda + \gamma_{k}$ are roots, for each $k$ different from $l$ satisfying $1 \leq k \leq n$. We just need to see that $\lambda + \gamma_l$ is a root. Using that $\gamma - \alpha_l$ is a root, the fact that we can and will assume  that $\Sigma$ is not a $G_2$ root system (it just contains two simple roots and this would lead to a connected $\Phi$), and that $\gamma$ and $\alpha_l$ are non-proportional, from Proposition~\ref{proposition:strings}~(\ref{proposition:strings:5}) we deduce that either $A_{\alpha_l, \gamma-\alpha_l} \leq -1$, or $\gamma -2 \alpha_l$ is also a root and $A_{\alpha_l, \gamma- 2\alpha_l} = -2$. Using induction again for $\gamma -2\alpha_l$ we deduce that $\lambda + \gamma_l - 2 \alpha_l$ is a root. Now, since the $\Phi_i$'s are mutually orthogonal, we have that either $A_{\alpha_l, \lambda+\gamma_l-\alpha_l}=A_{\alpha_l, \gamma-\alpha_l} \leq -1$, or $A_{\alpha_l, \lambda + \gamma_l - 2 \alpha_l } = A_{\alpha_l, \gamma - 2\alpha_l} = -2$. In both cases, we get that $\lambda + \gamma_l$ is a root by means of Proposition~\ref{proposition:strings}~(v). 

Now, let us prove the relation ``$\supset$'' in the equality~(\ref{proposition:phi:non-connected:1}). In order to do so, we will proceed by induction on the number $n$ of subsets in the decomposition of $\Phi$ into mutually orthogonal connected subsets of $\Pi$. If $n=1$ everything is trivial. Moreover, we can and will assume that the $\Phi_k$-string of $\lambda$ is not trivial for at least $k \in \{1, 2\}$. Otherwise, the $\Phi$-string of $\lambda$ will coincide with the $\Phi_j$-string of $\lambda$, for some $j \in \{1, \dots, n\}$ (see Corollary~\ref{corollary:three:simple:roots}), and since $\Phi_j$ is connected, then our claims are trivial.

Taking into account the above considerations, assume that $\lambda + \gamma_k$ is a root in the $\Phi_k$-string of $\lambda$, with $\gamma_k$ in $\spann \Phi_k$, for each $k \in \{1, \dots, n\}$. Note $\lambda + \gamma_k$ is of minimum level in its $\Phi \backslash \Phi_k$-string and in particular in its $\Psi$-string, for any $\Psi \subset \Phi \backslash \Phi_k$, for any $ k \in \{1, \dots, n\}$. Otherwise, there exists $\alpha \in \Phi \backslash \Phi_k$ such that $\lambda + \gamma_k - \alpha$ is a root by means of Proposition~\ref{proposition:string:simple:rest}. But this contradicts Proposition~\ref{proposition:simple:basis} for the simple system $\Pi_{\lambda, \Phi}$, as the integer corresponding to $\alpha$ is negative. 

Looking at the different configurations of Dynkin diagrams~\cite{K}, and after interchanging indices 1 and 2 if necessary, we can and will assume that all the roots in $\Sigma_{\lambda, \Phi_1}$ have the same length. Using this and the fact that the $\Phi_i$'s are mutually orthogonal, we get $A_{\nu, \lambda} = A_{\nu, \lambda + \gamma_1}$ and $A_{\lambda, \nu} = A_{\lambda + \gamma_1, \nu}$, for any $\nu \in \Phi \backslash \Phi_1$. Thus, the root systems $\Sigma_{\lambda + \gamma_1, \Phi \backslash \Phi_1}$ and $\Sigma_{\lambda, \Phi \backslash \Phi_1}$ are of the same type, have the same type of Dynkin diagram, and with $\lambda$ and $\lambda + \gamma_1$ playing equivalent roles in them. Now, by induction hypothesis, we have that $\lambda + \sum_{k=2}^{n} \gamma_k$ is a root in the $\Phi$-string of $\lambda$. But then so it is $\lambda +\gamma_1 + \sum_{k=2}^{n} \gamma_k$.

(\ref{proposition:phi:non-connected:2}): One of the implications is trivial. Now, assume that $\gamma = \lambda + \sum_{k = 1}^n  \gamma_k$ is not of minimum level in its $\Phi$-string, with $\gamma_k$ in $\spann \Phi_k$ for each $k \in \{1, \dots, n\}$. Then, there must exist $\alpha \in \Phi_l \subset \Phi$ such that $\gamma - \alpha$ is a root, for some $l \in \{1, \dots, n\}$, by virtue of Proposition~\ref{proposition:string:simple:rest}. From~(\ref{proposition:phi:non-connected:1}), we get that $\lambda + \gamma_l - \alpha$ is a root, and hence $\lambda + \gamma_l$ is not of minimum level in its $\Phi_l$-string.
\end{proof}

Now, we can generalize Proposition~\ref{proposition:root:minimum:level} to a proper non-necessarily connected subset $\Phi$ of $\Pi$. Indeed, from Proposition~\ref{proposition:root:minimum:level} and Proposition~\ref{proposition:phi:non-connected}, we deduce the following

\begin{corollary}
Let $\Phi_1, \dots, \Phi_n$ be connected and mutually orthogonal subsets of the set $\Pi$ of simple roots, and put $\Phi = \Phi_1 \cup \dots \cup \Phi_n$, with $n \geq 1$. Assume that $\Phi$ is a proper subset of $\Pi$. Let $\lambda \in \Sigma^\Phi$ and $I \subset \{1, \dots, n\}$ such that the $\Phi_k$-string of $\lambda$ is not trivial if with $k \in I$ and trivial otherwise. If $\lambda$ is of minimum level in its $\Phi$-string, then there exists a root $\alpha_k \in \Phi_k$ such that $A_{\alpha_k, \lambda} < 0$ and $A_{\beta, \lambda} = 0$, for each $\beta \in \Phi_k \backslash \{ \alpha_k\}$, and for each $k \in I$. Moreover, if $|\alpha| \geq |\nu|$ for any $(\alpha, \nu) \in \Phi_k \times I_\Phi^\lambda$, $k \in I$, the converse is true. 
\end{corollary}

We finish this section by proving the main result of the paper.

\begin{proof}[Proof of the Main Theorem]
Item~(\ref{Main:Theorem:1}) from the Main Theorem follows directly from the fact that $\Phi$ constitutes a simple system for the root subsystem $\Sigma_\Phi$ of $\Sigma$ that it spans. Item~(\ref{Main:Theorem:2}) follows from Proposition~\ref{proposition:root:system:lambda:phi}, and item~(\ref{Main:Theorem:3}) from Proposition~\ref{proposition:string:simple:rest}. Now, item~(\ref{Main:Theorem:4}), and hence Table~\ref{table:classical} and Table~\ref{table:exceptional}), they all follow from using together any single result in  bothe Section~\ref{section:classical:root:system} and Section~\ref{section:exceptional:root:system}. Finally, item~(\ref{Main:Theorem:5}) from the Main Theorem corresponds to Proposition~\ref{proposition:phi:non-connected}~(\ref{proposition:phi:non-connected:1}).
\end{proof}

\section{The diagram of a string}\label{section:diagrams}
The main aim of this last section of the paper is to associate an oriented graph to each $\Phi$-string, which allows us to understand better its configuration. Indeed, we typically expose mathematical knowledge in the opposite order of its original discovery. This is the case here, as we first started the investigation of $\Phi$-strings by making extensive use of these oriented graphs that we now present. 

They are constructed in the following way. Let $\Phi$ be a subset of $\Pi$ and $\lambda \in \Sigma^+$ be a root not spanned by $\Phi$ of minimum level in its $\Phi$-string. Recall that we write $I_\Phi^\lambda$ for the set of roots in the $\Phi$-string of $\lambda$. We draw a node in the graph corresponding to each root in $I_\Phi^\lambda$. Now, the nodes corresponding to the roots $\nu_1$, $\nu_2 \in I_\Phi^\lambda$ are connected by a single arrow pointing to $\nu_2$ and labeled as $\alpha$, if and only if $\nu_2 - \nu_1 = \alpha$ for some $\alpha \in \Phi$.

In order to illustrate this, in the following lines we include the oriented graph corresponding to each possible pair $(\Sigma_\Phi, \Sigma_{\lambda, \Phi})$ or equivalently, for each row in Table~\ref{table:classical} and in Table~\ref{table:exceptional}. However, we select a particular $n$ in each case for the sake of simplicity.

\subsection{Proposition~\ref{proposition:a:string:type}. $(\Sigma_\Phi, \Sigma_{\lambda, \Phi}) \cong (A_n, A_{n+1})$, $(A_n, B_{n+1})$ or $(A_n, BC_{n+1})$, with $n = 6$.}

\begin{equation*}
\begin{tikzpicture}[scale=1.4]
	\draw (0, 0) circle (0.1);
	\draw (1, 0) circle (0.1);
	\draw (2, 0) circle (0.1);
	\draw (3, 0) circle (0.1);
	\draw (4, 0) circle (0.1);
	\draw (5, 0) circle (0.1);
	\draw (6, 0) circle (0.1);
	\draw[->] (0.1, 0.) -- (0.5, 0.);
	\draw (0.5, 0.) -- (0.9, 0.);
	\draw (0.5, -0.175) node {$\alpha _1$};
	\draw[->] (1.1, 0.) -- (1.5, 0.);
	\draw (1.5, 0.) -- (1.9, 0.);
	\draw (1.5, -0.175) node {$\alpha _2$};
	\draw[->] (2.1, 0.) -- (2.5, 0.);
	\draw (2.5, 0.) -- (2.9, 0.);
	\draw (2.5, -0.175) node {$\alpha _3$};
	\draw[->] (3.1, 0.) -- (3.5, 0.);
	\draw (3.5, 0.) -- (3.9, 0.);
	\draw (3.5, -0.175) node {$\alpha _4$};
	\draw[->] (4.1, 0.) -- (4.5, 0.);
	\draw (4.5, 0.) -- (4.9, 0.);
	\draw (4.5, -0.175) node {$\alpha _5$};
	\draw[->] (5.1, 0.) -- (5.5, 0.);
	\draw (5.5, 0.) -- (5.9, 0.);
	\draw (5.5, -0.175) node {$\alpha _6$};
	\draw (0, -0.25) node {$\lambda$};
\end{tikzpicture}
\end{equation*}

\subsection{Proposition~\ref{proposition:c:string:type}. $(\Sigma_\Phi, \Sigma_{\lambda, \Phi}) \cong (A_n, C_{n+1})$ with $n=5$.}
\begin{equation*}
\begin{tikzpicture}[scale=1.4]
	\draw (0, 0) circle (0.1);
	\draw (1, 0) circle (0.1);
	\draw (2, 0) circle (0.1);
	\draw (3, 0) circle (0.1);
	\draw (1, 1) circle (0.1);
	\draw (2, 1) circle (0.1);
	\draw (3, 1) circle (0.1);
	\draw (2, 2) circle (0.1);
	\draw (3, 2) circle (0.1);
	\draw (3, 3) circle (0.1);
	\draw (4, 0) circle (0.1);
	\draw (4, 1) circle (0.1);
	\draw (4, 2) circle (0.1);
	\draw (4, 3) circle (0.1);
	\draw (4, 4) circle (0.1);
	\draw (5, 0) circle (0.1);
	\draw (5, 1) circle (0.1);
	\draw (5, 2) circle (0.1);
	\draw (5, 4) circle (0.1);
	\draw (5, 5) circle (0.1);
	\draw (5, 3) circle (0.1);
	\draw[->] (0.1, 0.) -- (0.5, 0.);
	\draw (0.5, 0.) -- (0.9, 0.);
	\draw (0.5, -0.175) node {$\alpha _1$};
	\draw[->] (1.1, 0.) -- (1.5, 0.);
	\draw (1.5, 0.) -- (1.9, 0.);
	\draw (1.5, -0.175) node {$\alpha _2$};
	\draw[->] (2.1, 0.) -- (2.5, 0.);
	\draw (2.5, 0.) -- (2.9, 0.);
	\draw (2.5, -0.175) node {$\alpha _3$};
	\draw[->] (3.1, 0.) -- (3.5, 0.);
	\draw (3.5, 0.) -- (3.9, 0.);
	\draw (3.5, -0.175) node {$\alpha _4$};
	\draw[->] (4.1, 0.) -- (4.5, 0.);
	\draw (4.5, 0.) -- (4.9, 0.);
	\draw (4.5, -0.175) node {$\alpha _5$};
	\draw[->] (1.1, 1.) -- (1.5, 1.);
	\draw (1.5, 1.) -- (1.9, 1.);
	\draw (1.5, 0.825) node {$\alpha _2$};
	\draw[->] (2.1, 1.) -- (2.5, 1.);
	\draw (2.5, 1.) -- (2.9, 1.);
	\draw (2.5, 0.825) node {$\alpha _3$};
	\draw[->] (4.1, 0.) -- (4.5, 0.);
	\draw (4.5, 0.) -- (4.9, 0.);
	\draw (4.5, -0.175) node {$\alpha _5$};
	\draw[->] (1., 0.1) -- (1., 0.5);
	\draw (1., 0.5) -- (1., 0.9);
	\draw (0.8, 0.5) node {$\alpha _1$};
	\draw[->] (2., 0.1) -- (2., 0.5);
	\draw (2., 0.5) -- (2., 0.9);
	\draw (1.8, 0.5) node {$\alpha _1$};
	\draw[->] (3., 0.1) -- (3., 0.5);
	\draw (3., 0.5) -- (3., 0.9);
	\draw (2.8, 0.5) node {$\alpha _1$};
	\draw[->] (2., 1.1) -- (2., 1.5);
	\draw (2., 1.5) -- (2., 1.9);
	\draw (1.8, 1.5) node {$\alpha _2$};
	\draw[->] (3., 1.1) -- (3., 1.5);
	\draw (3., 1.5) -- (3., 1.9);
	\draw (2.8, 1.5) node {$\alpha _2$};
	\draw[->] (2.1, 2.) -- (2.5, 2.);
	\draw (2.5, 2.) -- (2.9, 2.);
	\draw (2.5, 1.825) node {$\alpha _3$};
	\draw[->] (3., 2.1) -- (3., 2.5);
	\draw (3., 2.5) -- (3., 2.9);
	\draw (2.8, 2.5) node {$\alpha _3$};
	\draw[->] (3.1, 1.) -- (3.5, 1.);
	\draw (3.5, 1.) -- (3.9, 1.);
	\draw (3.5, 0.825) node {$\alpha _4$};
	\draw[->] (3.1, 2.) -- (3.5, 2.);
	\draw (3.5, 2.) -- (3.9, 2.);
	\draw (3.5, 1.825) node {$\alpha _4$};
	\draw[->] (4., 0.1) -- (4., 0.5);
	\draw (4., 0.5) -- (4., 0.9);
	\draw (3.8, 0.5) node {$\alpha _1$};
	\draw[->] (4., 1.1) -- (4., 1.5);
	\draw (4., 1.5) -- (4., 1.9);
	\draw (3.8, 1.5) node {$\alpha _2$};
	\draw[->] (4., 2.1) -- (4., 2.5);
	\draw (4., 2.5) -- (4., 2.9);
	\draw (3.8, 2.5) node {$\alpha _3$};
	\draw[->] (4., 3.1) -- (4., 3.5);
	\draw (4., 3.5) -- (4., 3.9);
	\draw (3.8, 3.5) node {$\alpha _4$};
	\draw[->] (3.1, 3.) -- (3.5, 3.);
	\draw (3.5, 3.) -- (3.9, 3.);
	\draw (3.5, 2.825) node {$\alpha _4$};
	\draw[->] (4.1, 1.) -- (4.5, 1.);
	\draw (4.5, 1.) -- (4.9, 1.);
	\draw (4.5, 0.825) node {$\alpha _5$};
	\draw[->] (4.1, 2.) -- (4.5, 2.);
	\draw (4.5, 2.) -- (4.9, 2.);
	\draw (4.5, 1.825) node {$\alpha _5$};
	\draw[->] (4.1, 3.) -- (4.5, 3.);
	\draw (4.5, 3.) -- (4.9, 3.);
	\draw (4.5, 2.825) node {$\alpha _5$};
	\draw[->] (4.1, 4.) -- (4.5, 4.);
	\draw (4.5, 4.) -- (4.9, 4.);
	\draw (4.5, 3.825) node {$\alpha _5$};
	\draw[->] (5., 0.1) -- (5., 0.5);
	\draw (5., 0.5) -- (5., 0.9);
	\draw (4.8, 0.5) node {$\alpha _1$};
	\draw[->] (5., 1.1) -- (5., 1.5);
	\draw (5., 1.5) -- (5., 1.9);
	\draw (4.8, 1.5) node {$\alpha _2$};
	\draw[->] (5., 2.1) -- (5., 2.5);
	\draw (5., 2.5) -- (5., 2.9);
	\draw (4.8, 2.5) node {$\alpha _3$};
	\draw[->] (5., 3.1) -- (5., 3.5);
	\draw (5., 3.5) -- (5., 3.9);
	\draw (4.8, 3.5) node {$\alpha _4$};
	\draw[->] (5., 4.1) -- (5., 4.5);
	\draw (5., 4.5) -- (5., 4.9);
	\draw (4.8, 4.5) node {$\alpha _5$};
	\draw (0, -0.25) node {$\lambda$};
\end{tikzpicture}
\end{equation*}
\subsection{Proposition~\ref{proposition:d:string:type}. $(\Sigma_\Phi, \Sigma_{\lambda, \Phi}) \cong (A_n, D_{n+1})$ with $n = 6$.}

\begin{equation*}
\begin{tikzpicture}[scale=1.4]
	\draw (0, 0) circle (0.1);
	\draw (1, 0) circle (0.1);
	\draw (2, 0) circle (0.1);
	\draw (3, 0) circle (0.1);
	\draw (1, 1) circle (0.1);
	\draw (2, 1) circle (0.1);
	\draw (3, 1) circle (0.1);
	\draw (2, 2) circle (0.1);
	\draw (3, 2) circle (0.1);
	\draw (3, 3) circle (0.1);
	\draw (4, 0) circle (0.1);
	\draw (4, 1) circle (0.1);
	\draw (4, 2) circle (0.1);
	\draw (4, 3) circle (0.1);
	\draw (4, 4) circle (0.1);
	\draw (5, 0) circle (0.1);
	\draw (5, 1) circle (0.1);
	\draw (5, 2) circle (0.1);
	\draw (5, 4) circle (0.1);
	\draw (5, 5) circle (0.1);
	\draw (5, 3) circle (0.1);
	\draw[->] (0.1, 0.) -- (0.5, 0.);
	\draw (0.5, 0.) -- (0.9, 0.);
	\draw (0.5, -0.175) node {$\alpha _2$};
	\draw[->] (1.1, 0.) -- (1.5, 0.);
	\draw (1.5, 0.) -- (1.9, 0.);
	\draw (1.5, -0.175) node {$\alpha _3$};
	\draw[->] (2.1, 0.) -- (2.5, 0.);
	\draw (2.5, 0.) -- (2.9, 0.);
	\draw (2.5, -0.175) node {$\alpha _4$};
	\draw[->] (1.1, 1.) -- (1.5, 1.);
	\draw (1.5, 1.) -- (1.9, 1.);
	\draw (1.5, 0.825) node {$\alpha _3$};
	\draw[->] (1., 0.1) -- (1., 0.5);
	\draw (1., 0.5) -- (1., 0.9);
	\draw (0.8, 0.5) node {$\alpha _1$};
	\draw[->] (2., 0.1) -- (2., 0.5);
	\draw (2., 0.5) -- (2., 0.9);
	\draw (1.8, 0.5) node {$\alpha _1$};
	\draw[->] (3., 0.1) -- (3., 0.5);
	\draw (3., 0.5) -- (3., 0.9);
	\draw (2.8, 0.5) node {$\alpha _1$};
	\draw[->] (2., 1.1) -- (2., 1.5);
	\draw (2., 1.5) -- (2., 1.9);
	\draw (1.8, 1.5) node {$\alpha _2$};
	\draw[->] (2.1, 1.) -- (2.5, 1.);
	\draw (2.5, 1.) -- (2.9, 1.);
	\draw (2.5, 0.825) node {$\alpha _4$};
	\draw[->] (3., 1.1) -- (3., 1.5);
	\draw (3., 1.5) -- (3., 1.9);
	\draw (2.8, 1.5) node {$\alpha _2$};
	\draw[->] (2.1, 2.) -- (2.5, 2.);
	\draw (2.5, 2.) -- (2.9, 2.);
	\draw (2.5, 1.825) node {$\alpha _4$};
	\draw[->] (3., 2.1) -- (3., 2.5);
	\draw (3., 2.5) -- (3., 2.9);
	\draw (2.8, 2.5) node {$\alpha _3$};
	\draw[->] (3.1, 0.) -- (3.5, 0.);
	\draw (3.5, 0.) -- (3.9, 0.);
	\draw (3.5, -0.175) node {$\alpha _5$};
	\draw[->] (3.1, 1.) -- (3.5, 1.);
	\draw (3.5, 1.) -- (3.9, 1.);
	\draw (3.5, 0.825) node {$\alpha _5$};
	\draw[->] (3.1, 2.) -- (3.5, 2.);
	\draw (3.5, 2.) -- (3.9, 2.);
	\draw (3.5, 1.825) node {$\alpha _5$};
	\draw[->] (4., 0.1) -- (4., 0.5);
	\draw (4., 0.5) -- (4., 0.9);
	\draw (3.8, 0.5) node {$\alpha _1$};
	\draw[->] (4., 1.1) -- (4., 1.5);
	\draw (4., 1.5) -- (4., 1.9);
	\draw (3.8, 1.5) node {$\alpha _2$};
	\draw[->] (4., 2.1) -- (4., 2.5);
	\draw (4., 2.5) -- (4., 2.9);
	\draw (3.8, 2.5) node {$\alpha _3$};
	\draw[->] (4., 3.1) -- (4., 3.5);
	\draw (4., 3.5) -- (4., 3.9);
	\draw (3.8, 3.5) node {$\alpha _4$};
	\draw[->] (3.1, 3.) -- (3.5, 3.);
	\draw (3.5, 3.) -- (3.9, 3.);
	\draw (3.5, 2.825) node {$\alpha _5$};
	\draw[->] (4.1, 0.) -- (4.5, 0.);
	\draw (4.5, 0.) -- (4.9, 0.);
	\draw (4.5, -0.175) node {$\alpha _6$};
	\draw[->] (4.1, 1.) -- (4.5, 1.);
	\draw (4.5, 1.) -- (4.9, 1.);
	\draw (4.5, 0.825) node {$\alpha _6$};
	\draw[->] (4.1, 2.) -- (4.5, 2.);
	\draw (4.5, 2.) -- (4.9, 2.);
	\draw (4.5, 1.825) node {$\alpha _6$};
	\draw[->] (4.1, 3.) -- (4.5, 3.);
	\draw (4.5, 3.) -- (4.9, 3.);
	\draw (4.5, 2.825) node {$\alpha _6$};
	\draw[->] (4.1, 4.) -- (4.5, 4.);
	\draw (4.5, 4.) -- (4.9, 4.);
	\draw (4.5, 3.825) node {$\alpha _6$};
	\draw[->] (5., 0.1) -- (5., 0.5);
	\draw (5., 0.5) -- (5., 0.9);
	\draw (4.8, 0.5) node {$\alpha _1$};
	\draw[->] (5., 1.1) -- (5., 1.5);
	\draw (5., 1.5) -- (5., 1.9);
	\draw (4.8, 1.5) node {$\alpha _2$};
	\draw[->] (5., 2.1) -- (5., 2.5);
	\draw (5., 2.5) -- (5., 2.9);
	\draw (4.8, 2.5) node {$\alpha _3$};
	\draw[->] (5., 3.1) -- (5., 3.5);
	\draw (5., 3.5) -- (5., 3.9);
	\draw (4.8, 3.5) node {$\alpha _4$};
	\draw[->] (5., 4.1) -- (5., 4.5);
	\draw (5., 4.5) -- (5., 4.9);
	\draw (4.8, 4.5) node {$\alpha _5$};
	\draw (0, -0.25) node {$\lambda$};
\end{tikzpicture}
\end{equation*}

\subsection{Proposition~\ref{proposition:b:string:type}. $(\Sigma_\Phi, \Sigma_{\lambda, \Phi}) \cong (B_n, B_{n+1})$ or $(BC_n, BC_{n+1})$ with $n = 5$.}

\begin{equation*}
\begin{tikzpicture}[scale=1.4]
	\draw (0, 0) circle (0.1);
	\draw (1, 0) circle (0.1);
	\draw (2, 0) circle (0.1);
	\draw (3, 0) circle (0.1);
	\draw (4, 0) circle (0.1);
	\draw (5, 0) circle (0.1);
	\draw (6, 0) circle (0.1);
	\draw (7, 0) circle (0.1);
	\draw (8, 0) circle (0.1);
	\draw (9, 0) circle (0.1);
	\draw (10, 0) circle (0.1);
	\draw[->] (0.1, 0.) -- (0.5, 0.);
	\draw (0.5, 0.) -- (0.9, 0.);
	\draw (0.5, -0.175) node {$\alpha _1$};
	\draw[->] (1.1, 0.) -- (1.5, 0.);
	\draw (1.5, 0.) -- (1.9, 0.);
	\draw (1.5, -0.175) node {$\alpha _2$};
	\draw[->] (2.1, 0.) -- (2.5, 0.);
	\draw (2.5, 0.) -- (2.9, 0.);
	\draw (2.5, -0.175) node {$\alpha _3$};
	\draw[->] (3.1, 0.) -- (3.5, 0.);
	\draw (3.5, 0.) -- (3.9, 0.);
	\draw (3.5, -0.175) node {$\alpha _4$};
	\draw[->] (4.1, 0.) -- (4.5, 0.);
	\draw (4.5, 0.) -- (4.9, 0.);
	\draw (4.5, -0.175) node {$\alpha _5$};
	\draw[->] (5.1, 0.) -- (5.5, 0.);
	\draw (5.5, 0.) -- (5.9, 0.);
	\draw (5.5, -0.175) node {$\alpha _5$};
	\draw[->] (6.1, 0.) -- (6.5, 0.);
	\draw (6.5, 0.) -- (6.9, 0.);
	\draw (6.5, -0.175) node {$\alpha _4$};
	\draw[->] (7.1, 0.) -- (7.5, 0.);
	\draw (7.5, 0.) -- (7.9, 0.);
	\draw (7.5, -0.175) node {$\alpha _3$};
	\draw[->] (8.1, 0.) -- (8.5, 0.);
	\draw (8.5, 0.) -- (8.9, 0.);
	\draw (8.5, -0.175) node {$\alpha _2$};
	\draw[->] (9.1, 0.) -- (9.5, 0.);
	\draw (9.5, 0.) -- (9.9, 0.);
	\draw (9.5, -0.175) node {$\alpha _1$};
	\draw (0, -0.25) node {$\lambda$};
\end{tikzpicture}
\end{equation*}

\subsection{Proposition~\ref{proposition:c:pure:string:type}. $(\Sigma_\Phi, \Sigma_{\lambda, \Phi}) \cong (C_n, C_{n+1})$ with $n = 5$.}

\begin{equation*}
\begin{tikzpicture}[scale=1.4]
	\draw (0, 0) circle (0.1);
	\draw (1, 0) circle (0.1);
	\draw (2, 0) circle (0.1);
	\draw (3, 0) circle (0.1);
	\draw (4, 0) circle (0.1);
	\draw (5, 0) circle (0.1);
	\draw (6, 0) circle (0.1);
	\draw (7, 0) circle (0.1);
	\draw (8, 0) circle (0.1);
	\draw (9, 0) circle (0.1);
	\draw[->] (0.1, 0.) -- (0.5, 0.);
	\draw (0.5, 0.) -- (0.9, 0.);
	\draw (0.5, -0.175) node {$\alpha _1$};
	\draw[->] (1.1, 0.) -- (1.5, 0.);
	\draw (1.5, 0.) -- (1.9, 0.);
	\draw (1.5, -0.175) node {$\alpha _2$};
	\draw[->] (2.1, 0.) -- (2.5, 0.);
	\draw (2.5, 0.) -- (2.9, 0.);
	\draw (2.5, -0.175) node {$\alpha _3$};
	\draw[->] (3.1, 0.) -- (3.5, 0.);
	\draw (3.5, 0.) -- (3.9, 0.);
	\draw (3.5, -0.175) node {$\alpha _4$};
	\draw[->] (4.1, 0.) -- (4.5, 0.);
	\draw (4.5, 0.) -- (4.9, 0.);
	\draw (4.5, -0.175) node {$\alpha _5$};
	\draw[->] (5.1, 0.) -- (5.5, 0.);
	\draw (5.5, 0.) -- (5.9, 0.);
	\draw (5.5, -0.175) node {$\alpha _4$};
	\draw[->] (6.1, 0.) -- (6.5, 0.);
	\draw (6.5, 0.) -- (6.9, 0.);
	\draw (6.5, -0.175) node {$\alpha _3$};
	\draw[->] (6.1, 0.) -- (6.5, 0.);
	\draw (6.5, 0.) -- (6.9, 0.);
	\draw (6.5, -0.175) node {$\alpha _3$};
	\draw[->] (7.1, 0.) -- (7.5, 0.);
	\draw (7.5, 0.) -- (7.9, 0.);
	\draw (7.5, -0.175) node {$\alpha _2$};
	\draw[->] (8.1, 0.) -- (8.5, 0.);
	\draw (8.5, 0.) -- (8.9, 0.);
	\draw (8.5, -0.175) node {$\alpha _1$};
	\draw (0, -0.25) node {$\lambda$};
\end{tikzpicture}
\end{equation*}

\subsection{Proposition~\ref{proposition:d:pure:string:type}. $(\Sigma_\Phi, \Sigma_{\lambda, \Phi}) \cong (D_n, D_{n+1})$ with $n = 6$.}

\begin{equation*}
\begin{tikzpicture}[scale=1.4]
	\draw (-2, 0) circle (0.1);
	\draw (-1, 0) circle (0.1);
	\draw (0, 0) circle (0.1);
	\draw (1, 0) circle (0.1);
	\draw (2, 0) circle (0.1);
	\draw (3, 0) circle (0.1);
	\draw (2.5, 0.5) circle (0.1);
	\draw (2.5, -0.5) circle (0.1);
	\draw (4, 0) circle (0.1);
	\draw (4, 0) circle (0.1);
	\draw (5, 0) circle (0.1);
	\draw (6, 0) circle (0.1);
	\draw (7, 0) circle (0.1);
	\draw[->] (-1.9, 0.) -- (-1.5, 0.);
	\draw (-1.5, 0.) -- (-1.1, 0.);
	\draw (-1.5, -0.175) node {$\alpha _1$};
	\draw[->] (-0.9, 0.) -- (-0.5, 0.);
	\draw (-0.5, 0.) -- (-0.1, 0.);
	\draw (-0.5, -0.175) node {$\alpha _2$};
	\draw[->] (0.1, 0.) -- (0.5, 0.);
	\draw (0.5, 0.) -- (0.9, 0.);
	\draw (0.5, -0.175) node {$\alpha _3$};
	\draw[->] (1.1, 0.) -- (1.5, 0.);
	\draw (1.5, 0.) -- (1.9, 0.);
	\draw (1.5, -0.175) node {$\alpha _4$};
	\draw[->] (2.57071, -0.429289) -- (2.75, -0.25);
	\draw (2.75, -0.25) -- (2.92929, -0.0707107);
	\draw (2.9, -0.33) node {$\alpha _5$};
	\draw[->] (2.07071, 0.0707107) -- (2.25, 0.25);
	\draw (2.25, 0.25) -- (2.42929, 0.429289);
	\draw (2.1, 0.35) node {$\alpha _5$};
	\draw[->] (3.1, 0.) -- (3.5, 0.);
	\draw (3.5, 0.) -- (3.9, 0.);
	\draw (3.5, -0.175) node {$\alpha _4$};
	\draw[->] (4.1, 0.) -- (4.5, 0.);
	\draw (4.5, 0.) -- (4.9, 0.);
	\draw (4.5, -0.175) node {$\alpha _3$};
	\draw[->] (5.1, 0.) -- (5.5, 0.);
	\draw (5.5, 0.) -- (5.9, 0.);
	\draw (5.5, -0.175) node {$\alpha _2$};
	\draw[->] (6.1, 0.) -- (6.5, 0.);
	\draw (6.5, 0.) -- (6.9, 0.);
	\draw (6.5, -0.175) node {$\alpha _1$};
	\draw[->] (2.07071, -0.0707107) -- (2.25, -0.25);
	\draw (2.25, -0.25) -- (2.42929, -0.429289);
	\draw (2.1, -0.33) node {$\alpha _6$};
	\draw[->] (2.57071, 0.429289) -- (2.75, 0.25);
	\draw (2.75, 0.25) -- (2.92929, 0.0707107);
	\draw (2.9, 0.35) node {$\alpha _6$};
	\draw (-2, -0.25) node {$\lambda$};
\end{tikzpicture}
\end{equation*}

\subsection{Proposition~\ref{proposition:a:e:string:type}. $(\Sigma_\Phi, \Sigma_{\lambda, \Phi}) \cong (A_5, E_{6})$.}
Since any square has to be commutative, we do not include all the labels in the central cube of the diagram for the sake of a cleaner picture.

\begin{equation*}
\begin{tikzpicture}[scale=1.4]
	\draw (1, 0) circle (0.1);
	\draw (2, 0) circle (0.1);
	\draw (3, 0) circle (0.1);
	\draw (4, 0) circle (0.1);
	\draw (2, 2) circle (0.1);
	\draw (2, 1) circle (0.1);
	\draw (3, 1) circle (0.1);
	\draw (4, 1) circle (0.1);
	\draw (3, 2) circle (0.1);
	\draw (4, 2) circle (0.1);
	\draw (3.5, 1.5) circle (0.1);
	\draw (4.5, 1.5) circle (0.1);
	\draw (3.5, 2.5) circle (0.1);
	\draw (3.5, 3.5) circle (0.1);
	\draw (4.5, 2.5) circle (0.1);
	\draw (4.5, 3.5) circle (0.1);
	\draw (5.5, 1.5) circle (0.1);
	\draw (5.5, 2.5) circle (0.1);
	\draw (5.5, 3.5) circle (0.1);
	\draw (6.5, 3.5) circle (0.1);
	\draw[->] (1.1, 0.) -- (1.5, 0.);
	\draw (1.5, 0.) -- (1.9, 0.);
	\draw (1.5, -0.175) node {$\alpha _3$};
	\draw[->] (2.1, 0.) -- (2.5, 0.);
	\draw (2.5, 0.) -- (2.9, 0.);
	\draw (2.5, -0.175) node {$\alpha _2$};
	\draw[->] (3.1, 0.) -- (3.5, 0.);
	\draw (3.5, 0.) -- (3.9, 0.);
	\draw (3.5, -0.175) node {$\alpha _1$};
	\draw[->] (2.1, 1.) -- (2.5, 1.);
	\draw (2.5, 1.) -- (2.9, 1.);
	\draw (2.5, 0.825) node {$\alpha _2$};
	\draw[->] (3.1, 1.) -- (3.5, 1.);
	\draw (3.5, 1.) -- (3.9, 1.);
	\draw (3.5, 0.825) node {$\alpha _1$};
	\draw[->] (2.1, 2.) -- (2.5, 2.);
	\draw (2.5, 2.) -- (2.9, 2.);
	\draw (2.5, 1.825) node {$\alpha _2$};
	\draw[->] (3.1, 2.) -- (3.5, 2.);
	\draw (3.5, 2.) -- (3.9, 2.);
	\draw (3.5, 1.825) node {$\text{}$};
	\draw[->] (2., 0.1) -- (2., 0.5);
	\draw (2., 0.5) -- (2., 0.9);
	\draw (1.8, 0.5) node {$\alpha _4$};
	\draw[->] (2., 0.1) -- (2., 0.5);
	\draw (2., 0.5) -- (2., 0.9);
	\draw (1.8, 0.5) node {$\alpha _4$};
	\draw[->] (3., 0.1) -- (3., 0.5);
	\draw (3., 0.5) -- (3., 0.9);
	\draw (2.8, 0.5) node {$\alpha _4$};
	\draw[->] (4., 0.1) -- (4., 0.5);
	\draw (4., 0.5) -- (4., 0.9);
	\draw (3.8, 0.5) node {$\alpha _4$};
	\draw[->] (3., 1.1) -- (3., 1.5);
	\draw (3., 1.5) -- (3., 1.9);
	\draw (2.8, 1.5) node {$\alpha _5$};
	\draw[->] (4., 1.1) -- (4., 1.5);
	\draw (4., 1.5) -- (4., 1.9);
	\draw (3.8, 1.5) node {$\text{}$};
	\draw[->] (2., 1.1) -- (2., 1.5);
	\draw (2., 1.5) -- (2., 1.9);
	\draw (1.8, 1.5) node {$\alpha _5$};
	\draw[->] (3.07071, 1.07071) -- (3.25, 1.25);
	\draw (3.25, 1.25) -- (3.42929, 1.42929);
	\draw (3.1, 1.35) node {$\text{}$};
	\draw[->] (4.07071, 1.07071) -- (4.25, 1.25);
	\draw (4.25, 1.25) -- (4.42929, 1.42929);
	\draw (4.1, 1.35) node {$\text{}$};
	\draw[->] (3.07071, 2.07071) -- (3.25, 2.25);
	\draw (3.25, 2.25) -- (3.42929, 2.42929);
	\draw (3.1, 2.35) node {$\alpha _3$};
	\draw[->] (4.07071, 2.07071) -- (4.25, 2.25);
	\draw (4.25, 2.25) -- (4.42929, 2.42929);
	\draw (4.1, 2.35) node {$\text{}$};
	\draw[->] (3.6, 1.5) -- (4., 1.5);
	\draw (4., 1.5) -- (4.4, 1.5);
	\draw (4., 1.325) node {$\text{}$};
	\draw[->] (3.5, 1.6) -- (3.5, 2.);
	\draw (3.5, 2.) -- (3.5, 2.4);
	\draw (3.3, 2.) node {$\text{}$};
	\draw[->] (4.5, 1.6) -- (4.5, 2.);
	\draw (4.5, 2.) -- (4.5, 2.4);
	\draw (4.3, 2.) node {$\text{}$};
	\draw[->] (3.6, 2.5) -- (4., 2.5);
	\draw (4., 2.5) -- (4.4, 2.5);
	\draw (4., 2.325) node {$\alpha _1$};
	\draw[->] (4.6, 1.5) -- (5., 1.5);
	\draw (5., 1.5) -- (5.4, 1.5);
	\draw (5., 1.325) node {$\alpha _2$};
	\draw[->] (4.6, 2.5) -- (5., 2.5);
	\draw (5., 2.5) -- (5.4, 2.5);
	\draw (5., 2.325) node {$\alpha _2$};
	\draw[->] (4.6, 3.5) -- (5., 3.5);
	\draw (5., 3.5) -- (5.4, 3.5);
	\draw (5., 3.325) node {$\alpha _2$};
	\draw[->] (3.6, 3.5) -- (4., 3.5);
	\draw (4., 3.5) -- (4.4, 3.5);
	\draw (4., 3.325) node {$\alpha _1$};
	\draw[->] (5.6, 3.5) -- (6., 3.5);
	\draw (6., 3.5) -- (6.4, 3.5);
	\draw (6., 3.325) node {$\alpha _3$};
	\draw[->] (3.5, 2.6) -- (3.5, 3.);
	\draw (3.5, 3.) -- (3.5, 3.4);
	\draw (3.3, 3.) node {$\alpha _4$};
	\draw[->] (4.5, 2.6) -- (4.5, 3.);
	\draw (4.5, 3.) -- (4.5, 3.4);
	\draw (4.3, 3.) node {$\alpha _4$};
	\draw[->] (5.5, 2.6) -- (5.5, 3.);
	\draw (5.5, 3.) -- (5.5, 3.4);
	\draw (5.3, 3.) node {$\alpha _4$};
	\draw[->] (5.5, 1.6) -- (5.5, 2.);
	\draw (5.5, 2.) -- (5.5, 2.4);
	\draw (5.3, 2.) node {$\alpha _5$};
	\draw (1, -0.25) node {$\lambda$};
\end{tikzpicture}
\end{equation*}

\subsection{Proposition~\ref{proposition:d:e:string:type}. $(\Sigma_\Phi, \Sigma_{\lambda, \Phi}) \cong (D_5, E_{6})$.}

\begin{equation*}
	\begin{tikzpicture}[scale=1.4]
		\draw (0, 0) circle (0.1);
		\draw (1, 0) circle (0.1);
		\draw (2, 0) circle (0.1);
		\draw (3, 0) circle (0.1);
		\draw (4, 0) circle (0.1);
		\draw (2, 1) circle (0.1);
		\draw (3, 1) circle (0.1);
		\draw (4, 1) circle (0.1);
		\draw (3, 2) circle (0.1);
		\draw (4, 2) circle (0.1);
		\draw (5, 2) circle (0.1);
		\draw (3, 3) circle (0.1);
		\draw (4, 3) circle (0.1);
		\draw (5, 3) circle (0.1);
		\draw (6, 3) circle (0.1);
		\draw (7, 3) circle (0.1);
		\draw[->] (0.1, 0.) -- (0.5, 0.);
		\draw (0.5, 0.) -- (0.9, 0.);
		\draw (0.5, -0.175) node {$\alpha _1$};
		\draw[->] (1.1, 0.) -- (1.5, 0.);
		\draw (1.5, 0.) -- (1.9, 0.);
		\draw (1.5, -0.175) node {$\alpha _3$};
		\draw[->] (2.1, 0.) -- (2.5, 0.);
		\draw (2.5, 0.) -- (2.9, 0.);
		\draw (2.5, -0.175) node {$\alpha _4$};
		\draw[->] (3.1, 0.) -- (3.5, 0.);
		\draw (3.5, 0.) -- (3.9, 0.);
		\draw (3.5, -0.175) node {$\alpha _5$};
		\draw[->] (2.1, 1.) -- (2.5, 1.);
		\draw (2.5, 1.) -- (2.9, 1.);
		\draw (2.5, 0.825) node {$\alpha _4$};
		\draw[->] (3.1, 1.) -- (3.5, 1.);
		\draw (3.5, 1.) -- (3.9, 1.);
		\draw (3.5, 0.825) node {$\alpha _5$};
		\draw[->] (3.1, 2.) -- (3.5, 2.);
		\draw (3.5, 2.) -- (3.9, 2.);
		\draw (3.5, 1.825) node {$\alpha _5$};
		\draw[->] (4.1, 2.) -- (4.5, 2.);
		\draw (4.5, 2.) -- (4.9, 2.);
		\draw (4.5, 1.825) node {$\alpha _4$};
		\draw[->] (3.1, 3.) -- (3.5, 3.);
		\draw (3.5, 3.) -- (3.9, 3.);
		\draw (3.5, 2.825) node {$\alpha _5$};
		\draw[->] (4.1, 3.) -- (4.5, 3.);
		\draw (4.5, 3.) -- (4.9, 3.);
		\draw (4.5, 2.825) node {$\alpha _4$};
		\draw[->] (5.1, 3.) -- (5.5, 3.);
		\draw (5.5, 3.) -- (5.9, 3.);
		\draw (5.5, 2.825) node {$\alpha _3$};
		\draw[->] (6.1, 3.) -- (6.5, 3.);
		\draw (6.5, 3.) -- (6.9, 3.);
		\draw (6.5, 2.825) node {$\alpha _2$};
		\draw[->] (2., 0.1) -- (2., 0.5);
		\draw (2., 0.5) -- (2., 0.9);
		\draw (1.8, 0.5) node {$\alpha _2$};
		\draw[->] (3., 0.1) -- (3., 0.5);
		\draw (3., 0.5) -- (3., 0.9);
		\draw (2.8, 0.5) node {$\alpha _2$};
		\draw[->] (4., 0.1) -- (4., 0.5);
		\draw (4., 0.5) -- (4., 0.9);
		\draw (3.8, 0.5) node {$\alpha _2$};
		\draw[->] (3., 1.1) -- (3., 1.5);
		\draw (3., 1.5) -- (3., 1.9);
		\draw (2.8, 1.5) node {$\alpha _3$};
		\draw[->] (4., 1.1) -- (4., 1.5);
		\draw (4., 1.5) -- (4., 1.9);
		\draw (3.8, 1.5) node {$\alpha _3$};
		\draw[->] (3., 2.1) -- (3., 2.5);
		\draw (3., 2.5) -- (3., 2.9);
		\draw (2.8, 2.5) node {$\alpha _1$};
		\draw[->] (4., 2.1) -- (4., 2.5);
		\draw (4., 2.5) -- (4., 2.9);
		\draw (3.8, 2.5) node {$\alpha _1$};
		\draw[->] (5., 2.1) -- (5., 2.5);
		\draw (5., 2.5) -- (5., 2.9);
		\draw (4.8, 2.5) node {$\alpha _1$};
		\draw (0, -0.25) node {$\lambda$};
	\end{tikzpicture}
\end{equation*}

\subsection{Proposition~\ref{proposition:e6:e7:string:type}. $(\Sigma_\Phi, \Sigma_{\lambda, \Phi}) \cong (E_6, E_{7})$.}
\begin{equation*}
	\begin{tikzpicture}[scale=1.4]
		\draw (0, 0) circle (0.1);
		\draw (1, 0) circle (0.1);
		\draw (2, 0) circle (0.1);
		\draw (3, 0) circle (0.1);
		\draw (4, 0) circle (0.1);
		\draw (5, 0) circle (0.1);
		\draw (3, 1) circle (0.1);
		\draw (4, 1) circle (0.1);
		\draw (5, 1) circle (0.1);
		\draw (4, 2) circle (0.1);
		\draw (5, 2) circle (0.1);
		\draw (6, 2) circle (0.1);
		\draw (4, 3) circle (0.1);
		\draw (5, 3) circle (0.1);
		\draw (6, 3) circle (0.1);
		\draw (7, 3) circle (0.1);
		\draw (8, 3) circle (0.1);
		\draw (4, 4) circle (0.1);
		\draw (5, 4) circle (0.1);
		\draw (6, 4) circle (0.1);
		\draw (7, 4) circle (0.1);
		\draw (8, 4) circle (0.1);
		\draw (7, 5) circle (0.1);
		\draw (8, 5) circle (0.1);
		\draw (9, 5) circle (0.1);
		\draw (10, 5) circle (0.1);
		\draw (11, 5) circle (0.1);
		\draw[->] (0.1, 0.) -- (0.5, 0.);
		\draw (0.5, 0.) -- (0.9, 0.);
		\draw (0.5, -0.175) node {$\alpha _6$};
		\draw[->] (1.1, 0.) -- (1.5, 0.);
		\draw (1.5, 0.) -- (1.9, 0.);
		\draw (1.5, -0.175) node {$\alpha _5$};
		\draw[->] (2.1, 0.) -- (2.5, 0.);
		\draw (2.5, 0.) -- (2.9, 0.);
		\draw (2.5, -0.175) node {$\alpha _4$};
		\draw[->] (3.1, 0.) -- (3.5, 0.);
		\draw (3.5, 0.) -- (3.9, 0.);
		\draw (3.5, -0.175) node {$\alpha _3$};
		\draw[->] (4.1, 0.) -- (4.5, 0.);
		\draw (4.5, 0.) -- (4.9, 0.);
		\draw (4.5, -0.175) node {$\alpha _1$};
		\draw[->] (3.1, 1.) -- (3.5, 1.);
		\draw (3.5, 1.) -- (3.9, 1.);
		\draw (3.5, 0.825) node {$\alpha _3$};
		\draw[->] (4.1, 1.) -- (4.5, 1.);
		\draw (4.5, 1.) -- (4.9, 1.);
		\draw (4.5, 0.825) node {$\alpha _1$};
		\draw[->] (4.1, 2.) -- (4.5, 2.);
		\draw (4.5, 2.) -- (4.9, 2.);
		\draw (4.5, 1.825) node {$\alpha _1$};
		\draw[->] (5.1, 2.) -- (5.5, 2.);
		\draw (5.5, 2.) -- (5.9, 2.);
		\draw (5.5, 1.825) node {$\alpha _3$};
		\draw[->] (4.1, 3.) -- (4.5, 3.);
		\draw (4.5, 3.) -- (4.9, 3.);
		\draw (4.5, 2.825) node {$\alpha _1$};
		\draw[->] (5.1, 3.) -- (5.5, 3.);
		\draw (5.5, 3.) -- (5.9, 3.);
		\draw (5.5, 2.825) node {$\alpha _3$};
		\draw[->] (6.1, 3.) -- (6.5, 3.);
		\draw (6.5, 3.) -- (6.9, 3.);
		\draw (6.5, 2.825) node {$\alpha _4$};
		\draw[->] (7.1, 3.) -- (7.5, 3.);
		\draw (7.5, 3.) -- (7.9, 3.);
		\draw (7.5, 2.825) node {$\alpha _2$};
		\draw[->] (4.1, 4.) -- (4.5, 4.);
		\draw (4.5, 4.) -- (4.9, 4.);
		\draw (4.5, 3.825) node {$\alpha _1$};
		\draw[->] (5.1, 4.) -- (5.5, 4.);
		\draw (5.5, 4.) -- (5.9, 4.);
		\draw (5.5, 3.825) node {$\alpha _3$};
		\draw[->] (6.1, 4.) -- (6.5, 4.);
		\draw (6.5, 4.) -- (6.9, 4.);
		\draw (6.5, 3.825) node {$\alpha _4$};
		\draw[->] (7.1, 4.) -- (7.5, 4.);
		\draw (7.5, 4.) -- (7.9, 4.);
		\draw (7.5, 3.825) node {$\alpha _2$};
		\draw[->] (7.1, 5.) -- (7.5, 5.);
		\draw (7.5, 5.) -- (7.9, 5.);
		\draw (7.5, 4.825) node {$\alpha _2$};
		\draw[->] (8.1, 5.) -- (8.5, 5.);
		\draw (8.5, 5.) -- (8.9, 5.);
		\draw (8.5, 4.825) node {$\alpha _4$};
		\draw[->] (9.1, 5.) -- (9.5, 5.);
		\draw (9.5, 5.) -- (9.9, 5.);
		\draw (9.5, 4.825) node {$\alpha _3$};
		\draw[->] (10.1, 5.) -- (10.5, 5.);
		\draw (10.5, 5.) -- (10.9, 5.);
		\draw (10.5, 4.825) node {$\alpha _1$};
		\draw[->] (3., 0.1) -- (3., 0.5);
		\draw (3., 0.5) -- (3., 0.9);
		\draw (2.8, 0.5) node {$\alpha _2$};
		\draw[->] (4., 0.1) -- (4., 0.5);
		\draw (4., 0.5) -- (4., 0.9);
		\draw (3.8, 0.5) node {$\alpha _2$};
		\draw[->] (5., 0.1) -- (5., 0.5);
		\draw (5., 0.5) -- (5., 0.9);
		\draw (4.8, 0.5) node {$\alpha _2$};
		\draw[->] (4., 1.1) -- (4., 1.5);
		\draw (4., 1.5) -- (4., 1.9);
		\draw (3.8, 1.5) node {$\alpha _4$};
		\draw[->] (5., 1.1) -- (5., 1.5);
		\draw (5., 1.5) -- (5., 1.9);
		\draw (4.8, 1.5) node {$\alpha _4$};
		\draw[->] (4., 2.1) -- (4., 2.5);
		\draw (4., 2.5) -- (4., 2.9);
		\draw (3.8, 2.5) node {$\alpha _5$};
		\draw[->] (5., 2.1) -- (5., 2.5);
		\draw (5., 2.5) -- (5., 2.9);
		\draw (4.8, 2.5) node {$\alpha _5$};
		\draw[->] (6., 2.1) -- (6., 2.5);
		\draw (6., 2.5) -- (6., 2.9);
		\draw (5.8, 2.5) node {$\alpha _5$};
		\draw[->] (4., 3.1) -- (4., 3.5);
		\draw (4., 3.5) -- (4., 3.9);
		\draw (3.8, 3.5) node {$\alpha _6$};
		\draw[->] (5., 3.1) -- (5., 3.5);
		\draw (5., 3.5) -- (5., 3.9);
		\draw (4.8, 3.5) node {$\alpha _6$};
		\draw[->] (6., 3.1) -- (6., 3.5);
		\draw (6., 3.5) -- (6., 3.9);
		\draw (5.8, 3.5) node {$\alpha _6$};
		\draw[->] (7., 3.1) -- (7., 3.5);
		\draw (7., 3.5) -- (7., 3.9);
		\draw (6.8, 3.5) node {$\alpha _6$};
		\draw[->] (8., 3.1) -- (8., 3.5);
		\draw (8., 3.5) -- (8., 3.9);
		\draw (7.8, 3.5) node {$\alpha _6$};
		\draw[->] (7., 4.1) -- (7., 4.5);
		\draw (7., 4.5) -- (7., 4.9);
		\draw (6.8, 4.5) node {$\alpha _5$};
		\draw[->] (8., 4.1) -- (8., 4.5);
		\draw (8., 4.5) -- (8., 4.9);
		\draw (7.8, 4.5) node {$\alpha _5$};
		\draw (0, -0.25) node {$\lambda$};
	\end{tikzpicture}
\end{equation*}

\subsection{Proposition~\ref{proposition:b:f:string:type}. $(\Sigma_\Phi, \Sigma_{\lambda, \Phi}) \cong (B_3, F_{4})$.}

\begin{equation*}
\begin{tikzpicture}[scale=1.4]
	\draw (0, 0) circle (0.1);
	\draw (1, 0) circle (0.1);
	\draw (2, 0) circle (0.1);
	\draw (3, 0) circle (0.1);
	\draw (2.5, 0.5) circle (0.1);
	\draw (2.5, -0.5) circle (0.1);
	\draw (4, 0) circle (0.1);
	\draw (4, 0) circle (0.1);
	\draw (5, 0) circle (0.1);
	\draw[->] (0.1, 0.) -- (0.5, 0.);
	\draw (0.5, 0.) -- (0.9, 0.);
	\draw (0.5, -0.175) node {$\alpha _1$};
	\draw[->] (1.1, 0.) -- (1.5, 0.);
	\draw (1.5, 0.) -- (1.9, 0.);
	\draw (1.5, -0.175) node {$\alpha _2$};
	\draw[->] (2.57071, -0.429289) -- (2.75, -0.25);
	\draw (2.75, -0.25) -- (2.92929, -0.0707107);
	\draw (2.9, -0.33) node {$\alpha _2$};
	\draw[->] (2.07071, 0.0707107) -- (2.25, 0.25);
	\draw (2.25, 0.25) -- (2.42929, 0.429289);
	\draw (2.1, 0.35) node {$\alpha _2$};
	\draw[->] (3.1, 0.) -- (3.5, 0.);
	\draw (3.5, 0.) -- (3.9, 0.);
	\draw (3.5, -0.175) node {$\alpha _2$};
	\draw[->] (4.1, 0.) -- (4.5, 0.);
	\draw (4.5, 0.) -- (4.9, 0.);
	\draw (4.5, -0.175) node {$\alpha _1$};
	\draw[->] (2.07071, -0.0707107) -- (2.25, -0.25);
	\draw (2.25, -0.25) -- (2.42929, -0.429289);
	\draw (2.1, -0.33) node {$\alpha _3$};
	\draw[->] (2.57071, 0.429289) -- (2.75, 0.25);
	\draw (2.75, 0.25) -- (2.92929, 0.0707107);
	\draw (2.9, 0.35) node {$\alpha _3$};
	\draw (0, -0.25) node {$\lambda$};
\end{tikzpicture}
\end{equation*}

\subsection{Proposition~\ref{proposition:c:f:string:type}. $(\Sigma_\Phi, \Sigma_{\lambda, \Phi}) \cong (C_3, F_{4})$.}

\begin{equation*}
\begin{tikzpicture}[scale=1.4]
	\draw (0, 0) circle (0.1);
	\draw (1, 0) circle (0.1);
	\draw (2, 0) circle (0.1);
	\draw (3, 0) circle (0.1);
	\draw (4, 0) circle (0.1);
	\draw (2, 1) circle (0.1);
	\draw (3, 1) circle (0.1);
	\draw (4, 1) circle (0.1);
	\draw (5, 1) circle (0.1);
	\draw (3, 2) circle (0.1);
	\draw (4, 2) circle (0.1);
	\draw (5, 2) circle (0.1);
	\draw (6, 2) circle (0.1);
	\draw (7, 2) circle (0.1);
	\draw[->] (0.1, 0.) -- (0.5, 0.);
	\draw (0.5, 0.) -- (0.9, 0.);
	\draw (0.5, -0.175) node {$\alpha _1$};
	\draw[->] (1.1, 0.) -- (1.5, 0.);
	\draw (1.5, 0.) -- (1.9, 0.);
	\draw (1.5, -0.175) node {$\alpha _2$};
	\draw[->] (2.1, 0.) -- (2.5, 0.);
	\draw (2.5, 0.) -- (2.9, 0.);
	\draw (2.5, -0.175) node {$\alpha _2$};
	\draw[->] (3.1, 0.) -- (3.5, 0.);
	\draw (3.5, 0.) -- (3.9, 0.);
	\draw (3.5, -0.175) node {$\alpha _1$};
	\draw[->] (3.1, 2.) -- (3.5, 2.);
	\draw (3.5, 2.) -- (3.9, 2.);
	\draw (3.5, 1.825) node {$\alpha _1$};
	\draw[->] (4.1, 2.) -- (4.5, 2.);
	\draw (4.5, 2.) -- (4.9, 2.);
	\draw (4.5, 1.825) node {$\alpha _2$};
	\draw[->] (5.1, 2.) -- (5.5, 2.);
	\draw (5.5, 2.) -- (5.9, 2.);
	\draw (5.5, 1.825) node {$\alpha _2$};
	\draw[->] (6.1, 2.) -- (6.5, 2.);
	\draw (6.5, 2.) -- (6.9, 2.);
	\draw (6.5, 1.825) node {$\alpha _1$};
	\draw[->] (2.1, 1.) -- (2.5, 1.);
	\draw (2.5, 1.) -- (2.9, 1.);
	\draw (2.5, 0.825) node {$\alpha _2$};
	\draw[->] (3.1, 1.) -- (3.5, 1.);
	\draw (3.5, 1.) -- (3.9, 1.);
	\draw (3.5, 0.825) node {$\alpha _1$};
	\draw[->] (4.1, 1.) -- (4.5, 1.);
	\draw (4.5, 1.) -- (4.9, 1.);
	\draw (4.5, 0.825) node {$\alpha _2$};
	\draw[->] (2., 0.1) -- (2., 0.5);
	\draw (2., 0.5) -- (2., 0.9);
	\draw (1.8, 0.5) node {$\alpha _3$};
	\draw[->] (3., 0.1) -- (3., 0.5);
	\draw (3., 0.5) -- (3., 0.9);
	\draw (2.8, 0.5) node {$\alpha _3$};
	\draw[->] (4., 0.1) -- (4., 0.5);
	\draw (4., 0.5) -- (4., 0.9);
	\draw (3.8, 0.5) node {$\alpha _3$};
	\draw[->] (3., 1.1) -- (3., 1.5);
	\draw (3., 1.5) -- (3., 1.9);
	\draw (2.8, 1.5) node {$\alpha _3$};
	\draw[->] (4., 1.1) -- (4., 1.5);
	\draw (4., 1.5) -- (4., 1.9);
	\draw (3.8, 1.5) node {$\alpha _3$};
	\draw[->] (5., 1.1) -- (5., 1.5);
	\draw (5., 1.5) -- (5., 1.9);
	\draw (4.8, 1.5) node {$\alpha _3$};
	\draw (0, -0.25) node {$\lambda$};
\end{tikzpicture}
\end{equation*}

\end{document}